%% file: xi_G.tex
\documentclass[a4paper, 11pt]{amsart}

\usepackage{amsmath,amsthm,amssymb}
\usepackage{color}

\theoremstyle{definition}
\newtheorem{dfn}{Definition}[section]
\newtheorem{rem}[dfn]{Remark}
\newtheorem{example}[dfn]{Example}

\theoremstyle{plain}
\newtheorem{thm}[dfn]{Theorem}
\newtheorem{prop}[dfn]{Proposition}
\newtheorem{lem}[dfn]{Lemma}
\newtheorem{cor}[dfn]{Corollary}
\newtheorem{quest}[dfn]{Question}

\makeatletter
    
    \@addtoreset{equation}{section}
  \makeatother

\begin{document}

\title[A homology valued invariant]{A homology valued invariant for trivalent fatgraph spines}
\author{Yusuke Kuno}
\address{Department of Mathematics, Tsuda College,
2-1-1 Tsuda-machi, Kodaira-shi Tokyo 187-8577, Japan}
\email{kunotti@tsuda.ac.jp}
\date{}

\thanks{The author is supported by JSPS KAKENHI (No.26800044).} 

\begin{abstract}
We introduce an invariant for trivalent fatgraph spines of a once bordered surface, which takes values in the first homology of the surface.
This invariant is the secondary object coming from two 1-cocycles on the dual fatgraph complex, one introduced by Morita and Penner in 2008, and the other by Penner, Turaev, and the author in 2013.
We present an explicit formula for this invariant and investigate its properties.
We also show that the mod 2 reduction of the invariant is the difference of naturally defined two spin structures on the surface.
\end{abstract}

\maketitle

\section{Introduction}

Let $\Sigma_{g,1}$ be a once bordered $C^{\infty}$-surface of genus $g>0$, and let $\mathcal{M}_{g,1}$ be the mapping class group of $\Sigma_{g,1}$ relative to the boundary.
It is known that the Teichm\"uller space $\mathcal{T}(\Sigma_{g,1})$ of $\Sigma_{g,1}$ has an $\mathcal{M}_{g,1}$-equivariant ideal simplicial decomposition \cite{P04}.
Taking its dual, one obtains a contractible CW complex $\widehat{\mathcal{G}}(\Sigma_{g,1})$ on which $\mathcal{M}_{g,1}$ acts freely and properly discontinuously.
This CW complex is called the dual fatgraph complex of $\Sigma_{g,1}$, since its cells are indexed by fatgraph spines of $\Sigma_{g,1}$, which are graphs embedded in the surface satisfying some conditions.
% and its face relation corresponds to the contraction of a non-loop edge of a fatgraph spine.
Each $0$-cell of $\widehat{\mathcal{G}}(\Sigma_{g,1})$ corresponds to a trivalent fatgraph spine, and by contracting non-loop edges we obtain higher dimensional cells.
In particular, each oriented $1$-cell of $\widehat{\mathcal{G}}(\Sigma_{g,1})$ corresponds to a flip (or a Whitehead move) between trivalent fatgraph spines of $\Sigma_{g,1}$.

This combinatorial structure of the Teichm\"uller space has a number of applications to the cohomology of the mapping class group and the moduli space of Riemann surfaces.
See e.g., \cite{Har1} \cite{Har2} \cite{HZ} \cite{P88} \cite{Kon}.
%For bordered surfaces, we mention \cite{Godin} where the homology of the mapping class group for low genus was carried out.

Recently, mainly motivated by the theory of the Johnson homomorphisms \cite{Joh80} \cite{Joh83} \cite{Mor93}, several authors considered $1$-cocycles on $\widehat{\mathcal{G}}(\Sigma_{g,1})$ with coefficients in various $\mathcal{M}_{g,1}$-modules.
In 2008, Morita and Penner \cite{MP} first gave such a $1$-cocycle $j\in Z^1(\widehat{\mathcal{G}}(\Sigma_{g,1});\Lambda^3 H)$, where $\Lambda^3 H$ is the third exterior product of the first homology group $H=H_1(\Sigma_{g,1};\mathbb{Z})$.
(In fact, they worked with a once punctured surface, but their construction works for $\Sigma_{g,1}$ as well.)
Being a $1$-cocycle on $\widehat{\mathcal{G}}(\Sigma_{g,1})$, the cocycle $j$ associates an element of $\Lambda^3 H$ to each flip.
Fixing a trivalent fatgraph spine of $\Sigma_{g,1}$, one obtains from $j$ a twisted $1$-cocycle on $\mathcal{M}_{g,1}$.
Morita and Penner proved that its cohomology class in $H^1(\mathcal{M}_{g,1};\Lambda^3 H)$ is six times the extended first Johnson homomorphism $\tilde{k}$ \cite{MorExt}.
Similar constructions are also considered by Bene, Kawazumi and Penner \cite{BKP} for the second and higher Johnson homomorphisms,
by Massuyeau \cite{Mas} for Morita's refinement \cite{Mor93} of the higher Johnson homomorphisms,
and by Kuno, Penner and Turaev \cite{KPT} for the Earle class $k\in H^1(\mathcal{M}_{g,1};H)$.

We emphasize that these cocycles on $\widehat{\mathcal{G}}(\Sigma_{g,1})$ are all explicit and simple.
In this way, the Johnson homomorphisms and related objects extend \emph{canonically} to the \emph{Ptolemy groupoid} \cite{BKP}, the combinatorial fundamental path groupoid of $\widehat{\mathcal{G}}(\Sigma_{g,1})$.

It is interesting that there are many ways of constructing cocycle representatives for the cohomology classes such as $\tilde{k}$ and $k$, and that each construction reflects its own viewpoint for studying the mapping class group.
It can happen that two cocycles constructed differently give the same cohomology class.
In such a case, it is quite natural to compare these cocycles and to expect a secondary object behind there.

In this paper, we compare the Morita-Penner cocycle $j$ and the cocycle $m\in Z^1(\widehat{\mathcal{G}}(\Sigma_{g,1});H)$ which is related to $k$ and considered in \cite{KPT}.
Contracting the coefficients by using the intersection pairing on $H$, one has a natural homomorphism
$$
C\colon Z^1(\widehat{\mathcal{G}}(\Sigma_{g,1});\Lambda^3 H)\to Z^1(\widehat{\mathcal{G}}(\Sigma_{g,1});H).
$$
Let $j^{\prime}=C\circ j$.
It turns out that there is an $\mathcal{M}_{g,1}$-equivariant $0$-cochain $\xi\in C^0(\widehat{\mathcal{G}}(\Sigma_{g,1});H)$ such that $2j^{\prime}-m=\delta \xi$ (Proposition \ref{prop:char_xi_G}).
The $0$-cochain $\xi$ associates an element $\xi_G\in H$ to each trivalent fatgraph spine $G \subset \Sigma_{g,1}$.

We will study the secondary object $\xi_G$ as an $H$-valued invariant for trivalent fatgraph spines $G\subset \Sigma_{g,1}$.
First of all, Theorem \ref{thm:xi_explicit} gives an explicit formula for $\xi_G$.
Based on this formula, we show in Theorem \ref{thm:non-trivial_b} that $\xi_G$ is non-trivial.
At the present moment, we do not have a full understanding of the topological meaning of the invariant $\xi_G$.
In Theorem \ref{thm:xi_spin}, we give a partial result in this direction by relating the mod $2$ reduction of $\xi_G$ to naturally defined two spin structures on $\Sigma_{g,1}$.

This paper is organized as follows.
In \S 2, we first review the dual fatgraph complex and in particular describe its $2$-skeleton.
Then we recall the $1$-cocycles $j$ from \cite{MP} and $m$ from \cite{KPT}.
Also, we correct an error in \cite{KPT} about the evaluation of $m$.
In \S 3, we show the existence and uniqueness of $\xi$, and then present an explicit formula for $\xi_G$ (Theorem \ref{thm:xi_explicit}).
In \S 4, we show a certain gluing formula for $\xi_G$, and then the behavior of $\xi_G$ under a special kind of flip.
The latter result makes it possible to define $\xi_G$ for a trivalent fatgraph spine $G$ of a \emph{punctured} surface.
In \S 5, we first prove the non-triviality of $\xi_G$. Then we discuss the primitivity of $\xi_G$ and show some partial results.
In \S 6, we first show that given a trivalent fatgraph spine $G\subset \Sigma_{g,1}$, one can associate two spin structures on $\Sigma_{g,1}$. Then we prove that their difference coincides with the mod $2$ reduction of $\xi_G$.
In Appendix A, we consider another spin structure coming from a naturally defined non-singular vector field on $\Sigma_{g,1}$.

The author would like to thank Gw\'ena\"el Massuyeau for communicating to him the construction of the vector field $\mathcal{X}_G$ in Appendix A, Robert Penner for helpful remarks on a description of spin structures on $\Sigma_{g,1}$ in \S 6, and
Vladimir Turaev and Nariya Kawazumi for valuable comments to a draft of this paper.

\section{Fatgraph complex and cocycles}
\label{sec:fatgraph}

We fix some notation about graphs.
By a graph we mean a finite CW complex of dimension one.
For a graph $G$, we denote by $V(G)$ the set of vertices of $G$, by $E(G)$ the set of edges of $G$, and by $E^{\rm ori}(G)$ the set of oriented edges of $G$.
For $v\in V(G)$, we denote by $E^{\rm ori}_v(G)$ the set of oriented edges toward $v$.
The number of elements of $E^{\rm ori}_v(G)$ is called the \emph{valency} of $v$.
For $e\in E^{\rm ori}(G)$, we denote by $\bar{e}\in E^{\rm ori}(G)$ the edge $e$ with reversed orientation.
A \emph{fatgraph} is a graph $G$ endowed with a cyclic ordering to $E^{\rm ori}_v(G)$ about each $v\in V(G)$.

Let $\Sigma_{g,1}$ be a compact connected oriented $C^{\infty}$-surface of genus $g>0$ with one boundary component.
We fix two distinct points $p$ and $q$ on the boundary $\partial \Sigma_{g,1}$.

\begin{dfn}
\label{dfn:fgs}
An embedding $\iota\colon G\hookrightarrow \Sigma_{g,1}$ of a fatgraph $G$ into $\Sigma_{g,1}$ is called a \emph{fatgraph spine} of $\Sigma_{g,1}$ if the following conditions are satisfied.
\begin{enumerate}
\item The map $\iota$ is a homotopy equivalence.
\item For any $v\in V(G)$, the cyclic ordering given to $E^{\rm ori}_v(G)$ is compatible with the orientation of $\Sigma_{g,1}$.
\item We have $\iota(G)\cap \partial \Sigma_{g,1}=\{ p\}$ and $\iota^{-1}(p)$ is a unique univalent vertex of $G$.
The other vertices have valencies greater than $2$.  
\end{enumerate}
\end{dfn}

A unique edge connected to $\iota^{-1}(p)$ is called the \emph{tail} of $G$.
We consider fatgraph spines up to isotopies relative to $\partial \Sigma_{g,1}$.
If there is no danger of confusion, we identify $G$ with $\iota(G)$, and write $G$ instead of $\iota\colon G\hookrightarrow \Sigma_{g,1}$.
We denote by $V^{\rm int}(G)$ the set of non-univalent vertices of $G$.
We say that $G$ is \emph{trivalent} if the valency of any non-univalent vertex of $G$ is 3.

Fatgraph spines appear naturally in the combinatorial description of the Teichm\"uller space of a punctured or bordered surface.
This was first shown for punctured surfaces by Harer-Mumford \cite{Har2} and Thurston from the holomorphic point of view based on a work by Strebel \cite{Stre}, and by Penner \cite{P87} and Bowditch-Epstein \cite{BE} from the point of view of hyperbolic geometry.

In this paper, we work mainly with the once bordered surface $\Sigma_{g,1}$.
For definiteness, let us define the Teichm\"uller space $\mathcal{T}(\Sigma_{g,1})$ as the space of Riemannian metric on $\Sigma_{g,1}$ of constant Gaussian curvature $-1$ with geodesic boundary, modulo pull-back of the metric by self-diffeomorphisms of $\Sigma_{g,1}$ fixing $q$ which are isotopic to the identity relative to $q$.
Let $\mathcal{M}_{g,1}$ be the mapping class group of $\Sigma_{g,1}$ relative to $\partial \Sigma_{g,1}$.
Namely, $\mathcal{M}_{g,1}$ is the group of self-diffeomorphisms of $\Sigma_{g,1}$ fixing the boundary $\partial \Sigma_{g,1}$ pointwise, modulo isotopies fixing $\partial \Sigma_{g,1}$ pointwise.
Note that $\mathcal{M}_{g,1}$ is identified with the group of connected components of the group of self-diffeomorphisms of $\Sigma_{g,1}$ fixing $q$.
Then pull-back of the metric induces an action of $\mathcal{M}_{g,1}$ on $\mathcal{T}(\Sigma_{g,1})$.
This action is known to be free and properly discontinuous.

\begin{thm}[Penner \cite{P04}]
\label{thm:decomposition}
There is an $\mathcal{M}_{g,1}$-equivariant ideal simplicial decomposition of $\mathcal{T}(\Sigma_{g,1})$ with the following properties.
\begin{itemize}
\item Each simplex corresponds to a fatgraph spine of $\Sigma_{g,1}$.
\item The face relation between simplices corresponds to the contraction of a non-loop edge of a fatgraph spine.
\end{itemize}
\end{thm}

Let $\widehat{\mathcal{G}}(\Sigma_{g,1})$ be the \emph{dual} of this ideal simplicial decomposition.
This is an honest CW complex of dimension $4g-2$.
We call $\widehat{\mathcal{G}}(\Sigma_{g,1})$ the \emph{dual fatgraph complex} of $\Sigma_{g,1}$.
Note that there is a natural cellular action of the mapping class group $\mathcal{M}_{g,1}$ on $\widehat{\mathcal{G}}(\Sigma_{g,1})$.
In fact, there is an $\mathcal{M}_{g,1}$-equivariant deformation retract of $\mathcal{T}(\Sigma_{g,1})$ onto $\widehat{\mathcal{G}}(\Sigma_{g,1})$.
See \cite{P-book}.

The $2$-skeleton of $\widehat{\mathcal{G}}(\Sigma_{g,1})$ is described as follows.

\begin{itemize}
\item Each $0$-cell corresponds to a trivalent fatgraph spine of $\Sigma_{g,1}$.
\item Each $1$-cell corresponds to a fatgraph spine $G$, where $G$ has a unique $4$-valent vertex and the other non-univalent vertices have valency $3$.
\item Each \emph{oriented} $1$-cell corresponds to a \emph{flip} (or a \emph{Whitehead move}) between trivalent fatgraph spines.
Here, if $e$ is a non-tail edge of a trivalent fatgraph spine, collapsing $e$ and expanding the resulting 4-valent vertex to the unique distinct direction, one produces another trivalent fatgraph spine.
We call this move a flip along $e$, and denote it by $W_e$. See Figure \ref{fig:flip}.
If $G^{\prime}$ is obtained from $G$ by a flip $W=W_e$, we write it as $G\overset{W}{\to} G^{\prime}$.
There is a natural bijection from $E(G)$ to $E(G^{\prime})$ which restricts to an obvious identification of $E(G)\setminus \{ e\}$ with $E(G^{\prime})\setminus \{ e^{\prime}\}$.
For this reason, we often use the same letter for edges of $G$ and $G^{\prime}$ corresponding to each other by this bijection.
\item Each $2$-cell corresponds to a fatgraph spine $G$, where either $G$ has a unique $5$-valent vertex  and the other non-univalent vertices have valency $3$, or $G$ has two $4$-valent vertices and the other non-univalent vertices have valency $3$.
\end{itemize}

\begin{figure}
\begin{center}
\caption{Flip along $e$}
\label{fig:flip}
\input{flip.tex}
\end{center}
\vspace{0.3cm}
\end{figure}

Let $G$ and $G^{\prime}$ be trivalent fatgraph spines.
Since $\widehat{\mathcal{G}}(\Sigma_{g,1})$ is connected, there is a finite sequence of flips
$$
G=G_0\overset{W_1}{\to} G_1 \overset{W_2}{\to} G_2 \to \cdots \overset{W_m}{\to} G_m=G^{\prime}
$$
from $G$ to $G^{\prime}$.
This sequence is not uniquely determined, but any two such sequences are related to each other by the following three types of relations among flips.

\begin{enumerate}
\item Involutivity relation: $W_{e^{\prime}}\circ W_e=1$ in the notation of Figure \ref{fig:flip}.
\item Commutativity relation: $W_{e_1}\circ W_{e_2}=W_{e_2}\circ W_{e_1}$ if $e_1$ and $e_2$ share no vertices.
\item Pentagon relation: $W_{f_4}\circ W_{g_3}\circ W_{f_2}\circ W_{g_1}\circ W_f=1$ in the notation of Figure \ref{fig:pentagon}.
\end{enumerate}

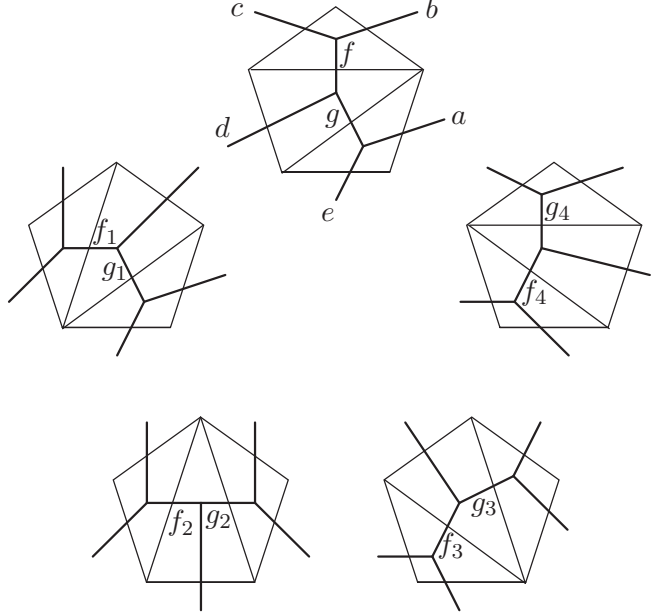
\begin{figure}
\begin{center}
\caption{Pentagon relation}
\label{fig:pentagon}
\input{pentagon.tex}
\end{center}
\vspace{0.3cm}
\end{figure}

Here, we read composition of flips from right to left.
The relations (2) and (3) come from the boundaries of 2-cells of $\widehat{\mathcal{G}}(\Sigma_{g,1})$.

There is a construction of twisted 1-cocycles on the mapping class group using the fatgraph complex appeared first in \cite{MP}.
Let $M$ be a (left) $\mathcal{M}_{g,1}$-module.
By definition, a cellular 1-cochain $c$ on $\widehat{\mathcal{G}}(\Sigma_{g,1})$ with values in $M$ is an assignment of an element of $M$ to each flip $W$ satisfying $c(W_{e^{\prime}})=-c(W_e)$ for any pair of flips $W_e$ and $W_{e^{\prime}}$ as in Figure \ref{fig:flip}.
Such a $c$ is a 1-cocycle if it satisfies the commutative equation
$$
c(W_{e_1})+c(W_{e_2})=c(W_{e_2})+c(W_{e_1}),
$$
where $e_1$ and $e_2$ are any edges on a trivalent fatgraph spine sharing no vertices, and the pentagon equation
$$
c(W_{f_4})+c(W_{g_3})+c(W_{f_2})+c(W_{g_1})+c(W_f)=0
$$
for any $5$-tuple of flips as in Figure \ref{fig:pentagon}.

Now we assume that $c$ is a $1$-cocycle and is $\mathcal{M}_{g,1}$-equivariant in the sense that $\varphi\cdot c(W)=c(\varphi W)$ for any flip $W$ and $\varphi \in \mathcal{M}_{g,1}$.
Fix a trivalent fatgraph spine $G$.
For $\varphi\in \mathcal{M}_{g,1}$, taking a sequence of flips
$$
G=G_0\overset{W_1}{\to} G_1 \overset{W_2}{\to} G_2 \to \cdots \overset{W_m}{\to} G_m=\varphi(G)
$$
from $G$ to $\varphi(G)$,
we set
$$
c_G(\varphi):=\sum_{i=1}^m c(W_i) \in M.
$$
Since $c$ is a $1$-cocycle, this value does not depend on the choice of the sequence.
The map $c_G\colon \mathcal{M}_{g,1}\to M$ is a twisted $1$-cocycle.
In fact, for $\varphi, \psi \in \mathcal{M}_{g,1}$, take a sequence of flips from $G$ to $\varphi(G)$, and one from $G$ to $\psi(G)$.
Then the first sequence followed by application of $\varphi$ to the second is a sequence of flips from $G$ to $\varphi \psi(G)$. 
Since $c$ is $\mathcal{M}_{g,1}$-equivariant, we obtain the cocycle condition
$$
c_G(\varphi \psi)=c_G(\varphi)+\varphi \cdot c_G(\psi).
$$
It is easy to see that the cohomology class $[c_G]\in H^1(\mathcal{M}_{g,1}; M)$ does not depend on the choice of $G$.

Here we record an elementary fact which will be used later.

\begin{lem}
\label{lem:c_elementary}
Let $M$ be an $\mathcal{M}_{g,1}$-module and suppose that $c$ is an $\mathcal{M}_{g,1}$-equivariant cellular $1$-cochain on $\widehat{\mathcal{G}}(\Sigma_{g,1})$ with values in $M$.
Then for any trivalent fatgraph spine $G$ and any $\varphi,\psi\in \mathcal{M}_{g,1}$, we have
$$
c_G(\psi)+c_{\psi(G)}(\varphi)=c_G(\varphi)+\varphi \cdot c_G(\psi).
$$
\end{lem}

\begin{proof}
Consider a sequence of flips from $G$ to $\psi(G)$ and one from $\psi(G)$ to $\varphi \psi(G)$.
The composition of these sequences is a sequence from $G$ to $\varphi \psi(G)$, and thus we obtain $c_G(\varphi \psi)=c_G(\psi)+c_{\psi(G)}(\varphi)$.
On the other hand, by the cocycle condition for $c_G$, we have $c_G(\varphi \psi)=c_G(\varphi)+\varphi \cdot c_G(\psi)$.
\end{proof}

We denote by $H=H_1(\Sigma_{g,1};\mathbb{Z})$ the first integral homology group of $\Sigma_{g,1}$.
Before giving examples of $\mathcal{M}_{g,1}$-equivariant cellular $1$-cochains on $\widehat{\mathcal{G}}(\Sigma_{g,1})$,
we recall from \cite{MP} homology markings for edges of fatgraph spines.
Let $G$ be a (not necessarily trivalent) fatgraph spine of $\Sigma_{g,1}$.
For $e\in E^{\rm ori}(G)$, there is an oriented simple loop $\hat{e}$ on $\Sigma_{g,1}$ satisfying the following two conditions.
\begin{itemize}
\item The loop $\hat{e}$ intersects $G$ once transversely at the middle point of $e$,
\item The ordered pair of the velocity vectors of $\hat{e}$ and $e$ at their intersection is compatible with the orientation of $\Sigma_{g,1}$.
\end{itemize}
Since the surface obtained from $\Sigma_{g,1}$ by cutting along $G$ is a disk, the homotopy class of such an $\hat{e}$ is unique.
We define $\mu(e)\in H$ to be the homology class of $\hat{e}$ and call it the \emph{homology marking} of $e$.
The map $\mu\colon E^{\rm ori}(G)\to H$ has the following properties.
\begin{enumerate}
\item For any $e\in E^{\rm ori}(G)$, we have $\mu(\bar{e})=-\mu(e)$.
\item The set $\{ \mu(e)\}_{e\in E^{\rm ori}(G)}$ generates $H$.
\item For any $v\in V(G)$, we have
$$
\sum_{e\in E^{\rm ori}_v(G)} \mu(e)=0.
$$
\end{enumerate}
For example, in the notation of the left part of Figure \ref{fig:flip}, where we orient edges $a,b,c,d$ as indicated, we have $\mu(a)+\mu(b)+\mu(c)+\mu(d)=0$.

In what follows, we consider $\mathcal{M}_{g,1}$-modules such as $H$ and its third exterior product $\Lambda^3 H$.
There is a twisted cohomology class $\tilde{k}\in H^1(\mathcal{M}_{g,1};\frac{1}{2}\Lambda^3 H)$ called the \emph{extended first Johnson homomorphism} \cite{MorExt}.
This cohomology class has a fundamental importance in the study of the cohomology of the mapping class group.
See \cite{KawMor}.

\begin{thm}[Morita-Penner \cite{MP}]
\label{thm:MP}
Keep the notation in Figure \ref{fig:flip}.
For the flip $W_e$, set
$$
j(W_e)=\mu(a)\wedge \mu(b)\wedge \mu(c) \in \Lambda^3 H.
$$
Then $j$ is an $\mathcal{M}_{g,1}$-equivariant $1$-cocycle on $\widehat{\mathcal{G}}(\Sigma_{g,1})$, and we have
$$
[j_G]=6\tilde{k}.
$$
\end{thm}

Using the intersection pairing $(\ \cdot \ )$ on the homology, we define an ${\rm Sp}(H)$-equivariant map
$$
C\colon \Lambda^3 H\to H, \quad x\wedge y\wedge z \mapsto (x\cdot y)z+(y\cdot z)x+(z\cdot x)y
$$
called the contraction.
Morita \cite{MorFam1} showed that if $g\ge 2$, the twisted cohomology group $H^1(\mathcal{M}_{g,1};H)$ is infinite cyclic.
As is remarked in \cite{MorExt}, the element $k:=C(2\tilde{k})$ is a generator of this cohomology group.
Since Earle \cite{Earle} first gave a cocycle representative for $k$, we call $k$ the \emph{Earle class}. See \cite{KawSurv}.

\begin{thm}[Kuno-Penner-Turaev \cite{KPT}]
\label{thm:KPT}
Keep the notation in Figure \ref{fig:flip}.
For the flip $W_e$, set
$$
m(W_e)=\mu(a)+\mu(c)\in H.
$$
Then $m$ is an $\mathcal{M}_{g,1}$-equivariant $1$-cocycle on $\widehat{\mathcal{G}}(\Sigma_{g,1})$, and we have
$$
[m_G]=6k.
$$
\end{thm}

Here we correct an error in \cite{KPT}.
Let $\varphi_{BP}=\varphi$ be the torus BP map in \cite{KPT} Fig.3, which was first considered in \cite{MP}.
In \cite{KPT} Lemma 1, it was asserted that $m(\varphi_{BP})=4a$, but this is not true.
More precisely, in the proof of the lemma, we computed the contribution of the second Dehn twist (5 flips) as $-4a$, but this should be corrected as $4a$.

\begin{lem}[correction of \cite{KPT} Lemma 1]
\label{lem:correction}
Let $\varphi_{BP}$ be the torus BP map as above.
Then we have
$$
m(\varphi_{BP})=12\mu(a).
$$
\end{lem}

In \cite{KPT} Theorem 6, it is asserted that $[m_G]=-2k$, but this should be corrected as in Theorem \ref{thm:KPT} above.

\section{A secondary invariant}
We consider the cocycle $j^{\prime}=C\circ j$.
For the flip $W_e$ in the notation of Figure \ref{fig:flip}, we have
$$
j^{\prime}(W_e)=(a\cdot b)\mu(c)+(b\cdot c)\mu(a)+(c\cdot a)\mu(b)\in H.
$$
Here and throughout the paper, to simplify the notation we write e.g., $(a\cdot b)$ instead of $(\mu(a)\cdot \mu(b))$.
By Theorems \ref{thm:MP} and \ref{thm:KPT}, for any trivalent fatgraph spine $G$, we have $2[j^{\prime}_G]=[m_G]=6k$.
Therefore, there exists an element $\xi_G\in H$ such that $2j^{\prime}_G-m_G=\delta \xi_G$.
Here the symbol $\delta$ in the right hand side means the coboundary map in the standard cochain complex of $\mathcal{M}_{g,1}$ with coefficients in $H$.
Explicitly, we have $(\delta \xi_G)(\varphi)=\varphi\cdot \xi_G-\xi_G$ for any $\varphi\in \mathcal{M}_{g,1}$.
Such a $\xi_G$ is unique since only $0$ is $\mathcal{M}_{g,1}$-invariant in $H$.
We regard the collection $\xi=\{ \xi_G \}_G$ as a cellular $0$-cochain of $\widehat{\mathcal{G}}(\Sigma_{g,1})$ with coefficients in $H$.

\begin{prop}
\label{prop:char_xi_G}
\begin{enumerate}
\item
The $0$-cochain $\xi$ is $\mathcal{M}_{g,1}$-equivariant in the sense that $\xi_{\psi(G)}=\psi\cdot \xi_G$
for any $\psi\in \mathcal{M}_{g,1}$ and any trivalent fatgraph spine $G$.
\item
We have $2j^{\prime}-m=\delta \xi$.
Namely, for any flip $G\overset{W}{\to}G^{\prime}$, we have $\xi_{G^{\prime}}-\xi_G=2j^{\prime}(W)-m(W)$.
\end{enumerate}
Moreover, these two properties characterize $\xi$.
\end{prop}

\begin{proof}
(1) For simplicity we write $s=2j^{\prime}-m$.
Take $\varphi\in \mathcal{M}_{g,1}$.
Using $s_G(\varphi)=\delta \xi_G (\varphi)=\varphi \cdot \xi_G-\xi_G$, etc., we compute from Lemma \ref{lem:c_elementary} that
\begin{align*}
s_{\psi(G)}(\varphi) &=s_G(\varphi)+\varphi \cdot s_G(\psi)-s_G(\psi) \\
&= \varphi \cdot \xi_G-\xi_G +\varphi \cdot (\psi \cdot \xi_G -\xi_G)
-(\psi \cdot \xi_G -\xi_G) \\
&=\varphi \cdot (\psi \cdot \xi_G)-\psi \cdot \xi_G \\
&=\delta (\psi \cdot \xi_G) (\varphi).
\end{align*}
This proves $s_{\psi(G)}=\delta (\psi \cdot \xi_G)$.
By the uniqueness of $\xi_{\psi(G)}$, it follows that $\xi_{\psi(G)}=\psi\cdot \xi_G$.

(2) can be proved analogously, and so we omit the detail.

Finally, suppose that $\xi^0$ is an $\mathcal{M}_{g,1}$-equivariant 0-cochain satisfying $2j^{\prime}-m=\delta \xi^0$.
Then $\xi-\xi^0$ is an $\mathcal{M}_{g,1}$-equivariant $0$-cocycle.
This shows that $\eta:=\xi(G)-\xi^0(G)\in H$ is independent of $G$ and $\varphi\cdot \eta=\eta$ for any $\varphi\in \mathcal{M}_{g,1}$.
Therefore $\eta$ must be zero and $\xi^0=\xi$.
\end{proof}

Let $G$ be a trivalent fatgraph spine of $\Sigma_{g,1}$.
We present an explicit formula for $\xi_G$.
To begin with, we introduce a total ordering for $E^{\rm ori}(G)$.
Note that if we cut $\Sigma_{g,1}$ along $G$, we obtain an oriented closed disk $D_G$.

\begin{dfn}
\label{dfn:total_order}
\begin{enumerate}
\item For $e,e^{\prime}\in E^{\rm ori}(G)$, we say $e\prec e^{\prime}$ if the edge $e$ occurs first when we go along the boundary of $D_G$ from $p$ by {\it clockwise} manner.
\item Let $e\in E^{\rm ori}(G)$.
We say that $e$ has the \emph{preferred orientation} (or $e$ is \emph{preferably oriented}) if $e\prec \bar{e}$.
\end{enumerate}
\end{dfn}

Note that any unoriented edge of $G$ has the unique preferred orientation.

Let $v\in V^{\rm int}(G)$.
We name three elements of $E^{\rm ori}_v(G)$ as $e_1,e_2$, and $e_3$, so that
\begin{enumerate}
\item $e_1\prec e_2$ and $e_1\prec e_3$, and
\item the edge $e_2$ is next to $e_1$ in the cyclic ordering given to $E^{\rm ori}_v(G)$.
\end{enumerate}
There are two possibilities for the ordering of $e_i$ and its inverse $\bar{e}_i$, namely,
$$
e_1\prec \bar{e}_2\prec e_2\prec \bar{e}_3\prec e_3\prec \bar{e}_1,
$$
or
$$
e_1\prec \bar{e}_2\prec e_3\prec \bar{e}_1\prec e_2\prec \bar{e}_3.
$$
The vertex $v$ is called {\it of type 1} if the former case happens, and is called {\it of type 2} otherwise.
Figure \ref{fig:2types} is an illustration of the situation.

\begin{figure}
\begin{center}
\caption{Two types for vertices}
\label{fig:2types}
\input{2types.tex}
\end{center}
\vspace{0.3cm}
\end{figure}

We can count the number of vertices of type 1 and that of type 2.

\begin{prop}
\label{prop:type1_2}
For any trivalent fatgraph spine $G$ of $\Sigma_{g,1}$, the number of trivalent vertices of type 1 is $2g-1$, and that of type 2 is $2g$.
\end{prop}

\begin{proof}
Let $V_1$ and $V_2$ be the numbers of trivalent vertices of type 1 and that of type 2, respectively.
Since the number of trivalent vertices of $G$ is $4g-1$, we have $V_1+V_2=4g-1$.
We observe that if a trivalent vertex $v$ is of type $i$ ($i=1,2$), the number of preferably oriented edges toward $v$ is $i$.
Thus $V_1+2V_2$ is equal to the number of edges of $G$, i.e., $6g-1$.
Hence we obtain $V_1=2g-1$ and $V_2=2g$.
\end{proof}

We set
$$
\begin{cases}
e_v=e_2 {\rm \ and\ } f_v=e_3 & {\rm \ if\ } v {\rm \ is\ of\ type\ 1},\\
e_v=e_1 {\rm \ and\ } f_v=e_3 & {\rm \ if\ } v {\rm \ is\ of\ type\ 2}.
\end{cases}
$$

\begin{thm}
\label{thm:xi_explicit}
We have
$$
\xi_G=\sum_v (\mu(e_v)-\mu(f_v)),
$$
where the sum is taken over all trivalent vertices of $G$.
\end{thm}

\begin{proof}
We set $\xi_G^0=\sum_v (\mu(e_v)-\mu(f_v))$ and consider the collection $\xi^0=\{ \xi_G^0 \}_G$.
Clearly, $\xi^0$ is $\mathcal{M}_{g,1}$-equivariant.
By Proposition \ref{prop:char_xi_G}, it is sufficient to prove $2j^{\prime}-m=\delta \xi^0$.

\begin{figure}
\begin{center}
\caption{The case where $a\prec c\prec b\prec d$}
\label{fig:acbd}
\input{proof.tex}
\end{center}
\vspace{0.3cm}
\end{figure}
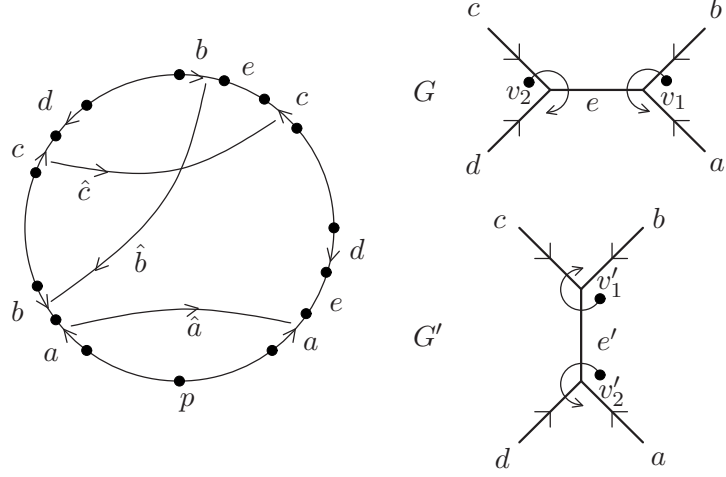

Take the notation as in Figure \ref{fig:flip}.
For example, assume that $a\prec c\prec b\prec d$.
For simplicity, we write $e$ instead of $\mu(e)$ for $e\in E^{\rm ori}(G)$.
Then we can see from the left part of Figure \ref{fig:acbd} that $(a\cdot b)=(c\cdot a)=0$ and $(b\cdot c)=1$, and so $j^{\prime}(W_e)=a$.
Thus $2j^{\prime}(W_e)-m(W_e)=2a-(a+c)=a-c$.
On the other hand, we can compute from the right part of Figure \ref{fig:acbd} that
\begin{align*}
\xi^0_{G^{\prime}}-\xi^0_G &=(e_{v^{\prime}_1}-f_{v^{\prime}_1})+(e_{v^{\prime}_2}-f_{v^{\prime}_2})-(e_{v_1}-f_{v_1})-(e_{v_2}-f_{v_2}) \\
&=(a+d-c)+(b+c-d)-(b-(c+d))-(c-(a+b)) \\
&=2a+b+d=2a+b+(-a-b-c)=a-c.
\end{align*}
We can compute similarly for other cases as well, and we obtain $2j^{\prime}(W_e)-m(W_e)=\xi^0_{G^{\prime}}-\xi^0_G$.
(There are essentially $6$ cases to consider; in each case in Figure \ref{fig:6cases}, there are two possibilities depending on whether $G$ corresponds to the left pictures or to the right pictures.
The latter case reduces to the former case by changing the role of $G$ and $G'$.)
Hence $2j^{\prime}-m=\delta \xi^0$, as required.
\end{proof}

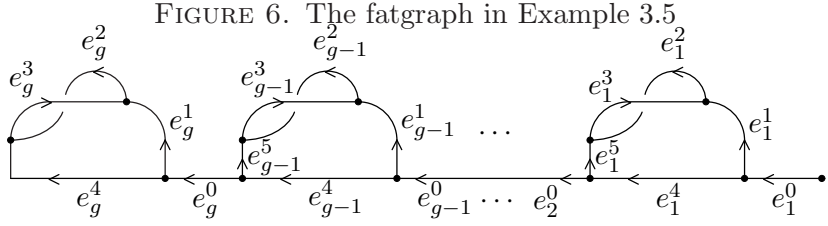
\begin{figure}
\begin{center}
\caption{Situations near $e$}
\label{fig:6cases}
\input{6cases.tex}
\end{center}
\vspace{0.3cm}
\end{figure}
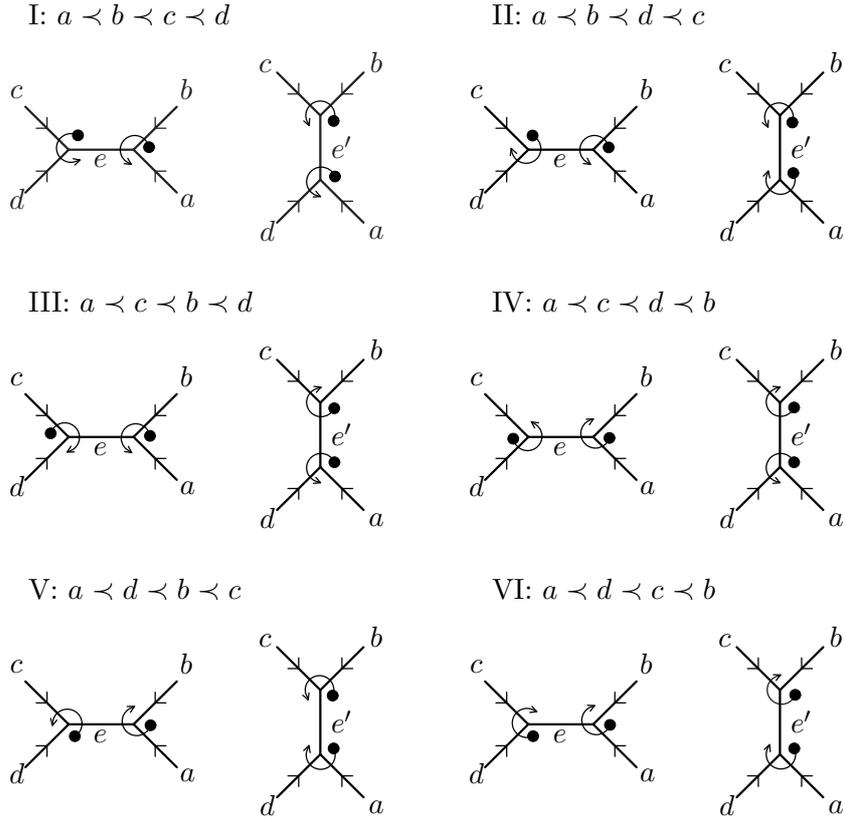

\begin{example}
\label{ex:xi_G}
Let $G$ be the fatgraph as shown in Figure \ref{fig:example}.
We name edges as in the figure and give them the preferred orientation.
For $1\le i\le g$ and $1\le j\le 3$, let $v_i^j\in V^{\rm int}(G)$ be the start point of $e_i^j$.
For $1\le i\le g-1$, let $v_i^4\in V^{\rm int}(G)$ be the end point of $e_i^4$.

Since $v_i^1$ is of type 2, its contribution is $\mu(\bar{e}_i^1)-\mu(\bar{e}_i^4)=\mu(e_i^4)-\mu(e_i^1)$.
Since $v_i^2$ is of type 1, its contribution is $\mu(e_i^1)-\mu(e_i^3)$.
Since $v_i^3$ is of type 1, its contribution is $\mu(e_i^2)-\mu(e_i^5)$.
Here we understand that $e_g^5=e_g^4$.
Since $v_i^4$ is of type 2, its contribution is $\mu(\bar{e}_i^5)-\mu(\bar{e}_{i+1}^0)=\mu(e_{i+1}^0)-\mu(e_i^5)$.

Moreover, we have $\mu(e_i^0)=0$, $\mu(e_i^1)+\mu(e_i^3)=\mu(e_i^2)$, and $\mu(e_i^4)=\mu(e_i^5)=-\mu(e_i^1)$.
Using these relations, we obtain
$$
\xi_G=\mu(e_g^1)+\sum_{i=1}^{g-1} 2\mu(e_i^1).
$$

\begin{figure}
\begin{center}
\caption{The fatgraph in Example \ref{ex:xi_G}}
\label{fig:example}
\input{example1.tex}
\end{center}
\vspace{0.3cm}
\end{figure}

\end{example}

\section{Elementary properties}

In this section, we record two elementary properties of $\xi_G$.

We first show a certain gluing formula.
Let $g$ and $g^{\prime}$ be positive integers, and suppose that we have two trivalent fatgraph spines
$\iota\colon G\hookrightarrow \Sigma_{g,1}$ and $\iota^{\prime}\colon G^{\prime}\hookrightarrow \Sigma_{g^{\prime},1}$.
Fix $e\in E^{\rm ori}(G)$.
Plugging the tail of $G^{\prime}$ in the right side of $e$, one produces a new fatgraph spine of $\Sigma_{g+g^{\prime},1}$. A precise construction is as follows. Let $v_e$ be the middle point of $e$.
\begin{enumerate}
\item Take a small closed disk $D_e$ in $\Sigma_{g,1}$ such that  ${\rm Int}(D_e)\cap G=\emptyset$, the boundary $\partial D_e$ intersects $G$ once at $v_e$, and the center of $D_e$ is on the right side of $e$ with respect to the orientation of $e$.
\item Glue $\Sigma_{g,1}\setminus {\rm Int}(D_e)$ with $\Sigma_{g^{\prime},1}$ along the boundaries $\partial D_e$ and $\partial \Sigma_{g^{\prime},1}$ so that the univalent vertex of $G^{\prime}$ is identified with $v_e$.
\item Let $G^{\prime \prime}$ be the union of the images of $G$ and $G^{\prime}$ in the result of gluing.
\end{enumerate}
The glued surface is diffeomorphic to $\Sigma_{g+g^{\prime},1}$.
We consider $G^{\prime \prime}$ as a trivalent fatgraph spine of $\Sigma_{g+g^{\prime},1}$ by dividing $e$ into two edges sharing the newly created trivalent vertex $v_e$.
These two edges receive their orientation from $e$.
We name them as $e_1,e_2\in E^{\rm ori}(G^{\prime \prime})$ so that $v_e$ is the end point of $e_1$.
The edges $e_1$ and $e_2$ have the same homology marking as $e$.

\begin{figure}
\begin{center}
\caption{Gluing}
\label{fig:gluing}
\input{gluing.tex}
\end{center}
\vspace{0.3cm}
\end{figure}

A schematic figure of this construction is Figure \ref{fig:gluing}.
We call $G^{\prime \prime}$ the \emph{gluing} of $G$ and $G^{\prime}$ at $e$.
Note that the inclusions $\Sigma_{g,1}\setminus {\rm Int}(D_e)\hookrightarrow \Sigma_{g+g^{\prime},1}$ and $\Sigma_{g^{\prime},1}\hookrightarrow \Sigma_{g+g^{\prime},1}$ induce a direct sum decomposition
\begin{equation}
\label{eq:direct}
H_1(\Sigma_{g+g^{\prime},1};\mathbb{Z})\cong H_1(\Sigma_{g,1};\mathbb{Z})\oplus H_1(\Sigma_{g^{\prime},1};\mathbb{Z}).
\end{equation}

\begin{prop}[gluing formula]
Let $G^{\prime \prime}$ be the gluing of $G$ and $G^{\prime}$ at $e$, as above.
Then we have
$$
\xi_{G^{\prime \prime}}=\xi_G+\mu(e)+\xi_{G^{\prime}}.
$$
\end{prop}

\begin{proof}
We have a natural identification $V^{\rm int}(G^{\prime \prime})\cong V^{\rm int}(G)\sqcup \{ v_e\} \sqcup V^{\rm int}(G^{\prime})$.
Observe that this identification respects the type of vertices.
With the direct sum decomposition (\ref{eq:direct}) in mind, we see that $V^{\rm int}(G)$ and $V^{\rm int}(G^{\prime})$ contribute to $\xi_{G^{\prime \prime}}$ as $\xi_G$ and $\xi_{G^{\prime}}$, respectively.

We compute the contribution from $v_e$.
Let $t^{\prime}\in E^{\rm ori}_{v_e}(G^{\prime \prime})$ be an edge coming from the tail of $G^{\prime}$.
The homology marking of $t^{\prime}$ is trivial.
If $e$ has the preferred orientation, we see that the contribution is $\mu(t^{\prime})-\mu(\bar{e}_2)=\mu(e)$.
Otherwise, the contribution is $\mu(e_1)-\mu(t^{\prime})=\mu(e)$.
This completes the proof.
\end{proof}

We next show a formula describing how $\xi_G$ changes under a special kind of flip.
Let $G$ be a trivalent fatgraph spine of $\Sigma_{g,1}$.
We use the following notation.
\begin{itemize}
\item We denote by $t$ the tail of $G$, and give it the preferred orientation.
\item $e_1\in E^{\rm ori}(G)$ is the oriented edge next to $t$ in the total ordering given to $E^{\rm ori}(G)$.
\item $v_1$ and $v_2$ are the start and end points of $e_1$, respectively.
\item $b,c\in E^{\rm ori}_{v_2}(G)$ are the edges such that $e_1$, $b$, and $c$ are in this order in the cyclic ordering given to $E^{\rm ori}_{v_2}(G)$.
\end{itemize}

The situation is illustrated in Figure \ref{fig:tail_slide}. 
We call the flip along (the unoriented edge underlying) $e_1$ the \emph{tail slide} to $G$.

\begin{prop}[tail slide formula]
\label{prop:tail_slide}
Let $G^{\prime}$ be the result of the tail slide to $G$. Then we have
$$
\xi_{G^{\prime}}=\xi_G+\mu(c).
$$
\end{prop}

\begin{proof}
We work with Figure \ref{fig:tail_slide}.
Suppose $b\prec c$ in $E^{\rm ori}(G)$.
For simplicity, we write $e$ instead of $\mu(e)$ for $e\in E^{\rm ori}(G)$.
Then we compute
\begin{align*}
\xi_{G^{\prime}}-\xi_G &=(e_{v^{\prime}_1}-f_{v^{\prime}_1})+(e_{v^{\prime}_2}-f_{v^{\prime}_2})-(e_{v_1}-f_{v_1})-(e_{v_2}-f_{v_2}) \\
&=(b-(-b))+(c-(-b-c))-(b+c-(-b-c))-(b-c) \\
&=c.
\end{align*}
The case where $c\prec b$ can be computed similarly.
%If $c<b$, we compute
%\begin{align*}
%\xi_{G^{\prime}}-\xi_G &=(e_{v^{\prime}_1}-f_{v^{\prime}_1})+(e_{v^{\prime}_2}-f_{v^{\prime}_2})-(e_{v_1}-f_{v_1})-(e_{v_2}-f_{v_2}) \\
%&=(b-(-b))+(c-b)-((-b-c)-c)-(b+c-(-b-c)) \\
%&=c.
%\end{align*}
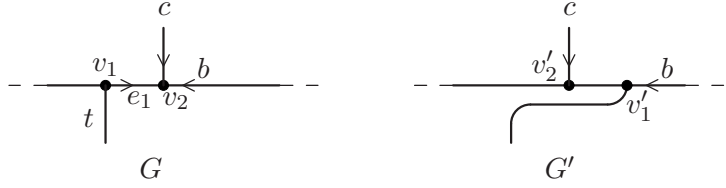
\begin{figure}
\begin{center}
\caption{Tail slide}
\label{fig:tail_slide}
\input{tail_slide.tex}
\end{center}
\vspace{0.3cm}
\end{figure}
\end{proof}

As an application of Proposition \ref{prop:tail_slide}, we can extend the definition of our invariant to trivalent fatgraph spines of a \emph{once punctured} surface.
Let $\Sigma_g^1$ be a surface obtained from $\Sigma_{g,1}$ by gluing a once punctured disk along the boundaries.
We regard $\Sigma_{g,1}$ as a subset of $\Sigma_g^1$.
By definition, a fatgraph spine of $\Sigma_g^1$ is an embedding $\iota\colon G\hookrightarrow \Sigma_g^1$ of a fatgraph $G$ into $\Sigma_g^1$ satisfying the first two conditions in Definition \ref{dfn:fgs} (with $\Sigma_{g,1}$ replaced by $\Sigma_g^1$), and the condition that all vertices have valency greater than $2$.

Let $G$ be a trivalent fatgraph spine of $\Sigma_g^1$.
By a suitable isotopy, we arrange that $G\subset \Sigma_{g,1}$.
Let $e\in E^{\rm ori}(G)$.
Take a simple arc $\ell$ on $\Sigma_{g,1}$ starting from $p$, reaching $v_e$ from the right, and disjoint from $G\setminus \{v_e\}$.
We say that such an arc $\ell$ is \emph{admissible} for $e$.
Regarding $v_e$ as a newly created trivalent vertex, we can consider the union $\widetilde{G}(e,\ell)=G\cup \ell$ as a trivalent fatgraph spine of $\Sigma_{g,1}$.
The arc $\ell$ becomes the tail of $\widetilde{G}(e,\ell)$.

\begin{cor}
\label{cor:punctured}
Keep the notation as above.
Then the element
$$
\xi_{\widetilde{G}(e,\ell)}-\mu(e)
$$
does not depend on the choice of $e$ and $\ell$.
In particular, for a trivalent fatgraph spine $G\subset \Sigma_g^1$, we can define $\xi_G\in H=H_1(\Sigma_{g,1};\mathbb{Z}) \cong H_1(\Sigma_g^1;\mathbb{Z})$ as
$$
\xi_G:=\xi_{\widetilde{G}(e,\ell)}-\mu(e).
$$
\end{cor}

\begin{proof}
Let $\ell^0$ be another admissible arc for $e$. Then $\ell^0$ is isotopic to the concatenation of some power of a simple based loop parallel to $\partial \Sigma_{g,1}$ and $\ell$.
This implies that $\widetilde{G}(e,\ell^0)$ is obtained from $\widetilde{G}(e,\ell)$ by application of some power of the Dehn twist along $\partial \Sigma_{g,1}$.
Since the Dehn twist along $\partial \Sigma_{g,1}$ acts on $H$ trivially, we have $\xi_{\widetilde{G}(e,\ell)}=\xi_{\widetilde{G}(e,\ell^0)}$.
Hence $\xi_{\widetilde{G}(e,\ell)}-\mu(e)$ does not depend on the choice of $\ell$.

Now, we can give a \emph{cyclic} ordering to the set $E^{\rm ori}(G)$ by a way similar to the case where $G\subset \Sigma_{g,1}$ as in Definition \ref{dfn:total_order}.
Suppose that $e,e^{\prime}\in E^{\rm ori}(G)$ are consecutive in this cyclic ordering.
Fix an admissible arc $\ell$ for $e$.
Let $v_0$ be the vertex of $G$ shared by $e$ and $e^{\prime}$, and let $c\in E^{\rm ori}_{v_0}(G)$ be an edge other than $e$ and $\bar{e}^{\prime}$.
We denote by $e_0$ an unoriented edge of $\widetilde{G}(e,\ell)$ with end points $v_e$ and $v_0$.

Let $\widetilde{G}^{\prime}$ be the result of flip along $e_0$.
Then $\widetilde{G}^{\prime}$ can be identified with $\widetilde{G}(e^{\prime},\ell^{\prime})$, where $\ell^{\prime}$ corresponds to the tail of $\widetilde{G}^{\prime}$.
By Proposition \ref{prop:tail_slide}, we have $\xi_{\widetilde{G}(e^{\prime},\ell^{\prime})}=\xi_{\widetilde{G}(e,\ell)}+\mu(c)$.
Since $\mu(c)+\mu(e)=\mu(e^{\prime})$, we obtain $\xi_{\widetilde{G}(e^{\prime},\ell^{\prime})}-\mu(e^{\prime})=\xi_{\widetilde{G}(e,\ell)}-\mu(e)$.
This shows that $\xi_{\widetilde{G}(e,\ell)}-\mu(e)$ does not depend on the choice of $e$, either.
\end{proof}

\section{Non-triviality and primitivity}
\label{sec:non-trivial}

In this section, we first prove that the invariant $\xi_G$ is non-trivial.
Then we consider the primitivity of $\xi_G$ and present some partial results.

Consider the mod $2$ reduction of $\xi_G$:
$$
\xi^2_G:=\xi_G\otimes (1 {\rm \ mod\ } 2)\in H\otimes \mathbb{Z}_2\cong H_1(\Sigma_{g,1};\mathbb{Z}_2).
$$
Hereafter, $\equiv$ stands for an equality in $H\otimes \mathbb{Z}_2$.
Since $\mu(\bar{e})=-\mu(e)\equiv \mu(e)\in H\otimes \mathbb{Z}_2$ for any $e\in E^{\rm ori}(G)$, the homology marking $\mu$ induces a well-defined map $\mu^2\colon E(G)\to H\otimes \mathbb{Z}_2$.
We call $\mu^2$ the \emph{mod $2$ homology marking}.

\begin{prop}
\label{prop:mod2_formula}
Let $G$ be a trivalent fatgraph spine of $\Sigma_{g,1}$.
Then we have
$$
\xi^2_G=\sum_{e\in E(G)} \mu^2(e).
$$
\end{prop}

\begin{proof}
Let $v\in V^{\rm int}(G)$.
We work with Figure \ref{fig:2types} and count preferably oriented edges toward $v$.
By abuse of notation, we use the same letter for an oriented edge and its underlying unoriented edge.
If $v$ is of type 1, only $e_1$ has the preferred orientation.
Since $\mu(e_1)+\mu(e_2)+\mu(e_3)=0$, we have
$$
\mu(e_v)-\mu(f_v)=\mu(e_2)-\mu(e_3)\equiv \mu(e_1).
$$
If $v$ is of type 2, $e_1$ and $e_3$ have the preferred orientation and $e_2$ does not.
Then we have
$$
\mu(e_v)-\mu(f_v)=\mu(e_1)-\mu(e_3)\equiv \mu(e_1)+\mu(e_3).
$$
Therefore, we have
\begin{align*}
\xi^2_G &=\sum_{v\in V^{\rm int}(G)}
\left( \begin{array}{l} {\rm the\ sum\ of\ the\ mod\ } 2 {\rm \ homology\ markings} \\
{\rm of\ preferably\ oriented\ edges\ toward\ } v \end{array} \right) \\
&= \sum_{e\in E(G)} \mu^2(e).
\end{align*}
The last equality holds since any preferably oriented edge of $G$ points to some trivalent vertex of $G$.
\end{proof}

\begin{thm}
\label{thm:non-trivial_b}
Let $G$ be a trivalent fatgraph spine of $\Sigma_{g,1}$.
Then the mod $2$ reduction $\xi^2_G$ is non-trivial.
In particular, we have $\xi_G\neq 0$.
\end{thm}

To prove this theorem, we need the following lemma.

\begin{lem}
\label{lem:odd_edge_cycle}
Let $G$ be a trivalent fatgraph spine of $\Sigma_{g,1}$.
Then $G$ contains an edge cycle of odd length.
\end{lem}

\begin{proof}
We introduce a terminology; a pair of consecutive oriented edges of $G$ is called a \emph{corner} of $G$.
There are $3\# V^{\rm int}(G)=3(4g-1)$ corners.
We number them as $c_1,\ldots,c_{3(4g-1)}$, so that $c_1$ contains the preferably oriented tail of $G$, and for each $i$, $c_i$ and $c_{i+1}$ share an oriented edge in common.
There are $n_o:=6g-1$ odd numbered corners, and $n_e:=6g-2$ even numbered corners.

Since $n_o$ and $n_e$ are not divisible by $3$, there exist distinct indices $i$ and $j$ with $1\le i<j\le 3(4g-1)$ such that the corners $c_i$ and $c_j$ are around the same vertex and $i-j\equiv 1\ {\rm mod\ } 2$. 
We can write $c_i$ and $c_j$ as $c_i=(e_i,e_i^{\prime})$ and $c_j=(e_j,e_j^{\prime})$ with $e_i\prec e_i^{\prime}$ and $e_j\prec e_j^{\prime}$.
Consider the edge cycle following consecutive oriented edges of $G$ from $e_i^{\prime}$ to $e_j$.
Since $i$ and $j$ have different parity, the length of this edge cycle must be odd.
This completes the proof.
\end{proof}

\begin{proof}[Proof of Theorem \ref{thm:non-trivial_b}]
By Lemma \ref{lem:odd_edge_cycle}, $G$ contains an edge cycle $\gamma$ of odd length.
By Proposition \ref{prop:mod2_formula}, the mod $2$ intersection pairing of $\xi^2_G$ and $\gamma$ is computed as
$$
(\xi^2_G\cdot \gamma)=\left( \sum_{e\in E(G)} \mu^2(e)\cdot \gamma \right)=({\rm the\ length\ of\ } \gamma)=1.
$$
Therefore, $\xi^2_G\neq 0$.
\end{proof}

An element $x\in H$ is called \emph{primitive} if there do not exist $m\in \mathbb{Z}$ and $y\in H$ such that $|m|\ge 2$ and $x=my$.
Note that $x$ is primitive if and only if there exists an element $x^{\prime}\in H$ such that $(x\cdot x^{\prime})=\pm 1$.
With Theorem \ref{thm:non-trivial_b} in mind, we would like to pose the following question.

\begin{quest}
For any trivalent fatgraph spine $G\subset \Sigma_{g,1}$, is $\xi_G$ a primitive element of $H$?
\end{quest}

The answer to this question is affirmative if $g\le 2$.
In fact, there is only one combinatorial isomorphism class of trivalent fatgraph spines for $g=1$, and there are $105$ combinatorial isomorphism classes for $g=2$.
By a direct computation, we can show the primitivity of $\xi_G$ for these cases.
The question remains open for $g\ge 3$.

In the below, we show that $\xi_G$ is primitive if $G$ is of a few special types.
We recall that for a fatgraph spine $G$ of $\Sigma_{g,1}$, the \emph{greedy algorithm} \cite{ABP} gives a maximal tree $T_G$ of $G$.
By definition, the set of vertices of $T_G$ is $V(G)$.
For $e\in E(G)$, give the preferred orientation to $e$ and let $v$ be the end point of $e$.
Then $e$ is an edge of $T_G$ if and only if $e\prec e^{\prime}$ for any $e^{\prime}\in E_v^{\rm ori}(G)$ with $e^{\prime}\neq e$.
One can see that $T_G$ is a maximal tree of $G$ and contains the tail of $G$. 

A trivalent fatgraph spine $G\subset \Sigma_{g,1}$ is a \emph{chord diagram} of genus $g$ if the first $4g$ elements in $E^{\rm ori}(G)$ have the preferred orientation.
To work with chord diagrams, we set up some notation.
Name and order the first $4g$ elements in $E^{\rm ori}(G)$ as $e_0,e_1,\ldots,e_{4g-2}$, and $f_0$.
Their underlying unoriented edges are distinct.
Let $\{ f_k\}_{k=1}^{2g-1}\subset E(G)$ be the other unoriented edges of $G$.
We give the preferred orientation to $f_k$, and let $v_k$ and $v_k^{\prime}$ be the start and end points of $f_k$.
Also, let $v_0^{\prime}$ be the end point of $f_0$.
We have $V^{\rm int}(G)=\{ v_k,v_k^{\prime} \}_{k=1}^{2g-1} \sqcup \{ v_0^{\prime} \}$.
Note that the maximal tree $T_G$ is straight, and the set of edges is $\{ e_i\}_{i=0}^{4g-2}$.
We can regard $G$ as a linear chord diagram constructed by attaching $2g$ chords $\{ f_k\}_{k=0}^{2g-1}$ suitably to the interval $[0,4g]$ at the integer points $\{ 1,\ldots, 4g\}$, where we identify $e_i$ with the interval $[i,i+1]$ (when we see it as a fatgraph spine, we remove the bivalent vertex at $4g\in [0,4g]$).
This is an explanation for our terminology. See also \cite{ABP} \cite{B1} \cite{B2}.

\begin{thm}
\label{thm:primitive1}
If $G$ is a chord diagram of genus $g$, then $\xi_G$ is a primitive element of $H$.
\end{thm}

\begin{proof}
We use the notation in the paragraph before the statement of the theorem.

\begin{figure}
\begin{center}
\caption{The proof of Theorem \ref{thm:primitive1}}
\label{fig:chorddiagram}
\input{chorddiagram.tex}
\end{center}
%\vspace{0.3cm}
\end{figure}

We fix an index $k$ with $1\le k\le 2g-1$ and compute the contribution from $v_k$ and $v_k^{\prime}$ to $\xi_G$.
Let $i_k$ and $i_k^{\prime}$ be indices such that $\bar{f}_k$ is next to $\bar{e}_{i_k}$ in $E^{\rm ori}_{v_k}(G)$ and $e_{i^{\prime}_k}$ is next to $f_k$ in $E^{\rm ori}_{v^{\prime}_k}(G)$.
We have either $i_k>i^{\prime}_k$ or $i_k<i^{\prime}_k$. See Figure \ref{fig:chorddiagram}.

If $i_k>i^{\prime}_k$, the contribution from $v_k$ is $\mu(\bar{e}_{i_k})-\mu(\bar{f}_k)=-\mu(e_{i_k})+\mu(f_k)$, and that from $v_k^{\prime}$ is $\mu(e_{i^{\prime}_k})-\mu(f_k)$.
As a total, the contribution from $v_k$ and $v_k^{\prime}$ is $\mu(e_{i^{\prime}_k})-\mu(e_{i_k})$.
Moreover, the edge cycle
$$
\gamma_k:=f_k\bar{e}_{i^{\prime}_k}\bar{e}_{i^{\prime}_k-1}\cdots \bar{e}_{i_k}
$$
represents the homology class $\mu(e_{i^{\prime}_k})-\mu(e_{i_k})$.
If $i_k<i^{\prime}_k$, we can argue similarly and the contribution from $v_k$ and $v_k^{\prime}$ is represented by the edge cycle
$$
\gamma_k=f_ke_{i^{\prime}_k+1}\cdots e_{i_k-1}.
$$
We observe that for any $1\le k\le 2g-1$ and $0\le k^{\prime}\le 2g-1$, one has
\begin{equation}
\label{eq:-delta}
(\gamma_k\cdot f_{k^{\prime}})=-\delta_{kk^{\prime}},
\end{equation}
where $\delta_{kk^{\prime}}$ is Kronecker's delta, and for simplicity we write $f_{k^{\prime}}$ instead of $\mu(f_{k^{\prime}})$.
This is because $\gamma_k\cap (G\setminus {\rm Int}(f_{k^{\prime}}))=\emptyset$ for $k\neq k^{\prime}$, and $\gamma_k$ intersects $\hat{f}_k$ once.

Let $i_0$ be an index such that $\bar{e}_{i_0}$ is next to $f_0$ in $E^{\rm ori}(G)$. The vertex $v_0^{\prime}$ is of type 2 and its contribution is given by $\mu(e_{i_0})-\mu(f_0)$.

We conclude that
\begin{equation}
\label{eq:xi_G=chord}
\xi_G=\sum_{k=1}^{2g-1} \gamma_k+\mu(e_{i_0})-\mu(f_0).
\end{equation}
By $(e_{i_0}\cdot f_0)=-1$ and (\ref{eq:-delta}), we have $(\xi_G\cdot f_0)=-1$.
This shows that $\xi_G$ is primitive.
\end{proof}

We can also prove the primitivity of $\xi_G$ for a trivalent fatgraph spine $G\subset \Sigma_{g,1}$ which is obtained from a chord diagram of genus $g$ by a single flip. 

\begin{thm}
\label{thm:primitive2}
Let $G$ be a trivalent fatgraph spine of $\Sigma_{g,1}$ and assume that the first $4g-1$ elements in $E^{\rm ori}(G)$ have the preferred orientation,
and that the $4g$-th element does not.
Then $\xi_G$ is a primitive element of $H$.
\end{thm}

\begin{proof}
Name and order the first $4g-1$ elements in $E^{\rm ori}(G)$ as $e_0,e_1,\ldots,e_{4g-2}$.
There exist an integer $n$ with $0\le n\le 4g-3$ and $h\in E(G)$ such that if $v$ is the end point of $e_n$, then $e_n$, $\bar{e}_{n+1}$, and $\bar{h}$ are in this order with respect to the cyclic ordering given to $E^{\rm ori}_v(G)$, and the set of edges of $T_G$ is $\{ e_i\}_{i=0}^{4g-3}\cup \{h\}$.
We denote by $v^{\prime}$ the end point of $h$ with the preferred orientation.
Let $f_1,f_2\in E(G)$ be edges which are different from $h$ and have $v^{\prime}$ as an end point.
We arrange that $f_1$ with the preferred orientation is next to $h$ in $E^{\rm ori}(G)$.

Let $G^{\prime}$ be the result of flip along $h$, and let $h^{\prime}\in E(G^{\prime})$ be the edge corresponding to $h\in E(G)$.
Then $G^{\prime}$ is a chord diagram in the sense of Theorem \ref{thm:primitive1}, and $\{ e_i\}_{i=0}^{4g-3} \cup \{ h^{\prime}\}$ becomes the set of edges of $T_{G^{\prime}}$.
Extending $f_1$ and $f_2$ to $\{ f_k\}_{k=1}^{2g-1}$ as in the proof of Theorem \ref{thm:primitive1}, by (\ref{eq:xi_G=chord}) we have
\begin{equation}
\label{eq:xi_G'}
\xi_{G^{\prime}}=\sum_{k=1}^{2g-1}\gamma_k+\mu(e_{i_0})-\mu(e_{4g-2}),
\end{equation}
where $i_0$ is an index such that $\bar{e}_{i_0}$ is next to $e_{4g-2}$ in $E^{\rm ori}(G^{\prime})$.
By the last part of the proof of Theorem \ref{thm:primitive1}, we have
\begin{equation}
\label{eq:xi_G_4g-2}
(\xi_{G^{\prime}}\cdot e_{4g-2})=-1.
\end{equation}
Note that for $j=1,2$, we have
\begin{equation}
\label{eq:eq_for_f_j}
(e_{i_0}\cdot f_j)=0 \quad {\rm \ and\ } \quad
\left(\sum_{k=1}^{2g-1}\gamma_k \cdot f_j \right)=-1.
\end{equation}
In fact, the first equation of (\ref{eq:eq_for_f_j}) follows from $\bar{e}_{i_0}\prec f_j$, and the second one follows from (\ref{eq:-delta}).

By Proposition \ref{prop:char_xi_G} (2), we have
\begin{equation}
\label{eq:xi_G_G'}
\xi_G=\xi_{G^{\prime}}-2j^{\prime}(W_h)+m(W_h).
\end{equation}
Furthermore we have $m(W_h)=\mu(e_n)+\mu(\bar{f}_1)=\mu(e_n)-\mu(f_1)$, and
\begin{eqnarray*}
j^{\prime}(W_h) &= (e_n\cdot \bar{e}_{n+1})\mu(\bar{f}_1)+(\bar{e}_{n+1}\cdot \bar{f}_1)\mu(e_n)+(\bar{f}_1\cdot e_n)\mu(\bar{e}_{n+1}) \\
&= (e_n\cdot e_{n+1})\mu(f_1)+(e_{n+1}\cdot f_1)\mu(e_n)+(f_1\cdot e_n)\mu(e_{n+1}).
\end{eqnarray*}

Now we have either $\bar{f}_1\prec \bar{e}_n$ or $\bar{e}_n\prec \bar{f}_1$ in $E^{\rm ori}(G)$.

\begin{figure}
\begin{center}
\caption{The proof of Theorem \ref{thm:primitive2}}
\label{fig:primitivity}
\input{primitive.tex}
\end{center}
%\vspace{0.3cm}
\end{figure}

Case 1: assume that $\bar{f}_1\prec \bar{e}_n$.
See the left part of Figure \ref{fig:primitivity}.
Then any two elements of $\{ \mu(e_n), \mu(e_{n+1}), \mu(f_1), \mu(f_2) \}$ have the trivial intersection pairing.
By a direct computation using equations (\ref{eq:xi_G'}) to (\ref{eq:xi_G_G'}),
%(\ref{eq:xi_G_4g-2}) (\ref{eq:mu(e_{i_0})}) (\ref{eq:sum_gamma_k})  (\ref{eq:xi_G_G'}),
we obtain
\begin{equation}
\label{eq:xi_G_3}
\begin{cases}
(\xi_G \cdot e_{4g-2}) &=-1+(e_n \cdot e_{4g-2})-(f_1 \cdot e_{4g-2}), \\
(\xi_G \cdot f_1) &=(f_1\cdot e_{4g-2})-1, \\
(\xi_G \cdot f_2) &=(f_2\cdot e_{4g-2})-1.
\end{cases}
\end{equation}
If $(f_2\cdot e_{4g-2})=0$, we have $(\xi_G\cdot f_2)=-1$.
Otherwise, we see that $(e_n\cdot e_{4g-2})=(f_1\cdot e_{4g-2})=0$ and hence $(\xi_G\cdot e_{4g-2})=-1$.

Case 2: assume that $\bar{e}_n\prec \bar{f}_1$.
See the right part of Figure \ref{fig:primitivity}.
We have $(e_n\cdot e_{n+1})=(f_1\cdot e_{n+1})=(f_2\cdot e_{n+1})=0$ and $(e_n\cdot f_1)=(e_n\cdot f_2)=(f_1\cdot f_2)=-1$. Again by using equations (\ref{eq:xi_G'}) to (\ref{eq:xi_G_G'}), we obtain
\begin{equation}
\begin{cases}
(\xi_G\cdot e_{4g-2}) &=-1+(e_n\cdot e_{4g-2})-(f_1\cdot e_{4g-2})-2(e_{n+1}\cdot e_{4g-2}), \\
(\xi_G\cdot f_1) &= (f_1\cdot e_{4g-2})-2, \\
(\xi_G\cdot f_2) &= (f_2\cdot e_{4g-2})-1.
\end{cases}
\end{equation}
If $(f_2\cdot e_{4g-2})=0$, we have $(\xi_G \cdot f_2)=-1$.
Otherwise, we have $(f_2\cdot e_{4g-2})=1$ and we obtain $(\xi_G\cdot f_1)=-1$.

In any case, we can find an element $x^{\prime}\in H$ such that $(\xi_G\cdot x^{\prime})=-1$.
Therefore, $\xi_G$ is primitive.
\end{proof}

In the case of trivalent fatgraph spines of a once punctured surface $\Sigma_g^1$, it can happen that $\xi_G=0$. Two examples of $G\subset \Sigma_2^1$ with $\xi_G=0$ are given in Figure \ref{fig:trivial}.

\begin{figure}
\begin{center}
\caption{Trivalent fatgraph spines with $\xi_G=0$}
\label{fig:trivial}
\vspace{0.3cm}
\input{punctured_trivial.tex}
\end{center}
\end{figure}
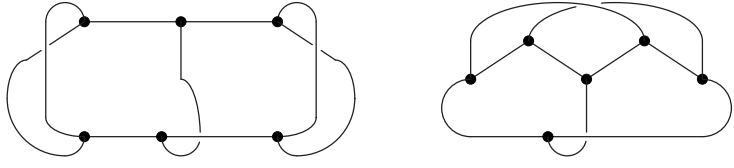

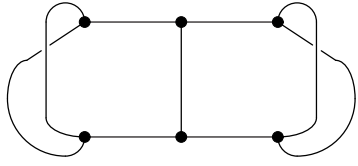
\begin{figure}
\begin{center}
\caption{A balanced trivalent fatgraph spine with $\xi_G\neq 0$}
\label{fig:balanced}
\vspace{0.3cm}
\input{balanced.tex}
\end{center}
\end{figure}

Let $G$ be a trivalent fatgraph spine of $\Sigma_g^1$.
Recall from the proof of Corollary \ref{cor:punctured} that if $G\subset \Sigma_g^1$, then $E^{\rm ori}(G)$ is cyclically ordered.
Hence it is possible to talk about corners of $G$ in a way similar to the case of trivalent fatgraph spines of $\Sigma_{g,1}$.
Now we give labels $\alpha$ or $\beta$ to each corner of $G$ so that any pair of consecutive corners of $G$ have distinct labels. Since the number of corners of $G$ is even, this labeling is always possible and is determined once we choose the label of a fixed corner.

We say that $G$ is \emph{balanced} if for any vertex of $G$, the three corners around the vertex have the same label.
For example, trivalent fatgraph spines in Figure 10 and Figure 11 are balanced.

\begin{thm}
\label{thm:non-trivial_p}
Let $G$ be a trivalent fatgraph spine of $\Sigma_g^1$.
Then the mod $2$ reduction $\xi_G^2=\xi_G\otimes (1 {\rm \ mod\ } 2)$ is trivial if and only if $G$ is balanced.
\end{thm}

\begin{proof}
Pick a corner $c$ of $G$ and write it as $c=(e,e^{\prime})$, where $e^{\prime}$ is next to $e$ in the cyclic ordering given to $E^{\rm ori}(G)$.
We give the label $\alpha$ to $c$ and extend this labeling to all other corners as above.
Take an admissible arc $\ell$ for $e$ and set $\widetilde{G}=\widetilde{G}(e,\ell)$.
The oriented edge $e$ is split at the middle point $v_e$ into two oriented edges.
We name them as $e_1, e_2\in E^{\rm ori}(\widetilde{G})$ so that $v_e$ is the end point of $e_1$.
We extend the labeling of corners of $G$ to that of corners of $\widetilde{G}$ by giving $\alpha$ to $(e_1,\bar{\ell})$ and $(\bar{e}_2,\bar{e}_1)$, and $\beta$ to $(\ell, e_2)$.

In view of Corollary \ref{cor:punctured}, the condition $\xi^2_G=0$ is equivalent to $\xi_{\widetilde{G}}^2=\mu^2(e_2)$.
Furthermore, since the mod $2$ homology markings $\{ \mu^2(f)\}_{f\in E(\widetilde{G})}$ generate the mod $2$ homology $H_1(\Sigma_{g,1};\mathbb{Z}_2)$, this condition is equivalent to the condition that $(\xi_{\widetilde{G}}^2\cdot \mu^2(f))=(\mu^2(e_2)\cdot \mu^2(f))$ for any $f\in E(\widetilde{G})$.

Assume that $G$ is balanced. For any vertex of $\widetilde{G}$ other than $v_e$, the three corners about it is labeled by the same symbol.
Let $f\in E(\widetilde{G})$. 
Let $\gamma(f)$ be the edge cycle following consecutive oriented edges of $\widetilde{G}$ from $f$ to $\bar{f}$, where we give the preferred orientation to $f$.
The mod $2$ homology class $\mu^2(f)$ is represented by $\gamma(f)$.
By the property of the labeling, the length of this edges cycle is odd if $f\prec \bar{e}_2 \prec \bar{f}$ (this also implies $f\neq e_2$), and is even otherwise.
Note that the condition $f\prec \bar{e}_2 \prec \bar{f}$ is equivalent to $(\mu^2(e_2)\cdot \mu^2(f))=1$.
Hence $(\xi_{\widetilde{G}}^2\cdot \mu^2(f))=({\rm the\ length\ of\ } \gamma(f))=1$ if and only if $(\mu^2(e_2)\cdot \mu^2(f))=1$.
Therefore, $\xi^2_G=0$.

On the other hand, assume that $\xi^2_G=0$.
Then for $f\in E(\widetilde{G})$, the length of $\gamma(f)$ is odd if and only if $f\prec \bar{e}_2 \prec \bar{f}$.
Now we remove the tail from $\widetilde{G}$ and go back to $G$.
Then $\gamma(f)$ is reduced to an edge cycle of $G$.
Its length is 1 less than the length of $\gamma(f)$ if $f\prec \bar{e}_2 \prec \bar{f}$, and is the same as the length of $\gamma(f)$ otherwise.
This implies that the reduced edge cycle of $G$ has even length.
Since $f$ can be arbitrary, this shows that $G$ is balanced.
\end{proof}

\section{Mod $2$ reduction and spin structures}
\label{sec:spin}

In this section, we give a topological interpretation of the mod 2 reduction $\xi^2_G$.
Namely, we show that to any trivalent fatgraph spine $G\subset \Sigma_{g,1}$ one can associate two distinct spin structures on $\Sigma_{g,1}$, and that $\xi^2_G$ is the difference of them.

We use the following description of the mod $2$ homology of $\Sigma_{g,1}$.

\begin{lem}
\label{lem:H_1(G)}
Let $G$ be a fatgraph spine of $\Sigma_{g,1}$.
For $v\in V^{\rm int}(G)$, let $\{ e_i^v\}_i$ be the unoriented edges of $G$ having $v$ as an end point.
If there is an edge loop based at $v$, we count it twice.
Then the mod $2$ homology marking induces an isomorphism
$$
H_1(\Sigma_{g,1};\mathbb{Z}_2)\cong
\bigoplus_{e\in E(G)} \mathbb{Z}_2 e \left/ \sum_{v\in V^{\rm int}(G)} \mathbb{Z}_2(\sum_i e_i^v) \right..
$$
\end{lem}
\begin{proof}
Recall from \S \ref{sec:fatgraph} that we associate an oriented simple loop $\hat{e}$ to each (oriented) edge $e$.
In the proof of this lemma we forget the orientation of $e$ and $\hat{e}$.
We can arrange that the simple loops $\{ \hat{e} \}_{e\in E(G)}$ share only one point $q\in \partial \Sigma_{g,1}$, and that if $t$ is the tail of $G$ then $\hat{t}=\partial \Sigma_{g,1}$ with basepoint $q$.
Then we obtain a cell decomposition of $\Sigma_{g,1}$ whose $1$-cells coincide with $\{ \hat{e} \}_{e\in E(G)}$.
Now the right hand side of the assertion can be identified with the first mod $2$ cellular homology group of this cell decomposition.
\end{proof}

Recall that a \emph{spin structure} on $\Sigma_{g,1}$ is an element $w\in H^1(UT\Sigma_{g,1};\mathbb{Z}_2)$, where $UT\Sigma_{g,1}$ is the unit tangent bundle of $\Sigma_{g,1}$ (with respect to some Riemannian metric), such that the restriction of $w$ to a fiber of the projection $UT\Sigma_{g,1}\to \Sigma_{g,1}$ is non-trivial.
As Johnson \cite{JohSpin} showed, the set of spin structures on $\Sigma_{g,1}$ is naturally identified with the set of quadratic forms on $H_1(\Sigma_{g,1};\mathbb{Z}_2)$.
Here, a map $q\colon H_1(\Sigma_{g,1};\mathbb{Z}_2)\to \mathbb{Z}_2$ is called a \emph{quadratic form} on $H_1(\Sigma_{g,1};\mathbb{Z}_2)$ if it satisfies
$$
q(x+y)=q(x)+q(y)+(x\cdot y)
$$
for any $x,y\in H_1(\Sigma_{g,1};\mathbb{Z}_2)$.
The set of spin structures on $\Sigma_{g,1}$ is a torsor under the action of $H^1(\Sigma_{g,1};\mathbb{Z}_2)$.
In other words, the difference of two quadratic forms on $H_1(\Sigma_{g,1};\mathbb{Z}_2)$ can be written as a uniquely determined element of ${\rm Hom}(H_1(\Sigma_{g,1};\mathbb{Z}_2),\mathbb{Z}_2)\cong H^1(\Sigma_{g,1};\mathbb{Z}_2)$.
Note that using the mod $2$ intersection paring, we have a natural isomorphism
\begin{equation}
\label{eq:pd}
H_1(\Sigma_{g,1};\mathbb{Z}_2) \cong {\rm Hom}(H_1(\Sigma_{g,1};\mathbb{Z}_2),\mathbb{Z}_2),
\quad x\mapsto [y\mapsto (x\cdot y)].
\end{equation}

In what follows, $G$ is a trivalent fatgraph spine of $\Sigma_{g,1}$.
The following result gives an identification of certain $\mathbb{Z}_2$-valued functions on $E(G)$ with the set of quadratic forms on $H_1(\Sigma_{g,1};\mathbb{Z}_2)$, thus with the set of spin structures on $\Sigma_{g,1}$ via Johnson's result stated above.

\begin{thm}
\label{thm:Q(tau)}
Let $G$ be a trivalent fatgraph spine of $\Sigma_{g,1}$.
Let $Q(G)$ be the set of maps $q\colon E(G)\to \mathbb{Z}_2$ such that for any $v\in V^{\rm int}(G)$, the sum of values of $q$ at the three edges having $v$ as an end point is $0$ if $v$ is of type $1$, and is $1$ if $v$ is of type $2$.
Then there is a natural bijection from $Q(G)$ to the set of quadratic forms on $H_1(\Sigma_{g,1};\mathbb{Z}_2)$.
\end{thm}

\begin{proof}
Given a map $q\colon E(G)\to \mathbb{Z}_2$, we extend $q$ to a map from the free $\mathbb{Z}_2$-module generated by $E(G)$ by
\begin{equation}
\label{eq:q_extended}
q\left( \sum_{e\in E(G)} m_e e\right):=\sum_{e\in E(G)} m_e q(e)+\sum_{e\prec e^{\prime}} m_e m_{e^{\prime}} (\mu^2(e)\cdot \mu^2(e^{\prime})),
\end{equation}
for $m_e\in \mathbb{Z}_2$, $e\in E(G)$.
Here $(\ \cdot \ )$ is the mod $2$ intersection pairing and we give the preferred orientation to each element of $E(G)$.
By a direct computation, we can check that for any $x,y\in \bigoplus_{e\in E(G)} \mathbb{Z}_2 e$,
\begin{equation}
\label{eq:quadratic}
q(x+y)=q(x)+q(y)+(x\cdot y).
\end{equation}
Here $(x\cdot y)$ is the mod $2$ intersection pairing of the homology class determined by $x$ and $y$ through the isomorphism in Lemma \ref{lem:H_1(G)}.

We claim that if $q\in Q(G)$, then for any $v\in V^{\rm int}(G)$,
$$q(e_1^v+e_2^v+e_3^v)=0.$$
By (\ref{eq:q_extended}), this condition is equivalent to the following.
\begin{equation}
\label{eq:extend_3}
\sum_{i=1}^3q(e_i^v)+(\mu^2(e_1^v)\cdot \mu^2(e_2^v))+(\mu^2(e_1^v)\cdot \mu^2(e_3^v))+(\mu^2(e_2^v)\cdot \mu^2(e_3^v))=0.
\end{equation}
If $v$ is of type $1$, then $(\mu^2(e_i^v)\cdot \mu^2(e_j^v))=0$ for any $1\le i,j\le 3$.
If $v$ is of type $2$, then $(\mu^2(e_i^v)\cdot \mu^2(e_j^v))=1$ for any $1\le i,j\le 3$ with $i\neq j$.
See Figure \ref{fig:2types}.
Therefore, the condition (\ref{eq:extend_3}) is exactly equivalent to the condition for $q$ being an element of $Q(G)$.
This proves the claim.

By the claim, Lemma \ref{lem:H_1(G)}, and (\ref{eq:quadratic}), it follows that the map $q$ induces a quadratic form on $H_1(\Sigma_{g,1};\mathbb{Z}_2)$.
The above construction gives a map from $Q(G)$ to the set of quadratic forms on $H_1(\Sigma_{g,1};\mathbb{Z}_2)$, and the inverse of this map is given by composing any quadratic form on $H_1(\Sigma_{g,1};\mathbb{Z}_2)$ with the mod $2$ homology marking $\mu^2\colon E(G)\to H_1(\Sigma_{g,1};\mathbb{Z}_2)$.
\end{proof}

We record how the set $Q(G)$ changes under a flip.

\begin{prop}
\label{prop:q_flip}
Let $W=W_e$ be a flip from $G$ to $G'$.
Then the bijection in Theorem \ref{thm:Q(tau)} induces a bijection from $Q(G)$ to $Q(G')$, which maps a given $q\in Q(G)$ to the element $q'\in Q(G')$ defined as follows.
\begin{itemize}
\item For any edge $f$ in $E(G')\setminus \{e' \} \cong E(G)\setminus \{ e\}$, we have $q'(f)=q(f)$.
\item We adopt the notation in Figure \ref{fig:6cases}, and assume that in each case $G$ and $G'$ correspond to the left and right pictures, respectively. Then the value $q'(e')$ is given by the following formula.
\begin{align*}
I:& \ q'(e')=q(b)+q(c)=q(a)+q(d), \\
II:& \ q'(e')=q(b)+q(c)=q(a)+q(d)+1, \\
III:& \ q'(e')=q(b)+q(c)+1=q(a)+q(d), \\
IV:& \ q'(e')=q(b)+q(c)+1=q(a)+q(d), \\
V:& \ q'(e')=q(b)+q(c)=q(a)+q(d)+1, \\
VI:& \ q'(e')=q(b)+q(c)+1=q(a)+q(d)+1.
\end{align*}
\end{itemize}
By a suitable replacement of labels of edges, one can similarly obtain a formula for $q'$ in terms of $q$ for the case where $G$ and $G'$ correspond to the right and left pictures, respectively, in each case in Figure \ref{fig:6cases}.
\end{prop}

\begin{proof}
To prove the first condition, note that the mod $2$ homology marking of $f$ as an edge of $E(G)$ is the same as that of $f$ as an edge of $E(G')$.
The second condition follows from the first condition and the defining relation for elements of $Q(G')$. For example, in case VI, two end points of $e'$ are of type $2$, and hence we have $q'(b)+q'(c)+q'(e')=q'(a)+q'(d)+q'(e')=1$.
\end{proof}

\begin{rem}
The description of spin structures on $\Sigma_{g,1}$ given in Theorem \ref{thm:Q(tau)} and how it changes under a flip as in Proposition \ref{prop:q_flip} was pointed out by Robert Penner \cite{PenPriv}.
Recently, Penner and Zeitlin \cite{PZ15} give another natural description of spin structures on a punctured surface in terms of orientations on a trivalent fatgraph spine of the surface, and they also show how it changes under a flip.
In other words, Penner and Zeitlin give a lift of the action of the mapping class group on the set of quadratic forms to the action of the Ptolemy groupoid, and the present construction gives another lift.
It should be remarked that while their description works for any surfaces with multiple punctures, our description here is for a once (punctured/bordered) surface.
It is an interesting question whether ours generalizes to any (punctured/bordered) surface.
\end{rem}

Now we give the preferred orientation (Definition \ref{dfn:total_order}) to each unoriented edge of $G$ and use the same letter for the resulting oriented edge of $G$.
For example, if we write $e\prec f\prec \bar{e}$ or $e\prec \bar{f}\prec \bar{e}$ for $e,f\in E(G)$, we understand that $e$ and $f$ have the preferred orientation.
Take $e\in E(G)$.
We define an element $q_G(e), \bar{q}_G(e)\in \mathbb{Z}_2$ by
$$
q_G(e)=\# \{ f\in E(G)|\ e\prec f\prec \bar{e} \} \ {\rm mod\ } 2,
$$
and
$$
\bar{q}_G(e)=\# \{ f\in E(G)|\ e\prec \bar{f}\prec \bar{e} \} \ {\rm mod\ } 2.
$$
Here $\#$ means the number of elements of a set.

\begin{prop}
\label{prop:spin}
The maps $q_G$ and $\bar{q}_G$ are elements of $Q(G)$.
In particular, they induce quadratic forms on $H_1(\Sigma_{g,1};\mathbb{Z}_2)$.
\end{prop}

\begin{proof}
We consider the case of $q_G$ only.

We work with Figure \ref{fig:2types}.
Suppose that $v$ is of type $1$. 
Then $e_1$, $\bar{e}_2$, and $\bar{e}_3$ have the preferred orientation, and we have a disjoint sum decomposition
\begin{align*}
& \{ f\in E(G)|\ e_1 \prec f\prec \bar{e}_1\} \\
=&\{ e_2, e_3\} \sqcup
\{ f\in E(G)|\ \bar{e}_2\prec f\prec e_2\} \sqcup \{ f\in E(G)|\ \bar{e}_3\prec f\prec e_3\}.
\end{align*}
This implies that $q_G(e_1)=q_G(e_2)+q_G(e_3)$.

Suppose that $v$ is of type $2$.
Then $e_1$, $\bar{e}_2$, and $e_3$ have the preferred orientation, and we have a disjoint sum decomposition
\begin{align*}
& \{ f\in E(G)|\ \bar{e}_2\prec f\prec e_2\} \\
=& (\{ f\in E(G)|\ e_1\prec f\prec \bar{e}_1\} \setminus \{ e_2 \} ) \sqcup \{ f\in E(G)|\ e_3\prec f\prec \bar{e}_3\}.
\end{align*}
This implies that $q_G(e_2)=q_G(e_1)+q_G(e_3)+1$.

Therefore, $q_G\in Q(G)$. By Theorem \ref{thm:Q(tau)}, $q_G$ induce a quadratic form on $H_1(\Sigma_{g,1};\mathbb{Z}_2)$.
\end{proof}

For simplicity, we use the same letter $q_G$ and $\bar{q}_G$ for the induced quadratic forms.
This construction of quadratic forms is $\mathcal{M}_{g,1}$-equivariant in the following sense.

\begin{prop}
\label{prop:spin_equiv}
Let $G$ be a trivalent fatgraph spine of $\Sigma_{g,1}$, and let $\varphi\in \mathcal{M}_{g,1}$.
Then we have
\begin{align*}
q_{\varphi(G)}\circ \varphi_* &=q_G, \\
\bar{q}_{\varphi(G)}\circ \varphi_* &=\bar{q}_G,
\end{align*}
where $\varphi_*$ is the automorphism of $H_1(\Sigma_{g,1};\mathbb{Z}_2)$ induced from $\varphi$.
\end{prop}

\begin{proof}
We consider the case of $q_G$ only.
Consider a homomorphism
$$
\Phi\colon \bigoplus_{e\in E(G)} \mathbb{Z}_2 e \to \bigoplus_{e^{\prime}\in E(\varphi(G))} \mathbb{Z}_2 e^{\prime},\quad \Phi(e)=\varphi(e).
$$
Since $\varphi$ gives a combinatorial isomorphism from $G$ to $\varphi(G)$, we have $q_{\varphi(G)}\circ \Phi=q_G$.
Now $\Phi$ induces the map $\varphi_*$ on the level of homology, and we conclude $q_{\varphi(G)}\circ \varphi_*=q_G$. 
\end{proof}

Finally, we compute the difference of $q_G$ and $\bar{q}_G$.

\begin{thm}
\label{thm:xi_spin}
Under the isomorphism (\ref{eq:pd}), we have
$$
q_G-\bar{q}_G=\xi^2_G.
$$
Moreover, we have $q_G \neq \bar{q}_G$.
\end{thm}
 
\begin{proof}
For $e\in E(G)$, we have
\begin{align*}
q_G(e)-\bar{q}_G(e) &=q_G(e)+\bar{q}_G(e) \\
&=\# \{ f\in E^{\rm ori}(G)|\ e\prec f\prec \bar{e} \} \ {\rm mod\ } 2 \\
&=\left(\sum_{f\in E(G)} \mu^2(f)\cdot \mu^2(e) \right)=(\xi^2_G\cdot \mu^2(e)), \\
\end{align*}
where the last equality follows from Proposition \ref{prop:mod2_formula}.
Since $\{ \mu^2(e)\}_{e\in E(G)}$ generates $H_1(\Sigma_{g,1};\mathbb{Z}_2)$, we obtain $q_G-\bar{q}_G=\xi^2_G$.
The second statement follows from Theorem \ref{thm:non-trivial_b}.
\end{proof}

\appendix

\section{A non-singular vector field associated to a once bordered trivalent fatgraph spine}

Let $G$ be a trivalent fatgraph spine of $\Sigma_{g,1}$.
In this appendix, we define a non-singular vector field $\mathcal{X}_G$ on $\Sigma_{g,1}$, and then consider the induced quadratic form on $H_1(\Sigma_{g,1};\mathbb{Z}_2)$.
In particular, we discuss a relationship between this quadratic form, $q_G$, and $\bar{q}_G$.

The following construction of $\mathcal{X}_G$ was communicated to the author by Gw\'ena\"el Massuyeau.

Let ${\rm Vect}(\Sigma_{g,1})$ be the homotopy set of non-singular vector fields on $\Sigma_{g,1}$.
In other words, ${\rm Vect}(\Sigma_{g,1})$ is the homotopy set of sections of the projection $\pi\colon UT\Sigma_{g,1}\to \Sigma_{g,1}$.
For $\mathcal{X}\in {\rm Vect}(\Sigma_{g,1})$, the winding number
$$
{\rm wind}_{\mathcal{X}}\colon \pi_1(UT\Sigma_{g,1})\to \mathbb{Z}
$$
is defined as follows.
Let $\widetilde{\gamma}\colon S^1\to UT\Sigma_{g,1}$ be a (based) loop.
For any $t\in S^1$, there uniquely exists an element $\Phi_t=\Phi(\mathcal{X},\widetilde{\gamma},t)\in S^1=U(1)$ such that $\mathcal{X}(\pi \circ \widetilde{\gamma}(t))\Phi_t=\widetilde{\gamma}(t)$.
Then ${\rm wind}_{\mathcal{X}}(\widetilde{\gamma})$ is defined to be the mapping degree of the map $S^1\to S^1,\ t\mapsto \Phi_t$.
The map ${\rm wind}_{\mathcal{X}}$ is a group homomorphism, and its mod $2$ reduction
$$
w_{\mathcal{X}}\in {\rm Hom}(\pi_1(UT\Sigma_{g,1}),\mathbb{Z}_2) \cong H^1(UT\Sigma_{g,1};\mathbb{Z}_2)
$$
is a spin structure on $\Sigma_{g,1}$.

Now we give the preferred orientation to any unoriented edge of $G$.
Let $v\in V^{\rm int}(G)$.
According to the type of $v$, we realize a small neighborhood $N_v$ of $v$ in the $xy$-plane as in Figure \ref{fig:app}, and then restrict the horizontal vector field $\partial/\partial x$ to $N_v$.
We extend the vector field on $\bigsqcup_v N_v$ thus obtained to a globally defined non-singular vector field $\mathcal{X}_G$, so that outside $\bigsqcup_v N_v$, each trajectory of $\mathcal{X}_G$ is perpendicular to $G$.

\begin{figure}
\begin{center}
\caption{$\mathcal{X}_G$ on $N_v$}
\label{fig:app}
\input{app.tex}
\end{center}
\vspace{0.3cm}
\end{figure}

Let $q_{\mathcal{X}_G}$ be the quadratic form on $H_1(\Sigma_{g,1};\mathbb{Z}_2)$ corresponding to $w_{\mathcal{X}}$.
Following Johnson \cite{JohSpin}, one can compute it as follows.
Let $\gamma$ be an oriented simple closed curve.
Consider the lift $\widetilde{\gamma}=(\gamma,\dot{\gamma})$ of $\gamma$ to a loop in $UT\Sigma_{g,1}$
(here $\dot{\gamma}$ is the velocity vector of $\gamma$ normalized to have the unit length).
Then one has
\begin{equation}
\label{eq:a_1}
q_{\mathcal{X}_G}([\gamma])={\rm wind}_{\mathcal{X}_G}(\widetilde{\gamma})+1 \ {\rm mod\ } 2.
\end{equation}
We apply this formula to $\gamma=\hat{e}$, where $e\in E^{\rm ori}(G)$.
Assume that $e$ has the preferred orientation.
Let $L(e)$ be the set of corners $(f,f')$ of $G$ (see the proof of Lemma \ref{lem:odd_edge_cycle}) such that
\begin{enumerate}
\item $e\preceq f \prec f' \preceq \bar{e}$, and
\item exactly one of $f$ and $f'$ have the preferred orientation.
\end{enumerate}
Here, $e\preceq f$ means $e\prec f$ or $e=f$.
For example, if $v$ is a vertex of type $1$ as in the left part of Figure \ref{fig:app}, only $(\bar{e}_2,e_3)$ is an element of $L(e)$ among the three corners around $v$.
Set $\lambda(e)=\# L(e)$.

\begin{lem}
We have ${\rm wind}_{\mathcal{X}_G}(\widetilde{\hat{e}})=(1-\lambda(e))/2$.
\end{lem}

\begin{proof}
Take a small regular neighborhood $N(G)$ of $G$, and we arrange that $\hat{e}$ stays inside $N(G)$ throughout.
Every time when $\hat{e}$ goes through a common vertex of a member of $L(e)$, the velocity vector of $\hat{e}$ rotates by an angle $-\pi$ with respect to $\mathcal{X}_G$.
Also, when $\hat{e}$ goes through the middle point of $e$, the velocity vector of $\hat{e}$ rotates by an angle $\pi$ with respect to $\mathcal{X}_G$.
This proves the lemma.
\end{proof}

In particular, using the fact that $\lambda(e)$ is odd (since $e$ has the preferred orientation and $\bar{e}$ does not), we have from (\ref{eq:a_1}) that
$$
q_{\mathcal{X}_G}(e)=q_{\mathcal{X}_G}([\hat{e}])=\frac{1-\lambda(e)}{2}+1\ {\rm mod\ } 2=
\frac{1+\lambda(e)}{2}\ {\rm mod\ } 2.
$$

\begin{prop}
Let $G$ be a trivalent fatgraph spine of $\Sigma_{g,1}$.
Then the quadratic forms $q_{\mathcal{X}_G}$, $q_G$, and $\bar{q}_G$ are distinct to each other.
\end{prop}

\begin{proof}
By Theorem \ref{thm:xi_spin}, it is sufficient to prove $q_{\mathcal{X}_G}\neq q_G$ and $q_{\mathcal{X}_G}\neq \bar{q}_G$.

Let $e_1\in E^{\rm ori}(G)$ be the ``last'' prefarably oriented edge.
Namely, $e_1$ is the unique element such that $e_1$ has the preferred orientation and if $e_1\prec f$, $f$ does not have the preferred orientation.
We have $\lambda(e_1)=1$ and $q_{\mathcal{X}_G}(e_1)=(1+1)/2=1$.
On the other hand, since there are no preferably oriented edge $f$ with $e_1\prec f\prec \bar{e}_1$, we have $q_G(e_1)=0$.
Hence $q_{\mathcal{X}_G}\neq q_G$.

Let $e_2\in E^{\rm ori}(G)$ be the unique element such that $e_2$ has the preferred orientation and if $f\prec \bar{e}_2$, $f$ has the preferred orientation.
We have $\lambda(e_2)=1$ and $q_{\mathcal{X}_G}(e_2)=1$.
On the other hand, since any edge $f\in E^{\rm ori}(G)$ with $e_2\prec f\prec \bar{e}_2$ has the preferred orientation, we have $\bar{q}_G(e_2)=0$.
Hence $q_{\mathcal{X}_G}\neq \bar{q}_G$.
\end{proof}

\end{document}

%% file: flip.tex
%WinTpicVersion4.31b
{\unitlength 0.1in%
\begin{picture}( 25.6800, 14.2600)(  2.5000,-14.7000)%
% LINE 1 0 3 0 Black White  
% 2 690 734 1170 734
% 
\special{pn 13}%
\special{pa 690 734}%
\special{pa 1170 734}%
\special{fp}%
% LINE 1 0 3 0 Black White  
% 2 690 734 370 414
% 
\special{pn 13}%
\special{pa 690 734}%
\special{pa 370 414}%
\special{fp}%
% LINE 1 0 3 0 Black White  
% 2 690 734 370 1054
% 
\special{pn 13}%
\special{pa 690 734}%
\special{pa 370 1054}%
\special{fp}%
% LINE 1 0 3 0 Black White  
% 2 1170 734 1490 414
% 
\special{pn 13}%
\special{pa 1170 734}%
\special{pa 1490 414}%
\special{fp}%
% LINE 1 0 3 0 Black White  
% 2 1170 734 1490 1054
% 
\special{pn 13}%
\special{pa 1170 734}%
\special{pa 1490 1054}%
\special{fp}%
% LINE 1 0 3 0 Black White  
% 2 2450 494 2130 174
% 
\special{pn 13}%
\special{pa 2450 494}%
\special{pa 2130 174}%
\special{fp}%
% LINE 1 0 3 0 Black White  
% 2 2450 494 2770 174
% 
\special{pn 13}%
\special{pa 2450 494}%
\special{pa 2770 174}%
\special{fp}%
% LINE 1 0 3 0 Black White  
% 2 2450 494 2450 974
% 
\special{pn 13}%
\special{pa 2450 494}%
\special{pa 2450 974}%
\special{fp}%
% LINE 1 0 3 0 Black White  
% 2 2450 974 2130 1294
% 
\special{pn 13}%
\special{pa 2450 974}%
\special{pa 2130 1294}%
\special{fp}%
% LINE 1 0 3 0 Black White  
% 2 2450 974 2770 1294
% 
\special{pn 13}%
\special{pa 2450 974}%
\special{pa 2770 1294}%
\special{fp}%
% LINE 2 0 3 0 Black White  
% 4 1330 894 1410 894 1330 894 1330 974
% 
\special{pn 8}%
\special{pa 1330 894}%
\special{pa 1410 894}%
\special{fp}%
\special{pa 1330 894}%
\special{pa 1330 974}%
\special{fp}%
% LINE 2 0 3 0 Black White  
% 4 2610 1134 2690 1134 2610 1134 2610 1214
% 
\special{pn 8}%
\special{pa 2610 1134}%
\special{pa 2690 1134}%
\special{fp}%
\special{pa 2610 1134}%
\special{pa 2610 1214}%
\special{fp}%
% LINE 2 0 3 0 Black White  
% 4 530 894 450 894 530 894 530 974
% 
\special{pn 8}%
\special{pa 530 894}%
\special{pa 450 894}%
\special{fp}%
\special{pa 530 894}%
\special{pa 530 974}%
\special{fp}%
% LINE 2 0 3 0 Black White  
% 4 2290 1134 2210 1134 2290 1134 2290 1214
% 
\special{pn 8}%
\special{pa 2290 1134}%
\special{pa 2210 1134}%
\special{fp}%
\special{pa 2290 1134}%
\special{pa 2290 1214}%
\special{fp}%
% LINE 2 0 3 0 Black White  
% 4 530 574 450 574 530 574 530 494
% 
\special{pn 8}%
\special{pa 530 574}%
\special{pa 450 574}%
\special{fp}%
\special{pa 530 574}%
\special{pa 530 494}%
\special{fp}%
% LINE 2 0 3 0 Black White  
% 4 2290 334 2210 334 2290 334 2290 254
% 
\special{pn 8}%
\special{pa 2290 334}%
\special{pa 2210 334}%
\special{fp}%
\special{pa 2290 334}%
\special{pa 2290 254}%
\special{fp}%
% LINE 2 0 3 0 Black White  
% 4 1330 574 1410 574 1330 574 1330 494
% 
\special{pn 8}%
\special{pa 1330 574}%
\special{pa 1410 574}%
\special{fp}%
\special{pa 1330 574}%
\special{pa 1330 494}%
\special{fp}%
% LINE 2 0 3 0 Black White  
% 4 2610 334 2690 334 2610 334 2610 254
% 
\special{pn 8}%
\special{pa 2610 334}%
\special{pa 2690 334}%
\special{fp}%
\special{pa 2610 334}%
\special{pa 2610 254}%
\special{fp}%
% STR 2 0 3 0 Black White  
% 4 1810 654 1810 734 2 0 0 0
% $\Longrightarrow$
\put(18.1000,-7.3400){\makebox(0,0)[lb]{$\Longrightarrow$}}%
% STR 2 0 3 0 Black White  
% 4 1834 814 1834 894 2 0 0 0
% $W_e$
\put(18.3400,-8.9400){\makebox(0,0)[lb]{$W_e$}}%
% STR 2 0 3 0 Black White  
% 4 1514 1078 1514 1158 2 0 0 0
% $a$
\put(15.1400,-11.5800){\makebox(0,0)[lb]{$a$}}%
% STR 2 0 3 0 Black White  
% 4 1514 278 1514 358 2 0 0 0
% $b$
\put(15.1400,-3.5800){\makebox(0,0)[lb]{$b$}}%
% STR 2 0 3 0 Black White  
% 4 258 270 258 350 2 0 0 0
% $c$
\put(2.5800,-3.5000){\makebox(0,0)[lb]{$c$}}%
% STR 2 0 3 0 Black White  
% 4 250 1094 250 1174 2 0 0 0
% $d$
\put(2.5000,-11.7400){\makebox(0,0)[lb]{$d$}}%
% STR 2 0 3 0 Black White  
% 4 874 766 874 846 2 0 0 0
% $e$
\put(8.7400,-8.4600){\makebox(0,0)[lb]{$e$}}%
% STR 2 0 3 0 Black White  
% 4 2810 1334 2810 1414 2 0 0 0
% $a$
\put(28.1000,-14.1400){\makebox(0,0)[lb]{$a$}}%
% STR 2 0 3 0 Black White  
% 4 2818 94 2818 174 2 0 0 0
% $b$
\put(28.1800,-1.7400){\makebox(0,0)[lb]{$b$}}%
% STR 2 0 3 0 Black White  
% 4 1994 94 1994 174 2 0 0 0
% $c$
\put(19.9400,-1.7400){\makebox(0,0)[lb]{$c$}}%
% STR 2 0 3 0 Black White  
% 4 2002 1350 2002 1430 2 0 0 0
% $d$
\put(20.0200,-14.3000){\makebox(0,0)[lb]{$d$}}%
% STR 2 0 3 0 Black White  
% 4 2530 734 2530 814 2 0 0 0
% $e^{\prime}$
\put(25.3000,-8.1400){\makebox(0,0)[lb]{$e^{\prime}$}}%
% STR 2 0 3 0 Black White  
% 4 800 1500 800 1600 2 0 0 0
% $G$
\put(8.0000,-16.0000){\makebox(0,0)[lb]{$G$}}%
% STR 2 0 3 0 Black White  
% 4 2400 1500 2400 1600 2 0 0 0
% $G'$
\put(24.0000,-16.0000){\makebox(0,0)[lb]{$G'$}}%
\end{picture}}%

%% file: pentagon.tex
%WinTpicVersion4.10
\unitlength 0.1in
\begin{picture}( 33.1800, 32.3200)(  7.9000,-34.4700)
% POLYGON 2 0 3 0 Black White
% 6 1502 3301 2062 3301 2235 2768 1782 2439 1329 2768 1502 3301
% 
{\color[named]{Black}{%
\special{pn 8}%
\special{pa 1502 3302}%
\special{pa 2062 3302}%
\special{pa 2236 2768}%
\special{pa 1782 2440}%
\special{pa 1330 2768}%
\special{pa 1502 3302}%
\special{pa 2062 3302}%
\special{fp}%
}}%
% LINE 2 0 3 0 Black White
% 2 2062 3301 1782 2440
% 
{\color[named]{Black}{%
\special{pn 8}%
\special{pa 2062 3302}%
\special{pa 1782 2440}%
\special{fp}%
}}%
% LINE 2 0 3 0 Black White
% 2 1502 3301 1782 2440
% 
{\color[named]{Black}{%
\special{pn 8}%
\special{pa 1502 3302}%
\special{pa 1782 2440}%
\special{fp}%
}}%
% POLYGON 2 0 3 0 Black White
% 6 1065 1971 1625 1971 1798 1438 1345 1109 892 1438 1065 1971
% 
{\color[named]{Black}{%
\special{pn 8}%
\special{pa 1066 1972}%
\special{pa 1626 1972}%
\special{pa 1798 1438}%
\special{pa 1346 1110}%
\special{pa 892 1438}%
\special{pa 1066 1972}%
\special{pa 1626 1972}%
\special{fp}%
}}%
% LINE 2 0 3 0 Black White
% 2 1067 1977 1795 1438
% 
{\color[named]{Black}{%
\special{pn 8}%
\special{pa 1068 1978}%
\special{pa 1796 1438}%
\special{fp}%
}}%
% LINE 2 0 3 0 Black White
% 2 1067 1977 1347 1116
% 
{\color[named]{Black}{%
\special{pn 8}%
\special{pa 1068 1978}%
\special{pa 1348 1116}%
\special{fp}%
}}%
% POLYGON 2 0 3 0 Black White
% 6 3465 3301 2905 3301 2732 2768 3185 2439 3638 2768 3465 3301
% 
{\color[named]{Black}{%
\special{pn 8}%
\special{pa 3466 3302}%
\special{pa 2906 3302}%
\special{pa 2732 2768}%
\special{pa 3186 2440}%
\special{pa 3638 2768}%
\special{pa 3466 3302}%
\special{pa 2906 3302}%
\special{fp}%
}}%
% LINE 2 0 3 0 Black White
% 2 3463 3307 2735 2768
% 
{\color[named]{Black}{%
\special{pn 8}%
\special{pa 3464 3308}%
\special{pa 2736 2768}%
\special{fp}%
}}%
% LINE 2 0 3 0 Black White
% 2 3463 3307 3183 2446
% 
{\color[named]{Black}{%
\special{pn 8}%
\special{pa 3464 3308}%
\special{pa 3184 2446}%
\special{fp}%
}}%
% POLYGON 2 0 3 0 Black White
% 6 2202 1159 2762 1159 2935 626 2482 297 2029 626 2202 1159
% 
{\color[named]{Black}{%
\special{pn 8}%
\special{pa 2202 1160}%
\special{pa 2762 1160}%
\special{pa 2936 626}%
\special{pa 2482 298}%
\special{pa 2030 626}%
\special{pa 2202 1160}%
\special{pa 2762 1160}%
\special{fp}%
}}%
% LINE 2 0 3 0 Black White
% 2 2932 626 2204 1165
% 
{\color[named]{Black}{%
\special{pn 8}%
\special{pa 2932 626}%
\special{pa 2204 1166}%
\special{fp}%
}}%
% POLYGON 2 0 3 0 Black White
% 6 3892 1971 3332 1971 3159 1438 3612 1109 4065 1438 3892 1971
% 
{\color[named]{Black}{%
\special{pn 8}%
\special{pa 3892 1972}%
\special{pa 3332 1972}%
\special{pa 3160 1438}%
\special{pa 3612 1110}%
\special{pa 4066 1438}%
\special{pa 3892 1972}%
\special{pa 3332 1972}%
\special{fp}%
}}%
% LINE 2 0 3 0 Black White
% 2 4065 1438 3176 1438
% 
{\color[named]{Black}{%
\special{pn 8}%
\special{pa 4066 1438}%
\special{pa 3176 1438}%
\special{fp}%
}}%
% LINE 2 0 3 0 Black White
% 2 3169 1438 3897 1977
% 
{\color[named]{Black}{%
\special{pn 8}%
\special{pa 3170 1438}%
\special{pa 3898 1978}%
\special{fp}%
}}%
% LINE 1 0 3 0 Black White
% 14 1350 1557 1770 1137 1350 1557 1490 1837 1490 1837 1910 1697 1490 1837 1350 2117 1350 1557 1070 1557 1070 1557 790 1837 1070 1557 1070 1137
% 
{\color[named]{Black}{%
\special{pn 13}%
\special{pa 1350 1558}%
\special{pa 1770 1138}%
\special{fp}%
\special{pa 1350 1558}%
\special{pa 1490 1838}%
\special{fp}%
\special{pa 1490 1838}%
\special{pa 1910 1698}%
\special{fp}%
\special{pa 1490 1838}%
\special{pa 1350 2118}%
\special{fp}%
\special{pa 1350 1558}%
\special{pa 1070 1558}%
\special{fp}%
\special{pa 1070 1558}%
\special{pa 790 1838}%
\special{fp}%
\special{pa 1070 1558}%
\special{pa 1070 1138}%
\special{fp}%
}}%
% LINE 1 0 3 0 Black White
% 14 3121 2887 2841 2467 3121 2887 2981 3167 2981 3167 2701 3167 2981 3167 3121 3447 3121 2887 3401 2747 3401 2747 3541 2467 3401 2747 3681 3027
% 
{\color[named]{Black}{%
\special{pn 13}%
\special{pa 3122 2888}%
\special{pa 2842 2468}%
\special{fp}%
\special{pa 3122 2888}%
\special{pa 2982 3168}%
\special{fp}%
\special{pa 2982 3168}%
\special{pa 2702 3168}%
\special{fp}%
\special{pa 2982 3168}%
\special{pa 3122 3448}%
\special{fp}%
\special{pa 3122 2888}%
\special{pa 3402 2748}%
\special{fp}%
\special{pa 3402 2748}%
\special{pa 3542 2468}%
\special{fp}%
\special{pa 3402 2748}%
\special{pa 3682 3028}%
\special{fp}%
}}%
% LINE 1 0 3 0 Black White
% 14 1784 2887 1784 3447 1784 2887 1504 2887 1504 2887 1504 2467 1504 2887 1224 3167 1784 2887 2064 2887 2064 2887 2344 3167 2064 2887 2064 2467
% 
{\color[named]{Black}{%
\special{pn 13}%
\special{pa 1784 2888}%
\special{pa 1784 3448}%
\special{fp}%
\special{pa 1784 2888}%
\special{pa 1504 2888}%
\special{fp}%
\special{pa 1504 2888}%
\special{pa 1504 2468}%
\special{fp}%
\special{pa 1504 2888}%
\special{pa 1224 3168}%
\special{fp}%
\special{pa 1784 2888}%
\special{pa 2064 2888}%
\special{fp}%
\special{pa 2064 2888}%
\special{pa 2344 3168}%
\special{fp}%
\special{pa 2064 2888}%
\special{pa 2064 2468}%
\special{fp}%
}}%
% LINE 1 0 3 0 Black White
% 14 3548 1557 4108 1697 3548 1557 3548 1277 3548 1277 3268 1137 3548 1277 3968 1137 3548 1557 3408 1837 3408 1837 3688 2117 3408 1837 3128 1837
% 
{\color[named]{Black}{%
\special{pn 13}%
\special{pa 3548 1558}%
\special{pa 4108 1698}%
\special{fp}%
\special{pa 3548 1558}%
\special{pa 3548 1278}%
\special{fp}%
\special{pa 3548 1278}%
\special{pa 3268 1138}%
\special{fp}%
\special{pa 3548 1278}%
\special{pa 3968 1138}%
\special{fp}%
\special{pa 3548 1558}%
\special{pa 3408 1838}%
\special{fp}%
\special{pa 3408 1838}%
\special{pa 3688 2118}%
\special{fp}%
\special{pa 3408 1838}%
\special{pa 3128 1838}%
\special{fp}%
}}%
% LINE 1 0 3 0 Black White
% 14 2484 745 1924 1025 2484 745 2484 465 2484 465 2064 325 2484 465 2904 325 2484 745 2624 1025 2624 1025 2484 1305 2624 1025 3044 885
% 
{\color[named]{Black}{%
\special{pn 13}%
\special{pa 2484 746}%
\special{pa 1924 1026}%
\special{fp}%
\special{pa 2484 746}%
\special{pa 2484 466}%
\special{fp}%
\special{pa 2484 466}%
\special{pa 2064 326}%
\special{fp}%
\special{pa 2484 466}%
\special{pa 2904 326}%
\special{fp}%
\special{pa 2484 746}%
\special{pa 2624 1026}%
\special{fp}%
\special{pa 2624 1026}%
\special{pa 2484 1306}%
\special{fp}%
\special{pa 2624 1026}%
\special{pa 3044 886}%
\special{fp}%
}}%
% LINE 2 0 3 0 Black White
% 2 2043 626 2925 626
% 
{\color[named]{Black}{%
\special{pn 8}%
\special{pa 2044 626}%
\special{pa 2926 626}%
\special{fp}%
}}%
% STR 2 0 3 0 Black White
% 4 2497 545 2497 615 2 0 0 0
% $f$
\put(24.9700,-6.1500){\makebox(0,0)[lb]{$f$}}%
% STR 2 0 3 0 Black White
% 4 2427 858 2427 928 2 0 0 0
% $g$
\put(24.2700,-9.2800){\makebox(0,0)[lb]{$g$}}%
% STR 2 0 3 0 Black White
% 4 1270 1658 1270 1728 2 0 0 0
% $g_1$
\put(12.7000,-17.2800){\makebox(0,0)[lb]{$g_1$}}%
% STR 2 0 3 0 Black White
% 4 1219 1471 1219 1541 2 0 0 0
% $f_1$
\put(12.1900,-15.4100){\makebox(0,0)[lb]{$f_1$}}%
% STR 2 0 3 0 Black White
% 4 1610 2991 1610 3061 2 0 0 0
% $f_2$
\put(16.1000,-30.6100){\makebox(0,0)[lb]{$f_2$}}%
% STR 2 0 3 0 Black White
% 4 1809 2954 1809 3024 2 0 0 0
% $g_2$
\put(18.0900,-30.2400){\makebox(0,0)[lb]{$g_2$}}%
% STR 2 0 3 0 Black White
% 4 3008 3111 3008 3181 2 0 0 0
% $f_3$
\put(30.0800,-31.8100){\makebox(0,0)[lb]{$f_3$}}%
% STR 2 0 3 0 Black White
% 4 3184 2888 3184 2958 2 0 0 0
% $g_3$
\put(31.8400,-29.5800){\makebox(0,0)[lb]{$g_3$}}%
% STR 2 0 3 0 Black White
% 4 3440 1787 3440 1857 2 0 0 0
% $f_4$
\put(34.4000,-18.5700){\makebox(0,0)[lb]{$f_4$}}%
% STR 2 0 3 0 Black White
% 4 3564 1344 3564 1414 2 0 0 0
% $g_4$
\put(35.6400,-14.1400){\makebox(0,0)[lb]{$g_4$}}%
% STR 2 0 3 0 Black White
% 4 3080 810 3080 910 2 0 0 0
% $a$
\put(30.8000,-9.1000){\makebox(0,0)[lb]{$a$}}%
% STR 2 0 3 0 Black White
% 4 2947 255 2947 355 2 0 0 0
% $b$
\put(29.4700,-3.5500){\makebox(0,0)[lb]{$b$}}%
% STR 2 0 3 0 Black White
% 4 1937 245 1937 345 2 0 0 0
% $c$
\put(19.3700,-3.4500){\makebox(0,0)[lb]{$c$}}%
% STR 2 0 3 0 Black White
% 4 1857 875 1857 975 2 0 0 0
% $d$
\put(18.5700,-9.7500){\makebox(0,0)[lb]{$d$}}%
% STR 2 0 3 0 Black White
% 4 2407 1315 2407 1415 2 0 0 0
% $e$
\put(24.0700,-14.1500){\makebox(0,0)[lb]{$e$}}%
\end{picture}%

%% file: 2types.tex
%WinTpicVersion4.31b
{\unitlength 0.1in%
\begin{picture}( 35.4000, 26.2200)(  4.4300,-27.8600)%
% CIRCLE 2 0 3 0 Black White  
% 4 1212 969 595 969 595 969 595 969
% 
\special{pn 8}%
\special{ar 1212 969 617 617  0.0000000  6.2831853}%
% DOT 0 0 3 0 Black White  
% 1 672 1277
% 
\special{pn 4}%
\special{sh 1}%
\special{ar 672 1277 16 16 0  6.28318530717959E+0000}%
% DOT 0 0 3 0 Black White  
% 1 595 969
% 
\special{pn 4}%
\special{sh 1}%
\special{ar 595 969 16 16 0  6.28318530717959E+0000}%
% DOT 0 0 3 0 Black White  
% 1 672 660
% 
\special{pn 4}%
\special{sh 1}%
\special{ar 672 660 16 16 0  6.28318530717959E+0000}%
% DOT 0 0 3 0 Black White  
% 1 1751 660
% 
\special{pn 4}%
\special{sh 1}%
\special{ar 1751 660 16 16 0  6.28318530717959E+0000}%
% DOT 0 0 3 0 Black White  
% 1 1751 1277
% 
\special{pn 4}%
\special{sh 1}%
\special{ar 1751 1277 16 16 0  6.28318530717959E+0000}%
% DOT 0 0 3 0 Black White  
% 1 1212 352
% 
\special{pn 4}%
\special{sh 1}%
\special{ar 1212 352 16 16 0  6.28318530717959E+0000}%
% DOT 0 0 3 0 Black White  
% 1 1829 969
% 
\special{pn 4}%
\special{sh 1}%
\special{ar 1829 969 16 16 0  6.28318530717959E+0000}%
% DOT 0 0 3 0 Black White  
% 1 1212 1586
% 
\special{pn 4}%
\special{sh 1}%
\special{ar 1212 1586 16 16 0  6.28318530717959E+0000}%
% DOT 0 0 3 0 Black White  
% 1 903 429
% 
\special{pn 4}%
\special{sh 1}%
\special{ar 903 429 16 16 0  6.28318530717959E+0000}%
% DOT 0 0 3 0 Black White  
% 1 1520 429
% 
\special{pn 4}%
\special{sh 1}%
\special{ar 1520 429 16 16 0  6.28318530717959E+0000}%
% LINE 2 0 3 0 Black White  
% 4 610 1123 608 1192 610 1123 665 1166
% 
\special{pn 8}%
\special{pa 610 1123}%
\special{pa 608 1192}%
\special{fp}%
\special{pa 610 1123}%
\special{pa 665 1166}%
\special{fp}%
% LINE 2 0 3 0 Black White  
% 4 1813 1123 1815 1192 1813 1123 1759 1166
% 
\special{pn 8}%
\special{pa 1813 1123}%
\special{pa 1815 1192}%
\special{fp}%
\special{pa 1813 1123}%
\special{pa 1759 1166}%
\special{fp}%
% LINE 2 0 3 0 Black White  
% 4 610 815 608 746 610 815 665 772
% 
\special{pn 8}%
\special{pa 610 815}%
\special{pa 608 746}%
\special{fp}%
\special{pa 610 815}%
\special{pa 665 772}%
\special{fp}%
% LINE 2 0 3 0 Black White  
% 4 1813 815 1815 746 1813 815 1759 772
% 
\special{pn 8}%
\special{pa 1813 815}%
\special{pa 1815 746}%
\special{fp}%
\special{pa 1813 815}%
\special{pa 1759 772}%
\special{fp}%
% LINE 2 0 3 0 Black White  
% 4 1368 377 1417 424 1368 377 1436 366
% 
\special{pn 8}%
\special{pa 1368 377}%
\special{pa 1417 424}%
\special{fp}%
\special{pa 1368 377}%
\special{pa 1436 366}%
\special{fp}%
% LINE 2 0 3 0 Black White  
% 4 1060 377 1010 424 1060 377 992 366
% 
\special{pn 8}%
\special{pa 1060 377}%
\special{pa 1010 424}%
\special{fp}%
\special{pa 1060 377}%
\special{pa 992 366}%
\special{fp}%
% STR 2 0 3 0 Black White  
% 4 1214 1656 1214 1734 2 0 0 0
% $p$
\put(12.1400,-17.3400){\makebox(0,0)[lb]{$p$}}%
% STR 2 0 3 0 Black White  
% 4 443 1071 443 1148 2 0 0 0
% $e_1$
\put(4.4300,-11.4800){\makebox(0,0)[lb]{$e_1$}}%
% STR 2 0 3 0 Black White  
% 4 443 762 443 840 2 0 0 0
% $e_2$
\put(4.4300,-8.4000){\makebox(0,0)[lb]{$e_2$}}%
% STR 2 0 3 0 Black White  
% 4 1021 223 1021 300 2 0 0 0
% $e_2$
\put(10.2100,-3.0000){\makebox(0,0)[lb]{$e_2$}}%
% STR 2 0 3 0 Black White  
% 4 1368 223 1368 300 2 0 0 0
% $e_3$
\put(13.6800,-3.0000){\makebox(0,0)[lb]{$e_3$}}%
% STR 2 0 3 0 Black White  
% 4 1905 760 1905 838 2 0 0 0
% $e_3$
\put(19.0500,-8.3800){\makebox(0,0)[lb]{$e_3$}}%
% STR 2 0 3 0 Black White  
% 4 1905 1069 1905 1146 2 0 0 0
% $e_1$
\put(19.0500,-11.4600){\makebox(0,0)[lb]{$e_1$}}%
% DOT 0 0 3 0 Black White  
% 1 1203 2358
% 
\special{pn 4}%
\special{sh 1}%
\special{ar 1203 2358 16 16 0  6.28318530717959E+0000}%
% LINE 2 0 3 0 Black White  
% 2 1203 2358 1203 2038
% 
\special{pn 8}%
\special{pa 1203 2358}%
\special{pa 1203 2038}%
\special{fp}%
% LINE 2 0 3 0 Black White  
% 2 1203 2358 883 2518
% 
\special{pn 8}%
\special{pa 1203 2358}%
\special{pa 883 2518}%
\special{fp}%
% LINE 2 0 3 0 Black White  
% 2 1203 2358 1523 2518
% 
\special{pn 8}%
\special{pa 1203 2358}%
\special{pa 1523 2518}%
\special{fp}%
% LINE 2 0 3 0 Black White  
% 4 1203 2198 1235 2134 1203 2198 1171 2134
% 
\special{pn 8}%
\special{pa 1203 2198}%
\special{pa 1235 2134}%
\special{fp}%
\special{pa 1203 2198}%
\special{pa 1171 2134}%
\special{fp}%
% LINE 2 0 3 0 Black White  
% 4 1359 2436 1407 2489 1359 2436 1430 2430
% 
\special{pn 8}%
\special{pa 1359 2436}%
\special{pa 1407 2489}%
\special{fp}%
\special{pa 1359 2436}%
\special{pa 1430 2430}%
\special{fp}%
% LINE 2 0 3 0 Black White  
% 4 1039 2436 991 2489 1039 2436 968 2430
% 
\special{pn 8}%
\special{pa 1039 2436}%
\special{pa 991 2489}%
\special{fp}%
\special{pa 1039 2436}%
\special{pa 968 2430}%
\special{fp}%
% CIRCLE 2 0 3 0 Black White  
% 4 1203 2358 1299 2310 1299 2310 1203 2446
% 
\special{pn 8}%
\special{ar 1203 2358 107 107  1.5707963  5.8195377}%
% DOT 0 0 3 0 Black White  
% 1 1299 2310
% 
\special{pn 4}%
\special{sh 1}%
\special{ar 1299 2310 16 16 0  6.28318530717959E+0000}%
% LINE 2 0 3 0 Black White  
% 4 1208 2470 1170 2432 1208 2470 1155 2478
% 
\special{pn 8}%
\special{pa 1208 2470}%
\special{pa 1170 2432}%
\special{fp}%
\special{pa 1208 2470}%
\special{pa 1155 2478}%
\special{fp}%
% STR 2 0 3 0 Black White  
% 4 1240 2420 1240 2500 2 0 0 0
% $v$
\put(12.4000,-25.0000){\makebox(0,0)[lb]{$v$}}%
% STR 2 0 3 0 Black White  
% 4 1537 2518 1537 2598 2 0 0 0
% $e_1$
\put(15.3700,-25.9800){\makebox(0,0)[lb]{$e_1$}}%
% STR 2 0 3 0 Black White  
% 4 1163 1910 1163 1990 2 0 0 0
% $e_2$
\put(11.6300,-19.9000){\makebox(0,0)[lb]{$e_2$}}%
% STR 2 0 3 0 Black White  
% 4 761 2502 761 2582 2 0 0 0
% $e_3$
\put(7.6100,-25.8200){\makebox(0,0)[lb]{$e_3$}}%
% STR 2 0 3 0 Black White  
% 4 1211 916 1211 996 2 0 0 0
% $D_G$
\put(12.1100,-9.9600){\makebox(0,0)[lb]{$D_G$}}%
% STR 2 0 3 0 Black White  
% 4 600 2836 600 2916 2 0 0 0
% a vertex of type 1
\put(6.0000,-29.1600){\makebox(0,0)[lb]{a vertex of type 1}}%
% CIRCLE 2 0 3 0 Black White  
% 4 3289 964 2673 964 2673 964 2673 964
% 
\special{pn 8}%
\special{ar 3289 964 616 616  0.0000000  6.2831853}%
% DOT 0 0 3 0 Black White  
% 1 2749 1272
% 
\special{pn 4}%
\special{sh 1}%
\special{ar 2749 1272 16 16 0  6.28318530717959E+0000}%
% DOT 0 0 3 0 Black White  
% 1 2673 964
% 
\special{pn 4}%
\special{sh 1}%
\special{ar 2673 964 16 16 0  6.28318530717959E+0000}%
% DOT 0 0 3 0 Black White  
% 1 2749 655
% 
\special{pn 4}%
\special{sh 1}%
\special{ar 2749 655 16 16 0  6.28318530717959E+0000}%
% DOT 0 0 3 0 Black White  
% 1 3829 655
% 
\special{pn 4}%
\special{sh 1}%
\special{ar 3829 655 16 16 0  6.28318530717959E+0000}%
% DOT 0 0 3 0 Black White  
% 1 3829 1272
% 
\special{pn 4}%
\special{sh 1}%
\special{ar 3829 1272 16 16 0  6.28318530717959E+0000}%
% DOT 0 0 3 0 Black White  
% 1 3289 347
% 
\special{pn 4}%
\special{sh 1}%
\special{ar 3289 347 16 16 0  6.28318530717959E+0000}%
% DOT 0 0 3 0 Black White  
% 1 3906 964
% 
\special{pn 4}%
\special{sh 1}%
\special{ar 3906 964 16 16 0  6.28318530717959E+0000}%
% DOT 0 0 3 0 Black White  
% 1 3289 1580
% 
\special{pn 4}%
\special{sh 1}%
\special{ar 3289 1580 16 16 0  6.28318530717959E+0000}%
% DOT 0 0 3 0 Black White  
% 1 2981 424
% 
\special{pn 4}%
\special{sh 1}%
\special{ar 2981 424 16 16 0  6.28318530717959E+0000}%
% DOT 0 0 3 0 Black White  
% 1 3597 424
% 
\special{pn 4}%
\special{sh 1}%
\special{ar 3597 424 16 16 0  6.28318530717959E+0000}%
% LINE 2 0 3 0 Black White  
% 4 2688 1117 2685 1186 2688 1117 2742 1160
% 
\special{pn 8}%
\special{pa 2688 1117}%
\special{pa 2685 1186}%
\special{fp}%
\special{pa 2688 1117}%
\special{pa 2742 1160}%
\special{fp}%
% LINE 2 0 3 0 Black White  
% 4 3890 1117 3893 1186 3890 1117 3837 1160
% 
\special{pn 8}%
\special{pa 3890 1117}%
\special{pa 3893 1186}%
\special{fp}%
\special{pa 3890 1117}%
\special{pa 3837 1160}%
\special{fp}%
% LINE 2 0 3 0 Black White  
% 4 2688 809 2685 740 2688 809 2742 766
% 
\special{pn 8}%
\special{pa 2688 809}%
\special{pa 2685 740}%
\special{fp}%
\special{pa 2688 809}%
\special{pa 2742 766}%
\special{fp}%
% LINE 2 0 3 0 Black White  
% 4 3890 809 3893 740 3890 809 3837 766
% 
\special{pn 8}%
\special{pa 3890 809}%
\special{pa 3893 740}%
\special{fp}%
\special{pa 3890 809}%
\special{pa 3837 766}%
\special{fp}%
% LINE 2 0 3 0 Black White  
% 4 3445 372 3495 419 3445 372 3513 360
% 
\special{pn 8}%
\special{pa 3445 372}%
\special{pa 3495 419}%
\special{fp}%
\special{pa 3445 372}%
\special{pa 3513 360}%
\special{fp}%
% LINE 2 0 3 0 Black White  
% 4 3137 372 3088 419 3137 372 3069 360
% 
\special{pn 8}%
\special{pa 3137 372}%
\special{pa 3088 419}%
\special{fp}%
\special{pa 3137 372}%
\special{pa 3069 360}%
\special{fp}%
% STR 2 0 3 0 Black White  
% 4 3292 1651 3292 1728 2 0 0 0
% $p$
\put(32.9200,-17.2800){\makebox(0,0)[lb]{$p$}}%
% STR 2 0 3 0 Black White  
% 4 2521 1065 2521 1142 2 0 0 0
% $e_1$
\put(25.2100,-11.4200){\makebox(0,0)[lb]{$e_1$}}%
% STR 2 0 3 0 Black White  
% 4 2521 756 2521 834 2 0 0 0
% $e_2$
\put(25.2100,-8.3400){\makebox(0,0)[lb]{$e_2$}}%
% STR 2 0 3 0 Black White  
% 4 3099 217 3099 294 2 0 0 0
% $e_3$
\put(30.9900,-2.9400){\makebox(0,0)[lb]{$e_3$}}%
% STR 2 0 3 0 Black White  
% 4 3445 217 3445 294 2 0 0 0
% $e_1$
\put(34.4500,-2.9400){\makebox(0,0)[lb]{$e_1$}}%
% STR 2 0 3 0 Black White  
% 4 3983 755 3983 832 2 0 0 0
% $e_2$
\put(39.8300,-8.3200){\makebox(0,0)[lb]{$e_2$}}%
% STR 2 0 3 0 Black White  
% 4 3983 1064 3983 1140 2 0 0 0
% $e_3$
\put(39.8300,-11.4000){\makebox(0,0)[lb]{$e_3$}}%
% DOT 0 0 3 0 Black White  
% 1 3281 2352
% 
\special{pn 4}%
\special{sh 1}%
\special{ar 3281 2352 16 16 0  6.28318530717959E+0000}%
% LINE 2 0 3 0 Black White  
% 2 3281 2352 3281 2032
% 
\special{pn 8}%
\special{pa 3281 2352}%
\special{pa 3281 2032}%
\special{fp}%
% LINE 2 0 3 0 Black White  
% 2 3281 2352 2961 2512
% 
\special{pn 8}%
\special{pa 3281 2352}%
\special{pa 2961 2512}%
\special{fp}%
% LINE 2 0 3 0 Black White  
% 2 3281 2352 3601 2512
% 
\special{pn 8}%
\special{pa 3281 2352}%
\special{pa 3601 2512}%
\special{fp}%
% LINE 2 0 3 0 Black White  
% 4 3281 2192 3313 2128 3281 2192 3249 2128
% 
\special{pn 8}%
\special{pa 3281 2192}%
\special{pa 3313 2128}%
\special{fp}%
\special{pa 3281 2192}%
\special{pa 3249 2128}%
\special{fp}%
% LINE 2 0 3 0 Black White  
% 4 3437 2430 3485 2484 3437 2430 3508 2424
% 
\special{pn 8}%
\special{pa 3437 2430}%
\special{pa 3485 2484}%
\special{fp}%
\special{pa 3437 2430}%
\special{pa 3508 2424}%
\special{fp}%
% LINE 2 0 3 0 Black White  
% 4 3117 2430 3069 2484 3117 2430 3045 2424
% 
\special{pn 8}%
\special{pa 3117 2430}%
\special{pa 3069 2484}%
\special{fp}%
\special{pa 3117 2430}%
\special{pa 3045 2424}%
\special{fp}%
% CIRCLE 2 0 3 0 Black White  
% 4 3293 2369 3382 2310 3217 2326 3382 2310
% 
\special{pn 8}%
\special{ar 3293 2369 107 107  5.6977801  3.6564774}%
% DOT 0 0 3 0 Black White  
% 1 3382 2310
% 
\special{pn 4}%
\special{sh 1}%
\special{ar 3382 2310 16 16 0  6.28318530717959E+0000}%
% LINE 2 0 3 0 Black White  
% 4 3194 2302 3208 2353 3194 2302 3161 2344
% 
\special{pn 8}%
\special{pa 3194 2302}%
\special{pa 3208 2353}%
\special{fp}%
\special{pa 3194 2302}%
\special{pa 3161 2344}%
\special{fp}%
% STR 2 0 3 0 Black White  
% 4 3169 2184 3169 2264 2 0 0 0
% $v$
\put(31.6900,-22.6400){\makebox(0,0)[lb]{$v$}}%
% STR 2 0 3 0 Black White  
% 4 3614 2512 3614 2592 2 0 0 0
% $e_1$
\put(36.1400,-25.9200){\makebox(0,0)[lb]{$e_1$}}%
% STR 2 0 3 0 Black White  
% 4 3241 1904 3241 1984 2 0 0 0
% $e_2$
\put(32.4100,-19.8400){\makebox(0,0)[lb]{$e_2$}}%
% STR 2 0 3 0 Black White  
% 4 2838 2496 2838 2576 2 0 0 0
% $e_3$
\put(28.3800,-25.7600){\makebox(0,0)[lb]{$e_3$}}%
% STR 2 0 3 0 Black White  
% 4 3289 910 3289 990 2 0 0 0
% $D_G$
\put(32.8900,-9.9000){\makebox(0,0)[lb]{$D_G$}}%
% STR 2 0 3 0 Black White  
% 4 2677 2830 2677 2910 2 0 0 0
% a vertex of type 2
\put(26.7700,-29.1000){\makebox(0,0)[lb]{a vertex of type 2}}%
\end{picture}}%

%% file: proof.tex
%WinTpicVersion4.10
\unitlength 0.1in
\begin{picture}( 36.1600, 23.6000)(  5.1000,-24.2400)
% CIRCLE 2 0 3 0 Black White
% 4 1382 1298 582 1298 582 1298 582 1298
% 
{\color[named]{Black}{%
\special{pn 8}%
\special{ar 1382 1298 800 800  0.0000000 6.2831853}%
}}%
% DOT 0 0 3 0 Black White
% 1 902 1938
% 
{\color[named]{Black}{%
\special{pn 4}%
\special{sh 1}%
\special{ar 902 1938 26 26 0  6.28318530717959E+0000}%
}}%
% DOT 0 0 3 0 Black White
% 1 742 1778
% 
{\color[named]{Black}{%
\special{pn 4}%
\special{sh 1}%
\special{ar 742 1778 26 26 0  6.28318530717959E+0000}%
}}%
% DOT 0 0 3 0 Black White
% 1 646 1602
% 
{\color[named]{Black}{%
\special{pn 4}%
\special{sh 1}%
\special{ar 646 1602 26 26 0  6.28318530717959E+0000}%
}}%
% DOT 0 0 3 0 Black White
% 1 638 1010
% 
{\color[named]{Black}{%
\special{pn 4}%
\special{sh 1}%
\special{ar 638 1010 26 26 0  6.28318530717959E+0000}%
}}%
% DOT 0 0 3 0 Black White
% 1 742 818
% 
{\color[named]{Black}{%
\special{pn 4}%
\special{sh 1}%
\special{ar 742 818 26 26 0  6.28318530717959E+0000}%
}}%
% DOT 0 0 3 0 Black White
% 1 902 658
% 
{\color[named]{Black}{%
\special{pn 4}%
\special{sh 1}%
\special{ar 902 658 26 26 0  6.28318530717959E+0000}%
}}%
% DOT 0 0 3 0 Black White
% 1 1382 498
% 
{\color[named]{Black}{%
\special{pn 4}%
\special{sh 1}%
\special{ar 1382 498 26 26 0  6.28318530717959E+0000}%
}}%
% DOT 0 0 3 0 Black White
% 1 1614 530
% 
{\color[named]{Black}{%
\special{pn 4}%
\special{sh 1}%
\special{ar 1614 530 26 26 0  6.28318530717959E+0000}%
}}%
% DOT 0 0 3 0 Black White
% 1 1822 626
% 
{\color[named]{Black}{%
\special{pn 4}%
\special{sh 1}%
\special{ar 1822 626 26 26 0  6.28318530717959E+0000}%
}}%
% DOT 0 0 3 0 Black White
% 1 1990 778
% 
{\color[named]{Black}{%
\special{pn 4}%
\special{sh 1}%
\special{ar 1990 778 26 26 0  6.28318530717959E+0000}%
}}%
% DOT 0 0 3 0 Black White
% 1 2182 1298
% 
{\color[named]{Black}{%
\special{pn 4}%
\special{sh 1}%
\special{ar 2182 1298 26 26 0  6.28318530717959E+0000}%
}}%
% DOT 0 0 3 0 Black White
% 1 2142 1530
% 
{\color[named]{Black}{%
\special{pn 4}%
\special{sh 1}%
\special{ar 2142 1530 26 26 0  6.28318530717959E+0000}%
}}%
% DOT 0 0 3 0 Black White
% 1 2038 1746
% 
{\color[named]{Black}{%
\special{pn 4}%
\special{sh 1}%
\special{ar 2038 1746 26 26 0  6.28318530717959E+0000}%
}}%
% DOT 0 0 3 0 Black White
% 1 1862 1938
% 
{\color[named]{Black}{%
\special{pn 4}%
\special{sh 1}%
\special{ar 1862 1938 26 26 0  6.28318530717959E+0000}%
}}%
% DOT 0 0 3 0 Black White
% 1 1382 2098
% 
{\color[named]{Black}{%
\special{pn 4}%
\special{sh 1}%
\special{ar 1382 2098 26 26 0  6.28318530717959E+0000}%
}}%
% STR 2 0 3 0 Black White
% 4 1382 2178 1382 2258 2 0 0 0
% $p$
\put(13.8200,-22.5800){\makebox(0,0)[lb]{$p$}}%
% STR 2 0 3 0 Black White
% 4 678 1914 678 1994 2 0 0 0
% $a$
\put(6.7800,-19.9400){\makebox(0,0)[lb]{$a$}}%
% LINE 2 0 3 0 Black White
% 4 784 1828 800 1895 784 1828 848 1855
% 
{\color[named]{Black}{%
\special{pn 8}%
\special{pa 784 1828}%
\special{pa 800 1896}%
\special{fp}%
\special{pa 784 1828}%
\special{pa 848 1856}%
\special{fp}%
}}%
% LINE 2 0 3 0 Black White
% 4 1896 684 1920 749 1896 684 1963 703
% 
{\color[named]{Black}{%
\special{pn 8}%
\special{pa 1896 684}%
\special{pa 1920 750}%
\special{fp}%
\special{pa 1896 684}%
\special{pa 1964 704}%
\special{fp}%
}}%
% LINE 2 0 3 0 Black White
% 4 701 1716 696 1646 701 1716 642 1679
% 
{\color[named]{Black}{%
\special{pn 8}%
\special{pa 702 1716}%
\special{pa 696 1646}%
\special{fp}%
\special{pa 702 1716}%
\special{pa 642 1680}%
\special{fp}%
}}%
% LINE 2 0 3 0 Black White
% 4 794 757 861 740 794 757 819 692
% 
{\color[named]{Black}{%
\special{pn 8}%
\special{pa 794 758}%
\special{pa 862 740}%
\special{fp}%
\special{pa 794 758}%
\special{pa 820 692}%
\special{fp}%
}}%
% LINE 2 0 3 0 Black White
% 4 690 901 635 943 690 901 692 971
% 
{\color[named]{Black}{%
\special{pn 8}%
\special{pa 690 902}%
\special{pa 636 944}%
\special{fp}%
\special{pa 690 902}%
\special{pa 692 972}%
\special{fp}%
}}%
% LINE 2 0 3 0 Black White
% 4 1501 515 1443 477 1501 515 1436 540
% 
{\color[named]{Black}{%
\special{pn 8}%
\special{pa 1502 516}%
\special{pa 1444 478}%
\special{fp}%
\special{pa 1502 516}%
\special{pa 1436 540}%
\special{fp}%
}}%
% LINE 2 0 3 0 Black White
% 4 2164 1472 2210 1420 2164 1472 2148 1404
% 
{\color[named]{Black}{%
\special{pn 8}%
\special{pa 2164 1472}%
\special{pa 2210 1420}%
\special{fp}%
\special{pa 2164 1472}%
\special{pa 2148 1404}%
\special{fp}%
}}%
% LINE 2 0 3 0 Black White
% 4 1982 1826 1919 1856 1982 1826 1970 1895
% 
{\color[named]{Black}{%
\special{pn 8}%
\special{pa 1982 1826}%
\special{pa 1920 1856}%
\special{fp}%
\special{pa 1982 1826}%
\special{pa 1970 1896}%
\special{fp}%
}}%
% STR 2 0 3 0 Black White
% 4 510 1706 510 1786 2 0 0 0
% $b$
\put(5.1000,-17.8600){\makebox(0,0)[lb]{$b$}}%
% STR 2 0 3 0 Black White
% 4 510 850 510 930 2 0 0 0
% $c$
\put(5.1000,-9.3000){\makebox(0,0)[lb]{$c$}}%
% STR 2 0 3 0 Black White
% 4 646 602 646 682 2 0 0 0
% $d$
\put(6.4600,-6.8200){\makebox(0,0)[lb]{$d$}}%
% STR 2 0 3 0 Black White
% 4 1462 338 1462 418 2 0 0 0
% $b$
\put(14.6200,-4.1800){\makebox(0,0)[lb]{$b$}}%
% STR 2 0 3 0 Black White
% 4 1982 578 1982 658 2 0 0 0
% $c$
\put(19.8200,-6.5800){\makebox(0,0)[lb]{$c$}}%
% STR 2 0 3 0 Black White
% 4 1702 418 1702 498 2 0 0 0
% $e$
\put(17.0200,-4.9800){\makebox(0,0)[lb]{$e$}}%
% STR 2 0 3 0 Black White
% 4 2262 1378 2262 1458 2 0 0 0
% $d$
\put(22.6200,-14.5800){\makebox(0,0)[lb]{$d$}}%
% STR 2 0 3 0 Black White
% 4 2158 1650 2158 1730 2 0 0 0
% $e$
\put(21.5800,-17.3000){\makebox(0,0)[lb]{$e$}}%
% STR 2 0 3 0 Black White
% 4 2022 1858 2022 1938 2 0 0 0
% $a$
\put(20.2200,-19.3800){\makebox(0,0)[lb]{$a$}}%
% SPLINE 2 0 3 0 Black White
% 4 718 1682 1350 1074 1518 546 1518 546
% 
{\color[named]{Black}{%
\special{pn 8}%
\special{pa 718 1682}%
\special{pa 746 1662}%
\special{pa 772 1644}%
\special{pa 880 1564}%
\special{pa 958 1504}%
\special{pa 984 1482}%
\special{pa 1060 1420}%
\special{pa 1084 1398}%
\special{pa 1108 1376}%
\special{pa 1130 1354}%
\special{pa 1176 1310}%
\special{pa 1198 1286}%
\special{pa 1218 1262}%
\special{pa 1238 1238}%
\special{pa 1258 1214}%
\special{pa 1276 1190}%
\special{pa 1294 1164}%
\special{pa 1312 1138}%
\special{pa 1328 1112}%
\special{pa 1344 1086}%
\special{pa 1360 1058}%
\special{pa 1374 1030}%
\special{pa 1386 1002}%
\special{pa 1398 974}%
\special{pa 1410 944}%
\special{pa 1422 914}%
\special{pa 1432 884}%
\special{pa 1442 854}%
\special{pa 1468 760}%
\special{pa 1476 730}%
\special{pa 1484 698}%
\special{pa 1492 666}%
\special{pa 1498 634}%
\special{pa 1506 602}%
\special{pa 1514 570}%
\special{pa 1518 546}%
\special{sp}%
}}%
% SPLINE 2 0 3 0 Black White
% 4 718 954 1366 994 1878 738 1878 738
% 
{\color[named]{Black}{%
\special{pn 8}%
\special{pa 718 954}%
\special{pa 752 960}%
\special{pa 784 968}%
\special{pa 816 974}%
\special{pa 850 980}%
\special{pa 882 984}%
\special{pa 914 990}%
\special{pa 948 994}%
\special{pa 980 1000}%
\special{pa 1012 1004}%
\special{pa 1076 1010}%
\special{pa 1140 1014}%
\special{pa 1204 1014}%
\special{pa 1234 1012}%
\special{pa 1296 1006}%
\special{pa 1328 1002}%
\special{pa 1358 996}%
\special{pa 1388 990}%
\special{pa 1418 982}%
\special{pa 1448 972}%
\special{pa 1476 962}%
\special{pa 1506 950}%
\special{pa 1536 938}%
\special{pa 1564 924}%
\special{pa 1592 910}%
\special{pa 1622 896}%
\special{pa 1650 880}%
\special{pa 1678 864}%
\special{pa 1706 848}%
\special{pa 1734 830}%
\special{pa 1762 814}%
\special{pa 1790 796}%
\special{pa 1874 742}%
\special{pa 1878 738}%
\special{sp}%
}}%
% SPLINE 2 0 3 0 Black White
% 4 822 1810 1414 1722 1958 1794 1958 1794
% 
{\color[named]{Black}{%
\special{pn 8}%
\special{pa 822 1810}%
\special{pa 854 1804}%
\special{pa 886 1796}%
\special{pa 950 1782}%
\special{pa 980 1776}%
\special{pa 1012 1770}%
\special{pa 1044 1764}%
\special{pa 1076 1758}%
\special{pa 1108 1752}%
\special{pa 1138 1746}%
\special{pa 1170 1742}%
\special{pa 1202 1738}%
\special{pa 1234 1734}%
\special{pa 1298 1728}%
\special{pa 1328 1726}%
\special{pa 1360 1724}%
\special{pa 1392 1722}%
\special{pa 1424 1722}%
\special{pa 1488 1724}%
\special{pa 1520 1726}%
\special{pa 1550 1728}%
\special{pa 1582 1732}%
\special{pa 1614 1736}%
\special{pa 1646 1740}%
\special{pa 1678 1744}%
\special{pa 1742 1754}%
\special{pa 1772 1758}%
\special{pa 1836 1770}%
\special{pa 1868 1776}%
\special{pa 1900 1782}%
\special{pa 1932 1790}%
\special{pa 1958 1794}%
\special{sp}%
}}%
% LINE 2 0 3 0 Black White
% 4 1477 1718 1409 1697 1477 1718 1422 1760
% 
{\color[named]{Black}{%
\special{pn 8}%
\special{pa 1478 1718}%
\special{pa 1410 1698}%
\special{fp}%
\special{pa 1478 1718}%
\special{pa 1422 1760}%
\special{fp}%
}}%
% LINE 2 0 3 0 Black White
% 4 1007 1008 951 966 1007 1008 941 1029
% 
{\color[named]{Black}{%
\special{pn 8}%
\special{pa 1008 1008}%
\special{pa 952 966}%
\special{fp}%
\special{pa 1008 1008}%
\special{pa 942 1030}%
\special{fp}%
}}%
% LINE 2 0 3 0 Black White
% 4 948 1520 1013 1496 948 1520 965 1453
% 
{\color[named]{Black}{%
\special{pn 8}%
\special{pa 948 1520}%
\special{pa 1014 1496}%
\special{fp}%
\special{pa 948 1520}%
\special{pa 966 1454}%
\special{fp}%
}}%
% STR 2 0 3 0 Black White
% 4 1148 1448 1148 1528 2 0 0 0
% $\hat{b}$
\put(11.4800,-15.2800){\makebox(0,0)[lb]{$\hat{b}$}}%
% STR 2 0 3 0 Black White
% 4 1420 1784 1420 1864 2 0 0 0
% $\hat{a}$
\put(14.2000,-18.6400){\makebox(0,0)[lb]{$\hat{a}$}}%
% STR 2 0 3 0 Black White
% 4 854 1074 854 1154 2 0 0 0
% $\hat{c}$
\put(8.5400,-11.5400){\makebox(0,0)[lb]{$\hat{c}$}}%
% LINE 1 0 3 0 Black White
% 2 3302 578 3782 578
% 
{\color[named]{Black}{%
\special{pn 13}%
\special{pa 3302 578}%
\special{pa 3782 578}%
\special{fp}%
}}%
% LINE 1 0 3 0 Black White
% 2 3302 578 2982 258
% 
{\color[named]{Black}{%
\special{pn 13}%
\special{pa 3302 578}%
\special{pa 2982 258}%
\special{fp}%
}}%
% LINE 1 0 3 0 Black White
% 2 3302 578 2982 898
% 
{\color[named]{Black}{%
\special{pn 13}%
\special{pa 3302 578}%
\special{pa 2982 898}%
\special{fp}%
}}%
% LINE 1 0 3 0 Black White
% 2 3782 578 4102 258
% 
{\color[named]{Black}{%
\special{pn 13}%
\special{pa 3782 578}%
\special{pa 4102 258}%
\special{fp}%
}}%
% LINE 1 0 3 0 Black White
% 2 3782 578 4102 898
% 
{\color[named]{Black}{%
\special{pn 13}%
\special{pa 3782 578}%
\special{pa 4102 898}%
\special{fp}%
}}%
% LINE 1 0 3 0 Black White
% 2 3462 1618 3142 1298
% 
{\color[named]{Black}{%
\special{pn 13}%
\special{pa 3462 1618}%
\special{pa 3142 1298}%
\special{fp}%
}}%
% LINE 1 0 3 0 Black White
% 2 3462 1618 3782 1298
% 
{\color[named]{Black}{%
\special{pn 13}%
\special{pa 3462 1618}%
\special{pa 3782 1298}%
\special{fp}%
}}%
% LINE 1 0 3 0 Black White
% 2 3462 1618 3462 2098
% 
{\color[named]{Black}{%
\special{pn 13}%
\special{pa 3462 1618}%
\special{pa 3462 2098}%
\special{fp}%
}}%
% LINE 1 0 3 0 Black White
% 2 3462 2098 3142 2418
% 
{\color[named]{Black}{%
\special{pn 13}%
\special{pa 3462 2098}%
\special{pa 3142 2418}%
\special{fp}%
}}%
% LINE 1 0 3 0 Black White
% 2 3462 2098 3782 2418
% 
{\color[named]{Black}{%
\special{pn 13}%
\special{pa 3462 2098}%
\special{pa 3782 2418}%
\special{fp}%
}}%
% LINE 2 0 3 0 Black White
% 4 3942 738 4022 738 3942 738 3942 818
% 
{\color[named]{Black}{%
\special{pn 8}%
\special{pa 3942 738}%
\special{pa 4022 738}%
\special{fp}%
\special{pa 3942 738}%
\special{pa 3942 818}%
\special{fp}%
}}%
% LINE 2 0 3 0 Black White
% 4 3622 2258 3702 2258 3622 2258 3622 2338
% 
{\color[named]{Black}{%
\special{pn 8}%
\special{pa 3622 2258}%
\special{pa 3702 2258}%
\special{fp}%
\special{pa 3622 2258}%
\special{pa 3622 2338}%
\special{fp}%
}}%
% LINE 2 0 3 0 Black White
% 4 3142 738 3062 738 3142 738 3142 818
% 
{\color[named]{Black}{%
\special{pn 8}%
\special{pa 3142 738}%
\special{pa 3062 738}%
\special{fp}%
\special{pa 3142 738}%
\special{pa 3142 818}%
\special{fp}%
}}%
% LINE 2 0 3 0 Black White
% 4 3302 2258 3222 2258 3302 2258 3302 2338
% 
{\color[named]{Black}{%
\special{pn 8}%
\special{pa 3302 2258}%
\special{pa 3222 2258}%
\special{fp}%
\special{pa 3302 2258}%
\special{pa 3302 2338}%
\special{fp}%
}}%
% LINE 2 0 3 0 Black White
% 4 3142 418 3062 418 3142 418 3142 338
% 
{\color[named]{Black}{%
\special{pn 8}%
\special{pa 3142 418}%
\special{pa 3062 418}%
\special{fp}%
\special{pa 3142 418}%
\special{pa 3142 338}%
\special{fp}%
}}%
% LINE 2 0 3 0 Black White
% 4 3302 1458 3222 1458 3302 1458 3302 1378
% 
{\color[named]{Black}{%
\special{pn 8}%
\special{pa 3302 1458}%
\special{pa 3222 1458}%
\special{fp}%
\special{pa 3302 1458}%
\special{pa 3302 1378}%
\special{fp}%
}}%
% LINE 2 0 3 0 Black White
% 4 3942 418 4022 418 3942 418 3942 338
% 
{\color[named]{Black}{%
\special{pn 8}%
\special{pa 3942 418}%
\special{pa 4022 418}%
\special{fp}%
\special{pa 3942 418}%
\special{pa 3942 338}%
\special{fp}%
}}%
% LINE 2 0 3 0 Black White
% 4 3622 1458 3702 1458 3622 1458 3622 1378
% 
{\color[named]{Black}{%
\special{pn 8}%
\special{pa 3622 1458}%
\special{pa 3702 1458}%
\special{fp}%
\special{pa 3622 1458}%
\special{pa 3622 1378}%
\special{fp}%
}}%
% STR 2 0 3 0 Black White
% 4 4126 922 4126 1002 2 0 0 0
% $a$
\put(41.2600,-10.0200){\makebox(0,0)[lb]{$a$}}%
% STR 2 0 3 0 Black White
% 4 4126 122 4126 202 2 0 0 0
% $b$
\put(41.2600,-2.0200){\makebox(0,0)[lb]{$b$}}%
% STR 2 0 3 0 Black White
% 4 2870 114 2870 194 2 0 0 0
% $c$
\put(28.7000,-1.9400){\makebox(0,0)[lb]{$c$}}%
% STR 2 0 3 0 Black White
% 4 2862 938 2862 1018 2 0 0 0
% $d$
\put(28.6200,-10.1800){\makebox(0,0)[lb]{$d$}}%
% STR 2 0 3 0 Black White
% 4 3486 610 3486 690 2 0 0 0
% $e$
\put(34.8600,-6.9000){\makebox(0,0)[lb]{$e$}}%
% STR 2 0 3 0 Black White
% 4 3822 2458 3822 2538 2 0 0 0
% $a$
\put(38.2200,-25.3800){\makebox(0,0)[lb]{$a$}}%
% STR 2 0 3 0 Black White
% 4 3830 1218 3830 1298 2 0 0 0
% $b$
\put(38.3000,-12.9800){\makebox(0,0)[lb]{$b$}}%
% STR 2 0 3 0 Black White
% 4 3006 1218 3006 1298 2 0 0 0
% $c$
\put(30.0600,-12.9800){\makebox(0,0)[lb]{$c$}}%
% STR 2 0 3 0 Black White
% 4 3014 2474 3014 2554 2 0 0 0
% $d$
\put(30.1400,-25.5400){\makebox(0,0)[lb]{$d$}}%
% STR 2 0 3 0 Black White
% 4 3542 1858 3542 1938 2 0 0 0
% $e^{\prime}$
\put(35.4200,-19.3800){\makebox(0,0)[lb]{$e^{\prime}$}}%
% STR 2 0 3 0 Black White
% 4 2590 554 2590 634 2 0 0 0
% $G$
\put(25.9000,-6.3400){\makebox(0,0)[lb]{$G$}}%
% STR 2 0 3 0 Black White
% 4 2590 1834 2590 1914 2 0 0 0
% $G^{\prime}$
\put(25.9000,-19.1400){\makebox(0,0)[lb]{$G^{\prime}$}}%
% CIRCLE 2 0 3 0 Black White
% 4 3288 586 3192 538 3288 674 3192 538
% 
{\color[named]{Black}{%
\special{pn 8}%
\special{ar 3288 586 108 108  3.6052403 6.2831853}%
\special{ar 3288 586 108 108  0.0000000 1.5707963}%
}}%
% DOT 0 0 3 0 Black White
% 1 3192 538
% 
{\color[named]{Black}{%
\special{pn 4}%
\special{sh 1}%
\special{ar 3192 538 26 26 0  6.28318530717959E+0000}%
}}%
% LINE 2 0 3 0 Black White
% 4 3284 698 3321 660 3284 698 3336 706
% 
{\color[named]{Black}{%
\special{pn 8}%
\special{pa 3284 698}%
\special{pa 3322 660}%
\special{fp}%
\special{pa 3284 698}%
\special{pa 3336 706}%
\special{fp}%
}}%
% CIRCLE 2 0 3 0 Black White
% 4 3808 578 3904 530 3904 530 3808 666
% 
{\color[named]{Black}{%
\special{pn 8}%
\special{ar 3808 578 108 108  1.5707963 5.8195377}%
}}%
% DOT 0 0 3 0 Black White
% 1 3904 530
% 
{\color[named]{Black}{%
\special{pn 4}%
\special{sh 1}%
\special{ar 3904 530 26 26 0  6.28318530717959E+0000}%
}}%
% LINE 2 0 3 0 Black White
% 4 3813 690 3776 652 3813 690 3760 698
% 
{\color[named]{Black}{%
\special{pn 8}%
\special{pa 3814 690}%
\special{pa 3776 652}%
\special{fp}%
\special{pa 3814 690}%
\special{pa 3760 698}%
\special{fp}%
}}%
% CIRCLE 2 0 3 0 Black White
% 4 3464 2114 3560 2066 3560 2066 3464 2202
% 
{\color[named]{Black}{%
\special{pn 8}%
\special{ar 3464 2114 108 108  1.5707963 5.8195377}%
}}%
% DOT 0 0 3 0 Black White
% 1 3560 2066
% 
{\color[named]{Black}{%
\special{pn 4}%
\special{sh 1}%
\special{ar 3560 2066 26 26 0  6.28318530717959E+0000}%
}}%
% LINE 2 0 3 0 Black White
% 4 3469 2226 3432 2188 3469 2226 3416 2234
% 
{\color[named]{Black}{%
\special{pn 8}%
\special{pa 3470 2226}%
\special{pa 3432 2188}%
\special{fp}%
\special{pa 3470 2226}%
\special{pa 3416 2234}%
\special{fp}%
}}%
% CIRCLE 2 0 3 0 Black White
% 4 3464 1620 3560 1668 3464 1532 3560 1668
% 
{\color[named]{Black}{%
\special{pn 8}%
\special{ar 3464 1620 108 108  0.4636476 4.7123890}%
}}%
% DOT 0 0 3 0 Black White
% 1 3560 1668
% 
{\color[named]{Black}{%
\special{pn 4}%
\special{sh 1}%
\special{ar 3560 1668 26 26 0  6.28318530717959E+0000}%
}}%
% LINE 2 0 3 0 Black White
% 4 3469 1508 3432 1545 3469 1508 3416 1500
% 
{\color[named]{Black}{%
\special{pn 8}%
\special{pa 3470 1508}%
\special{pa 3432 1546}%
\special{fp}%
\special{pa 3470 1508}%
\special{pa 3416 1500}%
\special{fp}%
}}%
% STR 2 0 3 0 Black White
% 4 3540 1560 3540 1660 2 0 0 0
% $v^{\prime}_1$
\put(35.4000,-16.6000){\makebox(0,0)[lb]{$v^{\prime}_1$}}%
% STR 2 0 3 0 Black White
% 4 3550 2130 3550 2230 2 0 0 0
% $v^{\prime}_2$
\put(35.5000,-22.3000){\makebox(0,0)[lb]{$v^{\prime}_2$}}%
% STR 2 0 3 0 Black White
% 4 3870 580 3870 680 2 0 0 0
% $v_1$
\put(38.7000,-6.8000){\makebox(0,0)[lb]{$v_1$}}%
% STR 2 0 3 0 Black White
% 4 3070 550 3070 650 2 0 0 0
% $v_2$
\put(30.7000,-6.5000){\makebox(0,0)[lb]{$v_2$}}%
\end{picture}%

%% file: 6cases.tex
%WinTpicVersion4.31b
{\unitlength 0.1in%
\begin{picture}( 42.4900, 41.7500)(  1.6000,-42.2500)%
% LINE 1 0 3 0 Black White  
% 2 468 844 804 844
% 
\special{pn 13}%
\special{pa 468 844}%
\special{pa 804 844}%
\special{fp}%
% LINE 1 0 3 0 Black White  
% 2 468 844 244 620
% 
\special{pn 13}%
\special{pa 468 844}%
\special{pa 244 620}%
\special{fp}%
% LINE 1 0 3 0 Black White  
% 2 468 844 244 1068
% 
\special{pn 13}%
\special{pa 468 844}%
\special{pa 244 1068}%
\special{fp}%
% LINE 1 0 3 0 Black White  
% 2 804 844 1028 620
% 
\special{pn 13}%
\special{pa 804 844}%
\special{pa 1028 620}%
\special{fp}%
% LINE 1 0 3 0 Black White  
% 2 804 844 1028 1068
% 
\special{pn 13}%
\special{pa 804 844}%
\special{pa 1028 1068}%
\special{fp}%
% LINE 1 0 3 0 Black White  
% 2 1771 665 1547 441
% 
\special{pn 13}%
\special{pa 1771 665}%
\special{pa 1547 441}%
\special{fp}%
% LINE 1 0 3 0 Black White  
% 2 1771 665 1995 441
% 
\special{pn 13}%
\special{pa 1771 665}%
\special{pa 1995 441}%
\special{fp}%
% LINE 1 0 3 0 Black White  
% 2 1771 665 1771 1001
% 
\special{pn 13}%
\special{pa 1771 665}%
\special{pa 1771 1001}%
\special{fp}%
% LINE 1 0 3 0 Black White  
% 2 1771 1001 1547 1225
% 
\special{pn 13}%
\special{pa 1771 1001}%
\special{pa 1547 1225}%
\special{fp}%
% LINE 1 0 3 0 Black White  
% 2 1771 1001 1995 1225
% 
\special{pn 13}%
\special{pa 1771 1001}%
\special{pa 1995 1225}%
\special{fp}%
% LINE 2 0 3 0 Black White  
% 4 916 956 972 956 916 956 916 1012
% 
\special{pn 8}%
\special{pa 916 956}%
\special{pa 972 956}%
\special{fp}%
\special{pa 916 956}%
\special{pa 916 1012}%
\special{fp}%
% LINE 2 0 3 0 Black White  
% 4 1883 1113 1939 1113 1883 1113 1883 1169
% 
\special{pn 8}%
\special{pa 1883 1113}%
\special{pa 1939 1113}%
\special{fp}%
\special{pa 1883 1113}%
\special{pa 1883 1169}%
\special{fp}%
% LINE 2 0 3 0 Black White  
% 4 356 956 300 956 356 956 356 1012
% 
\special{pn 8}%
\special{pa 356 956}%
\special{pa 300 956}%
\special{fp}%
\special{pa 356 956}%
\special{pa 356 1012}%
\special{fp}%
% LINE 2 0 3 0 Black White  
% 4 1659 1113 1603 1113 1659 1113 1659 1169
% 
\special{pn 8}%
\special{pa 1659 1113}%
\special{pa 1603 1113}%
\special{fp}%
\special{pa 1659 1113}%
\special{pa 1659 1169}%
\special{fp}%
% LINE 2 0 3 0 Black White  
% 4 356 732 300 732 356 732 356 676
% 
\special{pn 8}%
\special{pa 356 732}%
\special{pa 300 732}%
\special{fp}%
\special{pa 356 732}%
\special{pa 356 676}%
\special{fp}%
% LINE 2 0 3 0 Black White  
% 4 1659 553 1603 553 1659 553 1659 497
% 
\special{pn 8}%
\special{pa 1659 553}%
\special{pa 1603 553}%
\special{fp}%
\special{pa 1659 553}%
\special{pa 1659 497}%
\special{fp}%
% LINE 2 0 3 0 Black White  
% 4 916 732 972 732 916 732 916 676
% 
\special{pn 8}%
\special{pa 916 732}%
\special{pa 972 732}%
\special{fp}%
\special{pa 916 732}%
\special{pa 916 676}%
\special{fp}%
% LINE 2 0 3 0 Black White  
% 4 1883 553 1939 553 1883 553 1883 497
% 
\special{pn 8}%
\special{pa 1883 553}%
\special{pa 1939 553}%
\special{fp}%
\special{pa 1883 553}%
\special{pa 1883 497}%
\special{fp}%
% STR 2 0 3 0 Black White  
% 4 1045 1085 1045 1141 2 0 0 0
% $a$
\put(10.4500,-11.4100){\makebox(0,0)[lb]{$a$}}%
% STR 2 0 3 0 Black White  
% 4 1045 525 1045 581 2 0 0 0
% $b$
\put(10.4500,-5.8100){\makebox(0,0)[lb]{$b$}}%
% STR 2 0 3 0 Black White  
% 4 166 519 166 575 2 0 0 0
% $c$
\put(1.6600,-5.7500){\makebox(0,0)[lb]{$c$}}%
% STR 2 0 3 0 Black White  
% 4 160 1096 160 1152 2 0 0 0
% $d$
\put(1.6000,-11.5200){\makebox(0,0)[lb]{$d$}}%
% STR 2 0 3 0 Black White  
% 4 597 882 597 938 2 0 0 0
% $e$
\put(5.9700,-9.3800){\makebox(0,0)[lb]{$e$}}%
% STR 2 0 3 0 Black White  
% 4 2023 1253 2023 1309 2 0 0 0
% $a$
\put(20.2300,-13.0900){\makebox(0,0)[lb]{$a$}}%
% STR 2 0 3 0 Black White  
% 4 2029 385 2029 441 2 0 0 0
% $b$
\put(20.2900,-4.4100){\makebox(0,0)[lb]{$b$}}%
% STR 2 0 3 0 Black White  
% 4 1452 385 1452 441 2 0 0 0
% $c$
\put(14.5200,-4.4100){\makebox(0,0)[lb]{$c$}}%
% STR 2 0 3 0 Black White  
% 4 1458 1264 1458 1320 2 0 0 0
% $d$
\put(14.5800,-13.2000){\makebox(0,0)[lb]{$d$}}%
% STR 2 0 3 0 Black White  
% 4 1827 833 1827 889 2 0 0 0
% $e^{\prime}$
\put(18.2700,-8.8900){\makebox(0,0)[lb]{$e^{\prime}$}}%
% CIRCLE 2 0 3 0 Black White  
% 4 4160 3664 4230 3690 4153 3603 4230 3690
% 
\special{pn 8}%
\special{ar 4160 3664 75 75  0.3556359  4.5981347}%
% DOT 0 0 3 0 Black White  
% 1 4230 3690
% 
\special{pn 4}%
\special{sh 1}%
\special{ar 4230 3690 28 28 0  6.28318530717959E+0000}%
% LINE 2 0 3 0 Black White  
% 4 4154 3586 4131 3615 4154 3586 4117 3585
% 
\special{pn 8}%
\special{pa 4154 3586}%
\special{pa 4131 3615}%
\special{fp}%
\special{pa 4154 3586}%
\special{pa 4117 3585}%
\special{fp}%
% LINE 1 0 3 0 Black White  
% 2 2848 844 3184 844
% 
\special{pn 13}%
\special{pa 2848 844}%
\special{pa 3184 844}%
\special{fp}%
% LINE 1 0 3 0 Black White  
% 2 2848 844 2624 620
% 
\special{pn 13}%
\special{pa 2848 844}%
\special{pa 2624 620}%
\special{fp}%
% LINE 1 0 3 0 Black White  
% 2 2848 844 2624 1068
% 
\special{pn 13}%
\special{pa 2848 844}%
\special{pa 2624 1068}%
\special{fp}%
% LINE 1 0 3 0 Black White  
% 2 3184 844 3408 620
% 
\special{pn 13}%
\special{pa 3184 844}%
\special{pa 3408 620}%
\special{fp}%
% LINE 1 0 3 0 Black White  
% 2 3184 844 3408 1068
% 
\special{pn 13}%
\special{pa 3184 844}%
\special{pa 3408 1068}%
\special{fp}%
% LINE 1 0 3 0 Black White  
% 2 4151 665 3927 441
% 
\special{pn 13}%
\special{pa 4151 665}%
\special{pa 3927 441}%
\special{fp}%
% LINE 1 0 3 0 Black White  
% 2 4151 665 4375 441
% 
\special{pn 13}%
\special{pa 4151 665}%
\special{pa 4375 441}%
\special{fp}%
% LINE 1 0 3 0 Black White  
% 2 4151 665 4151 1001
% 
\special{pn 13}%
\special{pa 4151 665}%
\special{pa 4151 1001}%
\special{fp}%
% LINE 1 0 3 0 Black White  
% 2 4151 1001 3927 1225
% 
\special{pn 13}%
\special{pa 4151 1001}%
\special{pa 3927 1225}%
\special{fp}%
% LINE 1 0 3 0 Black White  
% 2 4151 1001 4375 1225
% 
\special{pn 13}%
\special{pa 4151 1001}%
\special{pa 4375 1225}%
\special{fp}%
% LINE 2 0 3 0 Black White  
% 4 3296 956 3352 956 3296 956 3296 1012
% 
\special{pn 8}%
\special{pa 3296 956}%
\special{pa 3352 956}%
\special{fp}%
\special{pa 3296 956}%
\special{pa 3296 1012}%
\special{fp}%
% LINE 2 0 3 0 Black White  
% 4 4263 1113 4319 1113 4263 1113 4263 1169
% 
\special{pn 8}%
\special{pa 4263 1113}%
\special{pa 4319 1113}%
\special{fp}%
\special{pa 4263 1113}%
\special{pa 4263 1169}%
\special{fp}%
% LINE 2 0 3 0 Black White  
% 4 2736 956 2680 956 2736 956 2736 1012
% 
\special{pn 8}%
\special{pa 2736 956}%
\special{pa 2680 956}%
\special{fp}%
\special{pa 2736 956}%
\special{pa 2736 1012}%
\special{fp}%
% LINE 2 0 3 0 Black White  
% 4 4039 1113 3983 1113 4039 1113 4039 1169
% 
\special{pn 8}%
\special{pa 4039 1113}%
\special{pa 3983 1113}%
\special{fp}%
\special{pa 4039 1113}%
\special{pa 4039 1169}%
\special{fp}%
% LINE 2 0 3 0 Black White  
% 4 2736 732 2680 732 2736 732 2736 676
% 
\special{pn 8}%
\special{pa 2736 732}%
\special{pa 2680 732}%
\special{fp}%
\special{pa 2736 732}%
\special{pa 2736 676}%
\special{fp}%
% LINE 2 0 3 0 Black White  
% 4 4039 553 3983 553 4039 553 4039 497
% 
\special{pn 8}%
\special{pa 4039 553}%
\special{pa 3983 553}%
\special{fp}%
\special{pa 4039 553}%
\special{pa 4039 497}%
\special{fp}%
% LINE 2 0 3 0 Black White  
% 4 3296 732 3352 732 3296 732 3296 676
% 
\special{pn 8}%
\special{pa 3296 732}%
\special{pa 3352 732}%
\special{fp}%
\special{pa 3296 732}%
\special{pa 3296 676}%
\special{fp}%
% LINE 2 0 3 0 Black White  
% 4 4263 553 4319 553 4263 553 4263 497
% 
\special{pn 8}%
\special{pa 4263 553}%
\special{pa 4319 553}%
\special{fp}%
\special{pa 4263 553}%
\special{pa 4263 497}%
\special{fp}%
% STR 2 0 3 0 Black White  
% 4 3425 1085 3425 1141 2 0 0 0
% $a$
\put(34.2500,-11.4100){\makebox(0,0)[lb]{$a$}}%
% STR 2 0 3 0 Black White  
% 4 3425 525 3425 581 2 0 0 0
% $b$
\put(34.2500,-5.8100){\makebox(0,0)[lb]{$b$}}%
% STR 2 0 3 0 Black White  
% 4 2546 519 2546 575 2 0 0 0
% $c$
\put(25.4600,-5.7500){\makebox(0,0)[lb]{$c$}}%
% STR 2 0 3 0 Black White  
% 4 2540 1096 2540 1152 2 0 0 0
% $d$
\put(25.4000,-11.5200){\makebox(0,0)[lb]{$d$}}%
% STR 2 0 3 0 Black White  
% 4 2977 882 2977 938 2 0 0 0
% $e$
\put(29.7700,-9.3800){\makebox(0,0)[lb]{$e$}}%
% STR 2 0 3 0 Black White  
% 4 4403 1253 4403 1309 2 0 0 0
% $a$
\put(44.0300,-13.0900){\makebox(0,0)[lb]{$a$}}%
% STR 2 0 3 0 Black White  
% 4 4409 385 4409 441 2 0 0 0
% $b$
\put(44.0900,-4.4100){\makebox(0,0)[lb]{$b$}}%
% STR 2 0 3 0 Black White  
% 4 3832 385 3832 441 2 0 0 0
% $c$
\put(38.3200,-4.4100){\makebox(0,0)[lb]{$c$}}%
% STR 2 0 3 0 Black White  
% 4 3838 1264 3838 1320 2 0 0 0
% $d$
\put(38.3800,-13.2000){\makebox(0,0)[lb]{$d$}}%
% STR 2 0 3 0 Black White  
% 4 4207 833 4207 889 2 0 0 0
% $e^{\prime}$
\put(42.0700,-8.8900){\makebox(0,0)[lb]{$e^{\prime}$}}%
% LINE 1 0 3 0 Black White  
% 2 468 2344 804 2344
% 
\special{pn 13}%
\special{pa 468 2344}%
\special{pa 804 2344}%
\special{fp}%
% LINE 1 0 3 0 Black White  
% 2 468 2344 244 2120
% 
\special{pn 13}%
\special{pa 468 2344}%
\special{pa 244 2120}%
\special{fp}%
% LINE 1 0 3 0 Black White  
% 2 468 2344 244 2568
% 
\special{pn 13}%
\special{pa 468 2344}%
\special{pa 244 2568}%
\special{fp}%
% LINE 1 0 3 0 Black White  
% 2 804 2344 1028 2120
% 
\special{pn 13}%
\special{pa 804 2344}%
\special{pa 1028 2120}%
\special{fp}%
% LINE 1 0 3 0 Black White  
% 2 804 2344 1028 2568
% 
\special{pn 13}%
\special{pa 804 2344}%
\special{pa 1028 2568}%
\special{fp}%
% LINE 1 0 3 0 Black White  
% 2 1771 2165 1547 1941
% 
\special{pn 13}%
\special{pa 1771 2165}%
\special{pa 1547 1941}%
\special{fp}%
% LINE 1 0 3 0 Black White  
% 2 1771 2165 1995 1941
% 
\special{pn 13}%
\special{pa 1771 2165}%
\special{pa 1995 1941}%
\special{fp}%
% LINE 1 0 3 0 Black White  
% 2 1771 2165 1771 2501
% 
\special{pn 13}%
\special{pa 1771 2165}%
\special{pa 1771 2501}%
\special{fp}%
% LINE 1 0 3 0 Black White  
% 2 1771 2501 1547 2725
% 
\special{pn 13}%
\special{pa 1771 2501}%
\special{pa 1547 2725}%
\special{fp}%
% LINE 1 0 3 0 Black White  
% 2 1771 2501 1995 2725
% 
\special{pn 13}%
\special{pa 1771 2501}%
\special{pa 1995 2725}%
\special{fp}%
% LINE 2 0 3 0 Black White  
% 4 916 2456 972 2456 916 2456 916 2512
% 
\special{pn 8}%
\special{pa 916 2456}%
\special{pa 972 2456}%
\special{fp}%
\special{pa 916 2456}%
\special{pa 916 2512}%
\special{fp}%
% LINE 2 0 3 0 Black White  
% 4 1883 2613 1939 2613 1883 2613 1883 2669
% 
\special{pn 8}%
\special{pa 1883 2613}%
\special{pa 1939 2613}%
\special{fp}%
\special{pa 1883 2613}%
\special{pa 1883 2669}%
\special{fp}%
% LINE 2 0 3 0 Black White  
% 4 356 2456 300 2456 356 2456 356 2512
% 
\special{pn 8}%
\special{pa 356 2456}%
\special{pa 300 2456}%
\special{fp}%
\special{pa 356 2456}%
\special{pa 356 2512}%
\special{fp}%
% LINE 2 0 3 0 Black White  
% 4 1659 2613 1603 2613 1659 2613 1659 2669
% 
\special{pn 8}%
\special{pa 1659 2613}%
\special{pa 1603 2613}%
\special{fp}%
\special{pa 1659 2613}%
\special{pa 1659 2669}%
\special{fp}%
% LINE 2 0 3 0 Black White  
% 4 356 2232 300 2232 356 2232 356 2176
% 
\special{pn 8}%
\special{pa 356 2232}%
\special{pa 300 2232}%
\special{fp}%
\special{pa 356 2232}%
\special{pa 356 2176}%
\special{fp}%
% LINE 2 0 3 0 Black White  
% 4 1659 2053 1603 2053 1659 2053 1659 1997
% 
\special{pn 8}%
\special{pa 1659 2053}%
\special{pa 1603 2053}%
\special{fp}%
\special{pa 1659 2053}%
\special{pa 1659 1997}%
\special{fp}%
% LINE 2 0 3 0 Black White  
% 4 916 2232 972 2232 916 2232 916 2176
% 
\special{pn 8}%
\special{pa 916 2232}%
\special{pa 972 2232}%
\special{fp}%
\special{pa 916 2232}%
\special{pa 916 2176}%
\special{fp}%
% LINE 2 0 3 0 Black White  
% 4 1883 2053 1939 2053 1883 2053 1883 1997
% 
\special{pn 8}%
\special{pa 1883 2053}%
\special{pa 1939 2053}%
\special{fp}%
\special{pa 1883 2053}%
\special{pa 1883 1997}%
\special{fp}%
% STR 2 0 3 0 Black White  
% 4 1045 2585 1045 2641 2 0 0 0
% $a$
\put(10.4500,-26.4100){\makebox(0,0)[lb]{$a$}}%
% STR 2 0 3 0 Black White  
% 4 1045 2025 1045 2081 2 0 0 0
% $b$
\put(10.4500,-20.8100){\makebox(0,0)[lb]{$b$}}%
% STR 2 0 3 0 Black White  
% 4 166 2019 166 2075 2 0 0 0
% $c$
\put(1.6600,-20.7500){\makebox(0,0)[lb]{$c$}}%
% STR 2 0 3 0 Black White  
% 4 160 2596 160 2652 2 0 0 0
% $d$
\put(1.6000,-26.5200){\makebox(0,0)[lb]{$d$}}%
% STR 2 0 3 0 Black White  
% 4 597 2382 597 2438 2 0 0 0
% $e$
\put(5.9700,-24.3800){\makebox(0,0)[lb]{$e$}}%
% STR 2 0 3 0 Black White  
% 4 2023 2753 2023 2809 2 0 0 0
% $a$
\put(20.2300,-28.0900){\makebox(0,0)[lb]{$a$}}%
% STR 2 0 3 0 Black White  
% 4 2029 1885 2029 1941 2 0 0 0
% $b$
\put(20.2900,-19.4100){\makebox(0,0)[lb]{$b$}}%
% STR 2 0 3 0 Black White  
% 4 1452 1885 1452 1941 2 0 0 0
% $c$
\put(14.5200,-19.4100){\makebox(0,0)[lb]{$c$}}%
% STR 2 0 3 0 Black White  
% 4 1458 2764 1458 2820 2 0 0 0
% $d$
\put(14.5800,-28.2000){\makebox(0,0)[lb]{$d$}}%
% STR 2 0 3 0 Black White  
% 4 1827 2333 1827 2389 2 0 0 0
% $e^{\prime}$
\put(18.2700,-23.8900){\makebox(0,0)[lb]{$e^{\prime}$}}%
% LINE 1 0 3 0 Black White  
% 2 2848 2344 3184 2344
% 
\special{pn 13}%
\special{pa 2848 2344}%
\special{pa 3184 2344}%
\special{fp}%
% LINE 1 0 3 0 Black White  
% 2 2848 2344 2624 2120
% 
\special{pn 13}%
\special{pa 2848 2344}%
\special{pa 2624 2120}%
\special{fp}%
% LINE 1 0 3 0 Black White  
% 2 2848 2344 2624 2568
% 
\special{pn 13}%
\special{pa 2848 2344}%
\special{pa 2624 2568}%
\special{fp}%
% LINE 1 0 3 0 Black White  
% 2 3184 2344 3408 2120
% 
\special{pn 13}%
\special{pa 3184 2344}%
\special{pa 3408 2120}%
\special{fp}%
% LINE 1 0 3 0 Black White  
% 2 3184 2344 3408 2568
% 
\special{pn 13}%
\special{pa 3184 2344}%
\special{pa 3408 2568}%
\special{fp}%
% LINE 1 0 3 0 Black White  
% 2 4151 2165 3927 1941
% 
\special{pn 13}%
\special{pa 4151 2165}%
\special{pa 3927 1941}%
\special{fp}%
% LINE 1 0 3 0 Black White  
% 2 4151 2165 4375 1941
% 
\special{pn 13}%
\special{pa 4151 2165}%
\special{pa 4375 1941}%
\special{fp}%
% LINE 1 0 3 0 Black White  
% 2 4151 2165 4151 2501
% 
\special{pn 13}%
\special{pa 4151 2165}%
\special{pa 4151 2501}%
\special{fp}%
% LINE 1 0 3 0 Black White  
% 2 4151 2501 3927 2725
% 
\special{pn 13}%
\special{pa 4151 2501}%
\special{pa 3927 2725}%
\special{fp}%
% LINE 1 0 3 0 Black White  
% 2 4151 2501 4375 2725
% 
\special{pn 13}%
\special{pa 4151 2501}%
\special{pa 4375 2725}%
\special{fp}%
% LINE 2 0 3 0 Black White  
% 4 3296 2456 3352 2456 3296 2456 3296 2512
% 
\special{pn 8}%
\special{pa 3296 2456}%
\special{pa 3352 2456}%
\special{fp}%
\special{pa 3296 2456}%
\special{pa 3296 2512}%
\special{fp}%
% LINE 2 0 3 0 Black White  
% 4 4263 2613 4319 2613 4263 2613 4263 2669
% 
\special{pn 8}%
\special{pa 4263 2613}%
\special{pa 4319 2613}%
\special{fp}%
\special{pa 4263 2613}%
\special{pa 4263 2669}%
\special{fp}%
% LINE 2 0 3 0 Black White  
% 4 2736 2456 2680 2456 2736 2456 2736 2512
% 
\special{pn 8}%
\special{pa 2736 2456}%
\special{pa 2680 2456}%
\special{fp}%
\special{pa 2736 2456}%
\special{pa 2736 2512}%
\special{fp}%
% LINE 2 0 3 0 Black White  
% 4 4039 2613 3983 2613 4039 2613 4039 2669
% 
\special{pn 8}%
\special{pa 4039 2613}%
\special{pa 3983 2613}%
\special{fp}%
\special{pa 4039 2613}%
\special{pa 4039 2669}%
\special{fp}%
% LINE 2 0 3 0 Black White  
% 4 2736 2232 2680 2232 2736 2232 2736 2176
% 
\special{pn 8}%
\special{pa 2736 2232}%
\special{pa 2680 2232}%
\special{fp}%
\special{pa 2736 2232}%
\special{pa 2736 2176}%
\special{fp}%
% LINE 2 0 3 0 Black White  
% 4 4039 2053 3983 2053 4039 2053 4039 1997
% 
\special{pn 8}%
\special{pa 4039 2053}%
\special{pa 3983 2053}%
\special{fp}%
\special{pa 4039 2053}%
\special{pa 4039 1997}%
\special{fp}%
% LINE 2 0 3 0 Black White  
% 4 3296 2232 3352 2232 3296 2232 3296 2176
% 
\special{pn 8}%
\special{pa 3296 2232}%
\special{pa 3352 2232}%
\special{fp}%
\special{pa 3296 2232}%
\special{pa 3296 2176}%
\special{fp}%
% LINE 2 0 3 0 Black White  
% 4 4263 2053 4319 2053 4263 2053 4263 1997
% 
\special{pn 8}%
\special{pa 4263 2053}%
\special{pa 4319 2053}%
\special{fp}%
\special{pa 4263 2053}%
\special{pa 4263 1997}%
\special{fp}%
% STR 2 0 3 0 Black White  
% 4 3425 2585 3425 2641 2 0 0 0
% $a$
\put(34.2500,-26.4100){\makebox(0,0)[lb]{$a$}}%
% STR 2 0 3 0 Black White  
% 4 3425 2025 3425 2081 2 0 0 0
% $b$
\put(34.2500,-20.8100){\makebox(0,0)[lb]{$b$}}%
% STR 2 0 3 0 Black White  
% 4 2546 2019 2546 2075 2 0 0 0
% $c$
\put(25.4600,-20.7500){\makebox(0,0)[lb]{$c$}}%
% STR 2 0 3 0 Black White  
% 4 2540 2596 2540 2652 2 0 0 0
% $d$
\put(25.4000,-26.5200){\makebox(0,0)[lb]{$d$}}%
% STR 2 0 3 0 Black White  
% 4 2977 2382 2977 2438 2 0 0 0
% $e$
\put(29.7700,-24.3800){\makebox(0,0)[lb]{$e$}}%
% STR 2 0 3 0 Black White  
% 4 4403 2753 4403 2809 2 0 0 0
% $a$
\put(44.0300,-28.0900){\makebox(0,0)[lb]{$a$}}%
% STR 2 0 3 0 Black White  
% 4 4409 1885 4409 1941 2 0 0 0
% $b$
\put(44.0900,-19.4100){\makebox(0,0)[lb]{$b$}}%
% STR 2 0 3 0 Black White  
% 4 3832 1885 3832 1941 2 0 0 0
% $c$
\put(38.3200,-19.4100){\makebox(0,0)[lb]{$c$}}%
% STR 2 0 3 0 Black White  
% 4 3838 2764 3838 2820 2 0 0 0
% $d$
\put(38.3800,-28.2000){\makebox(0,0)[lb]{$d$}}%
% STR 2 0 3 0 Black White  
% 4 4207 2333 4207 2389 2 0 0 0
% $e^{\prime}$
\put(42.0700,-23.8900){\makebox(0,0)[lb]{$e^{\prime}$}}%
% LINE 1 0 3 0 Black White  
% 2 468 3844 804 3844
% 
\special{pn 13}%
\special{pa 468 3844}%
\special{pa 804 3844}%
\special{fp}%
% LINE 1 0 3 0 Black White  
% 2 468 3844 244 3620
% 
\special{pn 13}%
\special{pa 468 3844}%
\special{pa 244 3620}%
\special{fp}%
% LINE 1 0 3 0 Black White  
% 2 468 3844 244 4068
% 
\special{pn 13}%
\special{pa 468 3844}%
\special{pa 244 4068}%
\special{fp}%
% LINE 1 0 3 0 Black White  
% 2 804 3844 1028 3620
% 
\special{pn 13}%
\special{pa 804 3844}%
\special{pa 1028 3620}%
\special{fp}%
% LINE 1 0 3 0 Black White  
% 2 804 3844 1028 4068
% 
\special{pn 13}%
\special{pa 804 3844}%
\special{pa 1028 4068}%
\special{fp}%
% LINE 1 0 3 0 Black White  
% 2 1771 3665 1547 3441
% 
\special{pn 13}%
\special{pa 1771 3665}%
\special{pa 1547 3441}%
\special{fp}%
% LINE 1 0 3 0 Black White  
% 2 1771 3665 1995 3441
% 
\special{pn 13}%
\special{pa 1771 3665}%
\special{pa 1995 3441}%
\special{fp}%
% LINE 1 0 3 0 Black White  
% 2 1771 3665 1771 4001
% 
\special{pn 13}%
\special{pa 1771 3665}%
\special{pa 1771 4001}%
\special{fp}%
% LINE 1 0 3 0 Black White  
% 2 1771 4001 1547 4225
% 
\special{pn 13}%
\special{pa 1771 4001}%
\special{pa 1547 4225}%
\special{fp}%
% LINE 1 0 3 0 Black White  
% 2 1771 4001 1995 4225
% 
\special{pn 13}%
\special{pa 1771 4001}%
\special{pa 1995 4225}%
\special{fp}%
% LINE 2 0 3 0 Black White  
% 4 916 3956 972 3956 916 3956 916 4012
% 
\special{pn 8}%
\special{pa 916 3956}%
\special{pa 972 3956}%
\special{fp}%
\special{pa 916 3956}%
\special{pa 916 4012}%
\special{fp}%
% LINE 2 0 3 0 Black White  
% 4 1883 4113 1939 4113 1883 4113 1883 4169
% 
\special{pn 8}%
\special{pa 1883 4113}%
\special{pa 1939 4113}%
\special{fp}%
\special{pa 1883 4113}%
\special{pa 1883 4169}%
\special{fp}%
% LINE 2 0 3 0 Black White  
% 4 356 3956 300 3956 356 3956 356 4012
% 
\special{pn 8}%
\special{pa 356 3956}%
\special{pa 300 3956}%
\special{fp}%
\special{pa 356 3956}%
\special{pa 356 4012}%
\special{fp}%
% LINE 2 0 3 0 Black White  
% 4 1659 4113 1603 4113 1659 4113 1659 4169
% 
\special{pn 8}%
\special{pa 1659 4113}%
\special{pa 1603 4113}%
\special{fp}%
\special{pa 1659 4113}%
\special{pa 1659 4169}%
\special{fp}%
% LINE 2 0 3 0 Black White  
% 4 356 3732 300 3732 356 3732 356 3676
% 
\special{pn 8}%
\special{pa 356 3732}%
\special{pa 300 3732}%
\special{fp}%
\special{pa 356 3732}%
\special{pa 356 3676}%
\special{fp}%
% LINE 2 0 3 0 Black White  
% 4 1659 3553 1603 3553 1659 3553 1659 3497
% 
\special{pn 8}%
\special{pa 1659 3553}%
\special{pa 1603 3553}%
\special{fp}%
\special{pa 1659 3553}%
\special{pa 1659 3497}%
\special{fp}%
% LINE 2 0 3 0 Black White  
% 4 916 3732 972 3732 916 3732 916 3676
% 
\special{pn 8}%
\special{pa 916 3732}%
\special{pa 972 3732}%
\special{fp}%
\special{pa 916 3732}%
\special{pa 916 3676}%
\special{fp}%
% LINE 2 0 3 0 Black White  
% 4 1883 3553 1939 3553 1883 3553 1883 3497
% 
\special{pn 8}%
\special{pa 1883 3553}%
\special{pa 1939 3553}%
\special{fp}%
\special{pa 1883 3553}%
\special{pa 1883 3497}%
\special{fp}%
% STR 2 0 3 0 Black White  
% 4 1045 4085 1045 4141 2 0 0 0
% $a$
\put(10.4500,-41.4100){\makebox(0,0)[lb]{$a$}}%
% STR 2 0 3 0 Black White  
% 4 1045 3525 1045 3581 2 0 0 0
% $b$
\put(10.4500,-35.8100){\makebox(0,0)[lb]{$b$}}%
% STR 2 0 3 0 Black White  
% 4 166 3519 166 3575 2 0 0 0
% $c$
\put(1.6600,-35.7500){\makebox(0,0)[lb]{$c$}}%
% STR 2 0 3 0 Black White  
% 4 160 4096 160 4152 2 0 0 0
% $d$
\put(1.6000,-41.5200){\makebox(0,0)[lb]{$d$}}%
% STR 2 0 3 0 Black White  
% 4 597 3882 597 3938 2 0 0 0
% $e$
\put(5.9700,-39.3800){\makebox(0,0)[lb]{$e$}}%
% STR 2 0 3 0 Black White  
% 4 2023 4253 2023 4309 2 0 0 0
% $a$
\put(20.2300,-43.0900){\makebox(0,0)[lb]{$a$}}%
% STR 2 0 3 0 Black White  
% 4 2029 3385 2029 3441 2 0 0 0
% $b$
\put(20.2900,-34.4100){\makebox(0,0)[lb]{$b$}}%
% STR 2 0 3 0 Black White  
% 4 1452 3385 1452 3441 2 0 0 0
% $c$
\put(14.5200,-34.4100){\makebox(0,0)[lb]{$c$}}%
% STR 2 0 3 0 Black White  
% 4 1458 4264 1458 4320 2 0 0 0
% $d$
\put(14.5800,-43.2000){\makebox(0,0)[lb]{$d$}}%
% STR 2 0 3 0 Black White  
% 4 1827 3833 1827 3889 2 0 0 0
% $e^{\prime}$
\put(18.2700,-38.8900){\makebox(0,0)[lb]{$e^{\prime}$}}%
% LINE 1 0 3 0 Black White  
% 2 2848 3844 3184 3844
% 
\special{pn 13}%
\special{pa 2848 3844}%
\special{pa 3184 3844}%
\special{fp}%
% LINE 1 0 3 0 Black White  
% 2 2848 3844 2624 3620
% 
\special{pn 13}%
\special{pa 2848 3844}%
\special{pa 2624 3620}%
\special{fp}%
% LINE 1 0 3 0 Black White  
% 2 2848 3844 2624 4068
% 
\special{pn 13}%
\special{pa 2848 3844}%
\special{pa 2624 4068}%
\special{fp}%
% LINE 1 0 3 0 Black White  
% 2 3184 3844 3408 3620
% 
\special{pn 13}%
\special{pa 3184 3844}%
\special{pa 3408 3620}%
\special{fp}%
% LINE 1 0 3 0 Black White  
% 2 3184 3844 3408 4068
% 
\special{pn 13}%
\special{pa 3184 3844}%
\special{pa 3408 4068}%
\special{fp}%
% LINE 1 0 3 0 Black White  
% 2 4151 3665 3927 3441
% 
\special{pn 13}%
\special{pa 4151 3665}%
\special{pa 3927 3441}%
\special{fp}%
% LINE 1 0 3 0 Black White  
% 2 4151 3665 4375 3441
% 
\special{pn 13}%
\special{pa 4151 3665}%
\special{pa 4375 3441}%
\special{fp}%
% LINE 1 0 3 0 Black White  
% 2 4151 3665 4151 4001
% 
\special{pn 13}%
\special{pa 4151 3665}%
\special{pa 4151 4001}%
\special{fp}%
% LINE 1 0 3 0 Black White  
% 2 4151 4001 3927 4225
% 
\special{pn 13}%
\special{pa 4151 4001}%
\special{pa 3927 4225}%
\special{fp}%
% LINE 1 0 3 0 Black White  
% 2 4151 4001 4375 4225
% 
\special{pn 13}%
\special{pa 4151 4001}%
\special{pa 4375 4225}%
\special{fp}%
% LINE 2 0 3 0 Black White  
% 4 3296 3956 3352 3956 3296 3956 3296 4012
% 
\special{pn 8}%
\special{pa 3296 3956}%
\special{pa 3352 3956}%
\special{fp}%
\special{pa 3296 3956}%
\special{pa 3296 4012}%
\special{fp}%
% LINE 2 0 3 0 Black White  
% 4 4263 4113 4319 4113 4263 4113 4263 4169
% 
\special{pn 8}%
\special{pa 4263 4113}%
\special{pa 4319 4113}%
\special{fp}%
\special{pa 4263 4113}%
\special{pa 4263 4169}%
\special{fp}%
% LINE 2 0 3 0 Black White  
% 4 2736 3956 2680 3956 2736 3956 2736 4012
% 
\special{pn 8}%
\special{pa 2736 3956}%
\special{pa 2680 3956}%
\special{fp}%
\special{pa 2736 3956}%
\special{pa 2736 4012}%
\special{fp}%
% LINE 2 0 3 0 Black White  
% 4 4039 4113 3983 4113 4039 4113 4039 4169
% 
\special{pn 8}%
\special{pa 4039 4113}%
\special{pa 3983 4113}%
\special{fp}%
\special{pa 4039 4113}%
\special{pa 4039 4169}%
\special{fp}%
% LINE 2 0 3 0 Black White  
% 4 2736 3732 2680 3732 2736 3732 2736 3676
% 
\special{pn 8}%
\special{pa 2736 3732}%
\special{pa 2680 3732}%
\special{fp}%
\special{pa 2736 3732}%
\special{pa 2736 3676}%
\special{fp}%
% LINE 2 0 3 0 Black White  
% 4 4039 3553 3983 3553 4039 3553 4039 3497
% 
\special{pn 8}%
\special{pa 4039 3553}%
\special{pa 3983 3553}%
\special{fp}%
\special{pa 4039 3553}%
\special{pa 4039 3497}%
\special{fp}%
% LINE 2 0 3 0 Black White  
% 4 3296 3732 3352 3732 3296 3732 3296 3676
% 
\special{pn 8}%
\special{pa 3296 3732}%
\special{pa 3352 3732}%
\special{fp}%
\special{pa 3296 3732}%
\special{pa 3296 3676}%
\special{fp}%
% LINE 2 0 3 0 Black White  
% 4 4263 3553 4319 3553 4263 3553 4263 3497
% 
\special{pn 8}%
\special{pa 4263 3553}%
\special{pa 4319 3553}%
\special{fp}%
\special{pa 4263 3553}%
\special{pa 4263 3497}%
\special{fp}%
% STR 2 0 3 0 Black White  
% 4 3425 4085 3425 4141 2 0 0 0
% $a$
\put(34.2500,-41.4100){\makebox(0,0)[lb]{$a$}}%
% STR 2 0 3 0 Black White  
% 4 3425 3525 3425 3581 2 0 0 0
% $b$
\put(34.2500,-35.8100){\makebox(0,0)[lb]{$b$}}%
% STR 2 0 3 0 Black White  
% 4 2546 3519 2546 3575 2 0 0 0
% $c$
\put(25.4600,-35.7500){\makebox(0,0)[lb]{$c$}}%
% STR 2 0 3 0 Black White  
% 4 2540 4096 2540 4152 2 0 0 0
% $d$
\put(25.4000,-41.5200){\makebox(0,0)[lb]{$d$}}%
% STR 2 0 3 0 Black White  
% 4 2977 3882 2977 3938 2 0 0 0
% $e$
\put(29.7700,-39.3800){\makebox(0,0)[lb]{$e$}}%
% STR 2 0 3 0 Black White  
% 4 4403 4253 4403 4309 2 0 0 0
% $a$
\put(44.0300,-43.0900){\makebox(0,0)[lb]{$a$}}%
% STR 2 0 3 0 Black White  
% 4 4409 3385 4409 3441 2 0 0 0
% $b$
\put(44.0900,-34.4100){\makebox(0,0)[lb]{$b$}}%
% STR 2 0 3 0 Black White  
% 4 3832 3385 3832 3441 2 0 0 0
% $c$
\put(38.3200,-34.4100){\makebox(0,0)[lb]{$c$}}%
% STR 2 0 3 0 Black White  
% 4 3838 4264 3838 4320 2 0 0 0
% $d$
\put(38.3800,-43.2000){\makebox(0,0)[lb]{$d$}}%
% STR 2 0 3 0 Black White  
% 4 4207 3833 4207 3889 2 0 0 0
% $e^{\prime}$
\put(42.0700,-38.8900){\makebox(0,0)[lb]{$e^{\prime}$}}%
% STR 2 0 3 0 Black White  
% 4 260 145 260 195 2 0 0 0
% I: $a\prec b\prec c\prec d$
\put(2.6000,-1.9500){\makebox(0,0)[lb]{I: $a\prec b\prec c\prec d$}}%
% STR 2 0 3 0 Black White  
% 4 2660 145 2660 195 2 0 0 0
% II: $a\prec b\prec d\prec c$
\put(26.6000,-1.9500){\makebox(0,0)[lb]{II: $a\prec b\prec d\prec c$}}%
% STR 2 0 3 0 Black White  
% 4 260 1645 260 1695 2 0 0 0
% III: $a\prec c\prec b\prec d$
\put(2.6000,-16.9500){\makebox(0,0)[lb]{III: $a\prec c\prec b\prec d$}}%
% STR 2 0 3 0 Black White  
% 4 2660 1645 2660 1695 2 0 0 0
% IV: $a\prec c\prec d\prec b$
\put(26.6000,-16.9500){\makebox(0,0)[lb]{IV: $a\prec c\prec d\prec b$}}%
% STR 2 0 3 0 Black White  
% 4 260 3145 260 3195 2 0 0 0
% V: $a\prec d\prec b\prec c$
\put(2.6000,-31.9500){\makebox(0,0)[lb]{V: $a\prec d\prec b\prec c$}}%
% STR 2 0 3 0 Black White  
% 4 2660 3145 2660 3195 2 0 0 0
% VI: $a\prec d\prec c\prec b$
\put(26.6000,-31.9500){\makebox(0,0)[lb]{VI: $a\prec d\prec c\prec b$}}%
% CIRCLE 2 0 3 0 Black White  
% 4 480 835 516 769 516 769 515 886
% 
\special{pn 8}%
\special{ar 480 835 75 75  0.9693416  5.2117357}%
% DOT 0 0 3 0 Black White  
% 1 516 769
% 
\special{pn 4}%
\special{sh 1}%
\special{ar 516 769 28 28 0  6.28318530717959E+0000}%
% LINE 2 0 3 0 Black White  
% 4 527 898 491 891 527 898 500 924
% 
\special{pn 8}%
\special{pa 527 898}%
\special{pa 491 891}%
\special{fp}%
\special{pa 527 898}%
\special{pa 500 924}%
\special{fp}%
% CIRCLE 2 0 3 0 Black White  
% 4 810 845 884 831 884 831 793 905
% 
\special{pn 8}%
\special{ar 810 845 75 75  1.8468933  6.0962060}%
% DOT 0 0 3 0 Black White  
% 1 884 831
% 
\special{pn 4}%
\special{sh 1}%
\special{ar 884 831 28 28 0  6.28318530717959E+0000}%
% LINE 2 0 3 0 Black White  
% 4 792 921 774 889 792 921 755 917
% 
\special{pn 8}%
\special{pa 792 921}%
\special{pa 774 889}%
\special{fp}%
\special{pa 792 921}%
\special{pa 755 917}%
\special{fp}%
% CIRCLE 2 0 3 0 Black White  
% 4 1771 665 1840 694 1840 694 1723 705
% 
\special{pn 8}%
\special{ar 1771 665 75 75  2.4468544  0.3978744}%
% DOT 0 0 3 0 Black White  
% 1 1840 694
% 
\special{pn 4}%
\special{sh 1}%
\special{ar 1840 694 28 28 0  6.28318530717959E+0000}%
% LINE 2 0 3 0 Black White  
% 4 1713 718 1716 681 1713 718 1685 694
% 
\special{pn 8}%
\special{pa 1713 718}%
\special{pa 1716 681}%
\special{fp}%
\special{pa 1713 718}%
\special{pa 1685 694}%
\special{fp}%
% CIRCLE 2 0 3 0 Black White  
% 4 1776 1011 1847 984 1847 984 1770 1071
% 
\special{pn 8}%
\special{ar 1776 1011 76 76  1.6704650  5.9197922}%
% DOT 0 0 3 0 Black White  
% 1 1847 984
% 
\special{pn 4}%
\special{sh 1}%
\special{ar 1847 984 28 28 0  6.28318530717959E+0000}%
% LINE 2 0 3 0 Black White  
% 4 1772 1089 1748 1061 1772 1089 1734 1091
% 
\special{pn 8}%
\special{pa 1772 1089}%
\special{pa 1748 1061}%
\special{fp}%
\special{pa 1772 1089}%
\special{pa 1734 1091}%
\special{fp}%
% CIRCLE 2 0 3 0 Black White  
% 4 2837 839 2870 770 2777 840 2870 770
% 
\special{pn 8}%
\special{ar 2837 839 76 76  5.1584945  3.1249275}%
% DOT 0 0 3 0 Black White  
% 1 2870 770
% 
\special{pn 4}%
\special{sh 1}%
\special{ar 2870 770 28 28 0  6.28318530717959E+0000}%
% LINE 2 0 3 0 Black White  
% 4 2759 837 2787 863 2759 837 2755 875
% 
\special{pn 8}%
\special{pa 2759 837}%
\special{pa 2787 863}%
\special{fp}%
\special{pa 2759 837}%
\special{pa 2755 875}%
\special{fp}%
% CIRCLE 2 0 3 0 Black White  
% 4 3190 845 3264 832 3264 832 3172 904
% 
\special{pn 8}%
\special{ar 3190 845 75 75  1.8669115  6.1092841}%
% DOT 0 0 3 0 Black White  
% 1 3264 832
% 
\special{pn 4}%
\special{sh 1}%
\special{ar 3264 832 28 28 0  6.28318530717959E+0000}%
% LINE 2 0 3 0 Black White  
% 4 3170 921 3153 888 3170 921 3133 916
% 
\special{pn 8}%
\special{pa 3170 921}%
\special{pa 3153 888}%
\special{fp}%
\special{pa 3170 921}%
\special{pa 3133 916}%
\special{fp}%
% CIRCLE 2 0 3 0 Black White  
% 4 4146 674 4215 702 4215 702 4099 714
% 
\special{pn 8}%
\special{ar 4146 674 74 74  2.4364813  0.3854939}%
% DOT 0 0 3 0 Black White  
% 1 4215 702
% 
\special{pn 4}%
\special{sh 1}%
\special{ar 4215 702 28 28 0  6.28318530717959E+0000}%
% LINE 2 0 3 0 Black White  
% 4 4089 727 4091 690 4089 727 4060 703
% 
\special{pn 8}%
\special{pa 4089 727}%
\special{pa 4091 690}%
\special{fp}%
\special{pa 4089 727}%
\special{pa 4060 703}%
\special{fp}%
% CIRCLE 2 0 3 0 Black White  
% 4 4154 1002 4220 966 4104 967 4220 966
% 
\special{pn 8}%
\special{ar 4154 1002 75 75  5.7838386  3.7523186}%
% DOT 0 0 3 0 Black White  
% 1 4220 966
% 
\special{pn 4}%
\special{sh 1}%
\special{ar 4220 966 28 28 0  6.28318530717959E+0000}%
% LINE 2 0 3 0 Black White  
% 4 4091 955 4098 991 4091 955 4065 982
% 
\special{pn 8}%
\special{pa 4091 955}%
\special{pa 4098 991}%
\special{fp}%
\special{pa 4091 955}%
\special{pa 4065 982}%
\special{fp}%
% CIRCLE 2 0 3 0 Black White  
% 4 448 2346 375 2325 459 2406 375 2325
% 
\special{pn 8}%
\special{ar 448 2346 76 76  3.4217006  1.3894766}%
% DOT 0 0 3 0 Black White  
% 1 375 2325
% 
\special{pn 4}%
\special{sh 1}%
\special{ar 375 2325 28 28 0  6.28318530717959E+0000}%
% LINE 2 0 3 0 Black White  
% 4 460 2423 480 2392 460 2423 497 2422
% 
\special{pn 8}%
\special{pa 460 2423}%
\special{pa 480 2392}%
\special{fp}%
\special{pa 460 2423}%
\special{pa 497 2422}%
\special{fp}%
% CIRCLE 2 0 3 0 Black White  
% 4 815 2350 889 2338 889 2338 796 2408
% 
\special{pn 8}%
\special{ar 815 2350 75 75  1.8873656  6.1224226}%
% DOT 0 0 3 0 Black White  
% 1 889 2338
% 
\special{pn 4}%
\special{sh 1}%
\special{ar 889 2338 28 28 0  6.28318530717959E+0000}%
% LINE 2 0 3 0 Black White  
% 4 794 2426 777 2393 794 2426 757 2420
% 
\special{pn 8}%
\special{pa 794 2426}%
\special{pa 777 2393}%
\special{fp}%
\special{pa 794 2426}%
\special{pa 757 2420}%
\special{fp}%
% CIRCLE 2 0 3 0 Black White  
% 4 1775 2163 1843 2194 1773 2101 1843 2194
% 
\special{pn 8}%
\special{ar 1775 2163 75 75  0.4277349  4.6801421}%
% DOT 0 0 3 0 Black White  
% 1 1843 2194
% 
\special{pn 4}%
\special{sh 1}%
\special{ar 1843 2194 28 28 0  6.28318530717959E+0000}%
% LINE 2 0 3 0 Black White  
% 4 1775 2084 1750 2111 1775 2084 1738 2080
% 
\special{pn 8}%
\special{pa 1775 2084}%
\special{pa 1750 2111}%
\special{fp}%
\special{pa 1775 2084}%
\special{pa 1738 2080}%
\special{fp}%
% CIRCLE 2 0 3 0 Black White  
% 4 1775 2505 1844 2476 1844 2476 1770 2566
% 
\special{pn 8}%
\special{ar 1775 2505 75 75  1.6525807  5.8853109}%
% DOT 0 0 3 0 Black White  
% 1 1844 2476
% 
\special{pn 4}%
\special{sh 1}%
\special{ar 1844 2476 28 28 0  6.28318530717959E+0000}%
% LINE 2 0 3 0 Black White  
% 4 1772 2583 1748 2555 1772 2583 1735 2586
% 
\special{pn 8}%
\special{pa 1772 2583}%
\special{pa 1748 2555}%
\special{fp}%
\special{pa 1772 2583}%
\special{pa 1735 2586}%
\special{fp}%
% CIRCLE 2 0 3 0 Black White  
% 4 2844 2340 2770 2353 2770 2353 2862 2282
% 
\special{pn 8}%
\special{ar 2844 2340 75 75  5.0133092  2.9676915}%
% DOT 0 0 3 0 Black White  
% 1 2770 2353
% 
\special{pn 4}%
\special{sh 1}%
\special{ar 2770 2353 28 28 0  6.28318530717959E+0000}%
% LINE 2 0 3 0 Black White  
% 4 2864 2265 2881 2297 2864 2265 2901 2270
% 
\special{pn 8}%
\special{pa 2864 2265}%
\special{pa 2881 2297}%
\special{fp}%
\special{pa 2864 2265}%
\special{pa 2901 2270}%
\special{fp}%
% CIRCLE 2 0 3 0 Black White  
% 4 3196 2326 3268 2349 3185 2267 3268 2349
% 
\special{pn 8}%
\special{ar 3196 2326 76 76  0.3091989  4.5280646}%
% DOT 0 0 3 0 Black White  
% 1 3268 2349
% 
\special{pn 4}%
\special{sh 1}%
\special{ar 3268 2349 28 28 0  6.28318530717959E+0000}%
% LINE 2 0 3 0 Black White  
% 4 3185 2250 3164 2280 3185 2250 3148 2251
% 
\special{pn 8}%
\special{pa 3185 2250}%
\special{pa 3164 2280}%
\special{fp}%
\special{pa 3185 2250}%
\special{pa 3148 2251}%
\special{fp}%
% CIRCLE 2 0 3 0 Black White  
% 4 4155 2161 4224 2190 4150 2100 4224 2190
% 
\special{pn 8}%
\special{ar 4155 2161 75 75  0.3978744  4.6306046}%
% DOT 0 0 3 0 Black White  
% 1 4224 2190
% 
\special{pn 4}%
\special{sh 1}%
\special{ar 4224 2190 28 28 0  6.28318530717959E+0000}%
% LINE 2 0 3 0 Black White  
% 4 4152 2083 4128 2111 4152 2083 4115 2080
% 
\special{pn 8}%
\special{pa 4152 2083}%
\special{pa 4128 2111}%
\special{fp}%
\special{pa 4152 2083}%
\special{pa 4115 2080}%
\special{fp}%
% CIRCLE 2 0 3 0 Black White  
% 4 4155 2505 4224 2476 4224 2476 4150 2566
% 
\special{pn 8}%
\special{ar 4155 2505 75 75  1.6525807  5.8853109}%
% DOT 0 0 3 0 Black White  
% 1 4224 2476
% 
\special{pn 4}%
\special{sh 1}%
\special{ar 4224 2476 28 28 0  6.28318530717959E+0000}%
% LINE 2 0 3 0 Black White  
% 4 4151 2583 4128 2555 4151 2583 4114 2586
% 
\special{pn 8}%
\special{pa 4151 2583}%
\special{pa 4128 2555}%
\special{fp}%
\special{pa 4151 2583}%
\special{pa 4114 2586}%
\special{fp}%
% CIRCLE 2 0 3 0 Black White  
% 4 461 3841 500 3906 500 3906 400 3845
% 
\special{pn 8}%
\special{ar 461 3841 76 76  3.0761126  1.0303768}%
% DOT 0 0 3 0 Black White  
% 1 500 3906
% 
\special{pn 4}%
\special{sh 1}%
\special{ar 500 3906 28 28 0  6.28318530717959E+0000}%
% LINE 2 0 3 0 Black White  
% 4 384 3850 407 3822 384 3850 375 3813
% 
\special{pn 8}%
\special{pa 384 3850}%
\special{pa 407 3822}%
\special{fp}%
\special{pa 384 3850}%
\special{pa 375 3813}%
\special{fp}%
% CIRCLE 2 0 3 0 Black White  
% 4 820 3828 892 3849 809 3767 892 3849
% 
\special{pn 8}%
\special{ar 820 3828 75 75  0.2837941  4.5339785}%
% DOT 0 0 3 0 Black White  
% 1 892 3849
% 
\special{pn 4}%
\special{sh 1}%
\special{ar 892 3849 28 28 0  6.28318530717959E+0000}%
% LINE 2 0 3 0 Black White  
% 4 809 3750 788 3781 809 3750 771 3751
% 
\special{pn 8}%
\special{pa 809 3750}%
\special{pa 788 3781}%
\special{fp}%
\special{pa 809 3750}%
\special{pa 771 3751}%
\special{fp}%
% CIRCLE 2 0 3 0 Black White  
% 4 1767 3664 1835 3694 1835 3694 1719 3703
% 
\special{pn 8}%
\special{ar 1767 3664 74 74  2.4592761  0.4154921}%
% DOT 0 0 3 0 Black White  
% 1 1835 3694
% 
\special{pn 4}%
\special{sh 1}%
\special{ar 1835 3694 28 28 0  6.28318530717959E+0000}%
% LINE 2 0 3 0 Black White  
% 4 1708 3716 1712 3680 1708 3716 1680 3692
% 
\special{pn 8}%
\special{pa 1708 3716}%
\special{pa 1712 3680}%
\special{fp}%
\special{pa 1708 3716}%
\special{pa 1680 3692}%
\special{fp}%
% CIRCLE 2 0 3 0 Black White  
% 4 1773 4004 1840 3970 1724 3968 1840 3970
% 
\special{pn 8}%
\special{ar 1773 4004 75 75  5.8135854  3.7752257}%
% DOT 0 0 3 0 Black White  
% 1 1840 3970
% 
\special{pn 4}%
\special{sh 1}%
\special{ar 1840 3970 28 28 0  6.28318530717959E+0000}%
% LINE 2 0 3 0 Black White  
% 4 1712 3955 1718 3992 1712 3955 1685 3981
% 
\special{pn 8}%
\special{pa 1712 3955}%
\special{pa 1718 3992}%
\special{fp}%
\special{pa 1712 3955}%
\special{pa 1685 3981}%
\special{fp}%
% CIRCLE 2 0 3 0 Black White  
% 4 2841 3837 2873 3906 2879 3789 2873 3906
% 
\special{pn 8}%
\special{ar 2841 3837 76 76  1.1365520  5.3820279}%
% DOT 0 0 3 0 Black White  
% 1 2873 3906
% 
\special{pn 4}%
\special{sh 1}%
\special{ar 2873 3906 28 28 0  6.28318530717959E+0000}%
% LINE 2 0 3 0 Black White  
% 4 2892 3778 2856 3782 2892 3778 2867 3750
% 
\special{pn 8}%
\special{pa 2892 3778}%
\special{pa 2856 3782}%
\special{fp}%
\special{pa 2892 3778}%
\special{pa 2867 3750}%
\special{fp}%
% CIRCLE 2 0 3 0 Black White  
% 4 3200 3823 3271 3847 3191 3763 3271 3847
% 
\special{pn 8}%
\special{ar 3200 3823 75 75  0.3259699  4.5634990}%
% DOT 0 0 3 0 Black White  
% 1 3271 3847
% 
\special{pn 4}%
\special{sh 1}%
\special{ar 3271 3847 28 28 0  6.28318530717959E+0000}%
% LINE 2 0 3 0 Black White  
% 4 3191 3746 3169 3775 3191 3746 3154 3745
% 
\special{pn 8}%
\special{pa 3191 3746}%
\special{pa 3169 3775}%
\special{fp}%
\special{pa 3191 3746}%
\special{pa 3154 3745}%
\special{fp}%
% CIRCLE 2 0 3 0 Black White  
% 4 4156 4001 4225 3970 4109 3963 4225 3970
% 
\special{pn 8}%
\special{ar 4156 4001 76 76  5.8609342  3.8215015}%
% DOT 0 0 3 0 Black White  
% 1 4225 3970
% 
\special{pn 4}%
\special{sh 1}%
\special{ar 4225 3970 28 28 0  6.28318530717959E+0000}%
% LINE 2 0 3 0 Black White  
% 4 4097 3950 4101 3987 4097 3950 4070 3975
% 
\special{pn 8}%
\special{pa 4097 3950}%
\special{pa 4101 3987}%
\special{fp}%
\special{pa 4097 3950}%
\special{pa 4070 3975}%
\special{fp}%
\end{picture}}%

%% file: example1.tex
%WinTpicVersion4.31b
{\unitlength 0.1in%
\begin{picture}( 42.1000,  8.3000)(  3.9000,-10.6000)%
% LINE 2 0 3 0 Black White  
% 2 400 1000 400 800
% 
\special{pn 8}%
\special{pa 400 1000}%
\special{pa 400 800}%
\special{fp}%
% CIRCLE 2 0 3 0 Black White  
% 4 600 800 600 600 600 600 400 800
% 
\special{pn 8}%
\special{ar 600 800 200 200  3.1415927  4.7123890}%
% LINE 2 0 3 0 Black White  
% 2 1200 1000 1200 800
% 
\special{pn 8}%
\special{pa 1200 1000}%
\special{pa 1200 800}%
\special{fp}%
% CIRCLE 2 0 3 0 Black White  
% 4 1000 800 1200 800 1200 800 1000 600
% 
\special{pn 8}%
\special{ar 1000 800 200 200  4.7123890  6.2831853}%
% LINE 2 0 3 0 Black White  
% 2 1000 600 600 600
% 
\special{pn 8}%
\special{pa 1000 600}%
\special{pa 600 600}%
\special{fp}%
% ELLIPSE 2 0 3 0 Black White  
% 4 400 600 680 800 400 800 600 640
% 
\special{pn 8}%
\special{ar 400 600 280 200  0.2730087  1.5707963}%
% CIRCLE 2 0 3 0 Black White  
% 4 840 600 680 600 1080 600 680 560
% 
\special{pn 8}%
\special{ar 840 600 160 160  3.3865713  6.2831853}%
% DOT 0 0 3 0 Black White  
% 1 400 800
% 
\special{pn 4}%
\special{sh 1}%
\special{ar 400 800 16 16 0  6.28318530717959E+0000}%
% DOT 0 0 3 0 Black White  
% 1 1000 600
% 
\special{pn 4}%
\special{sh 1}%
\special{ar 1000 600 16 16 0  6.28318530717959E+0000}%
% LINE 2 0 3 0 Black White  
% 2 1600 1000 1600 800
% 
\special{pn 8}%
\special{pa 1600 1000}%
\special{pa 1600 800}%
\special{fp}%
% CIRCLE 2 0 3 0 Black White  
% 4 1800 800 1800 600 1800 600 1600 800
% 
\special{pn 8}%
\special{ar 1800 800 200 200  3.1415927  4.7123890}%
% LINE 2 0 3 0 Black White  
% 2 2400 1000 2400 800
% 
\special{pn 8}%
\special{pa 2400 1000}%
\special{pa 2400 800}%
\special{fp}%
% CIRCLE 2 0 3 0 Black White  
% 4 2200 800 2400 800 2400 800 2200 600
% 
\special{pn 8}%
\special{ar 2200 800 200 200  4.7123890  6.2831853}%
% LINE 2 0 3 0 Black White  
% 2 2200 600 1800 600
% 
\special{pn 8}%
\special{pa 2200 600}%
\special{pa 1800 600}%
\special{fp}%
% ELLIPSE 2 0 3 0 Black White  
% 4 1600 600 1880 800 1600 800 1800 640
% 
\special{pn 8}%
\special{ar 1600 600 280 200  0.2730087  1.5707963}%
% CIRCLE 2 0 3 0 Black White  
% 4 2040 600 1880 600 2280 600 1880 560
% 
\special{pn 8}%
\special{ar 2040 600 160 160  3.3865713  6.2831853}%
% DOT 0 0 3 0 Black White  
% 1 1600 800
% 
\special{pn 4}%
\special{sh 1}%
\special{ar 1600 800 16 16 0  6.28318530717959E+0000}%
% DOT 0 0 3 0 Black White  
% 1 2200 600
% 
\special{pn 4}%
\special{sh 1}%
\special{ar 2200 600 16 16 0  6.28318530717959E+0000}%
% LINE 2 0 3 0 Black White  
% 2 3400 1000 3400 800
% 
\special{pn 8}%
\special{pa 3400 1000}%
\special{pa 3400 800}%
\special{fp}%
% CIRCLE 2 0 3 0 Black White  
% 4 3600 800 3600 600 3600 600 3400 800
% 
\special{pn 8}%
\special{ar 3600 800 200 200  3.1415927  4.7123890}%
% LINE 2 0 3 0 Black White  
% 2 4200 1000 4200 800
% 
\special{pn 8}%
\special{pa 4200 1000}%
\special{pa 4200 800}%
\special{fp}%
% CIRCLE 2 0 3 0 Black White  
% 4 4000 800 4200 800 4200 800 4000 600
% 
\special{pn 8}%
\special{ar 4000 800 200 200  4.7123890  6.2831853}%
% LINE 2 0 3 0 Black White  
% 2 4000 600 3600 600
% 
\special{pn 8}%
\special{pa 4000 600}%
\special{pa 3600 600}%
\special{fp}%
% ELLIPSE 2 0 3 0 Black White  
% 4 3400 600 3680 800 3400 800 3600 640
% 
\special{pn 8}%
\special{ar 3400 600 280 200  0.2730087  1.5707963}%
% CIRCLE 2 0 3 0 Black White  
% 4 3840 600 3680 600 4080 600 3680 560
% 
\special{pn 8}%
\special{ar 3840 600 160 160  3.3865713  6.2831853}%
% DOT 0 0 3 0 Black White  
% 1 3400 800
% 
\special{pn 4}%
\special{sh 1}%
\special{ar 3400 800 16 16 0  6.28318530717959E+0000}%
% DOT 0 0 3 0 Black White  
% 1 4000 600
% 
\special{pn 4}%
\special{sh 1}%
\special{ar 4000 600 16 16 0  6.28318530717959E+0000}%
% DOT 0 0 3 0 Black White  
% 1 1200 1000
% 
\special{pn 4}%
\special{sh 1}%
\special{ar 1200 1000 16 16 0  6.28318530717959E+0000}%
% DOT 0 0 3 0 Black White  
% 1 1600 1000
% 
\special{pn 4}%
\special{sh 1}%
\special{ar 1600 1000 16 16 0  6.28318530717959E+0000}%
% DOT 0 0 3 0 Black White  
% 1 2400 1000
% 
\special{pn 4}%
\special{sh 1}%
\special{ar 2400 1000 16 16 0  6.28318530717959E+0000}%
% DOT 0 0 3 0 Black White  
% 1 3400 1000
% 
\special{pn 4}%
\special{sh 1}%
\special{ar 3400 1000 16 16 0  6.28318530717959E+0000}%
% DOT 0 0 3 0 Black White  
% 1 4200 1000
% 
\special{pn 4}%
\special{sh 1}%
\special{ar 4200 1000 16 16 0  6.28318530717959E+0000}%
% DOT 0 0 3 0 Black White  
% 1 4600 1000
% 
\special{pn 4}%
\special{sh 1}%
\special{ar 4600 1000 16 16 0  6.28318530717959E+0000}%
% STR 2 0 3 0 Black White  
% 4 2820 700 2820 800 2 0 0 0
% $\cdots$
\put(28.2000,-8.0000){\makebox(0,0)[lb]{$\cdots$}}%
% STR 2 0 3 0 Black White  
% 4 4220 700 4220 800 2 0 0 0
% $e_1^1$
\put(42.2000,-8.0000){\makebox(0,0)[lb]{$e_1^1$}}%
% STR 2 0 3 0 Black White  
% 4 3765 275 3765 375 2 0 0 0
% $e_1^2$
\put(37.6500,-3.7500){\makebox(0,0)[lb]{$e_1^2$}}%
% STR 2 0 3 0 Black White  
% 4 3390 500 3390 600 2 0 0 0
% $e_1^3$
\put(33.9000,-6.0000){\makebox(0,0)[lb]{$e_1^3$}}%
% STR 2 0 3 0 Black White  
% 4 3740 1090 3740 1190 2 0 0 0
% $e_1^4$
\put(37.4000,-11.9000){\makebox(0,0)[lb]{$e_1^4$}}%
% STR 2 0 3 0 Black White  
% 4 3425 865 3425 965 2 0 0 0
% $e_1^5$
\put(34.2500,-9.6500){\makebox(0,0)[lb]{$e_1^5$}}%
% STR 2 0 3 0 Black White  
% 4 1340 1105 1340 1205 2 0 0 0
% $e_g^0$
\put(13.4000,-12.0500){\makebox(0,0)[lb]{$e_g^0$}}%
% STR 2 0 3 0 Black White  
% 4 3105 1100 3105 1200 2 0 0 0
% $e_2^0$
\put(31.0500,-12.0000){\makebox(0,0)[lb]{$e_2^0$}}%
% LINE 2 0 3 0 Black White  
% 2 400 1000 4600 1000
% 
\special{pn 8}%
\special{pa 400 1000}%
\special{pa 4600 1000}%
\special{fp}%
% STR 2 0 3 0 Black White  
% 4 2505 1105 2505 1205 2 0 0 0
% $e_{g-1}^0$
\put(25.0500,-12.0500){\makebox(0,0)[lb]{$e_{g-1}^0$}}%
% STR 2 0 3 0 Black White  
% 4 2830 1105 2830 1155 2 0 0 0
% $\cdots$
\put(28.3000,-11.5500){\makebox(0,0)[lb]{$\cdots$}}%
% STR 2 0 3 0 Black White  
% 4 4340 1090 4340 1190 2 0 0 0
% $e_1^0$
\put(43.4000,-11.9000){\makebox(0,0)[lb]{$e_1^0$}}%
% LINE 2 0 3 0 Black White  
% 4 4300 1005 4350 980 4300 1005 4350 1030
% 
\special{pn 8}%
\special{pa 4300 1005}%
\special{pa 4350 980}%
\special{fp}%
\special{pa 4300 1005}%
\special{pa 4350 1030}%
\special{fp}%
% LINE 2 0 3 0 Black White  
% 4 3600 1005 3650 980 3600 1005 3650 1030
% 
\special{pn 8}%
\special{pa 3600 1005}%
\special{pa 3650 980}%
\special{fp}%
\special{pa 3600 1005}%
\special{pa 3650 1030}%
\special{fp}%
% LINE 2 0 3 0 Black White  
% 4 3250 1005 3300 980 3250 1005 3300 1030
% 
\special{pn 8}%
\special{pa 3250 1005}%
\special{pa 3300 980}%
\special{fp}%
\special{pa 3250 1005}%
\special{pa 3300 1030}%
\special{fp}%
% LINE 2 0 3 0 Black White  
% 4 2500 1005 2550 980 2500 1005 2550 1030
% 
\special{pn 8}%
\special{pa 2500 1005}%
\special{pa 2550 980}%
\special{fp}%
\special{pa 2500 1005}%
\special{pa 2550 1030}%
\special{fp}%
% LINE 2 0 3 0 Black White  
% 4 1800 1005 1850 980 1800 1005 1850 1030
% 
\special{pn 8}%
\special{pa 1800 1005}%
\special{pa 1850 980}%
\special{fp}%
\special{pa 1800 1005}%
\special{pa 1850 1030}%
\special{fp}%
% LINE 2 0 3 0 Black White  
% 4 1300 1005 1350 980 1300 1005 1350 1030
% 
\special{pn 8}%
\special{pa 1300 1005}%
\special{pa 1350 980}%
\special{fp}%
\special{pa 1300 1005}%
\special{pa 1350 1030}%
\special{fp}%
% LINE 2 0 3 0 Black White  
% 4 600 1005 650 980 600 1005 650 1030
% 
\special{pn 8}%
\special{pa 600 1005}%
\special{pa 650 980}%
\special{fp}%
\special{pa 600 1005}%
\special{pa 650 1030}%
\special{fp}%
% LINE 2 0 3 0 Black White  
% 4 830 445 880 420 830 445 880 470
% 
\special{pn 8}%
\special{pa 830 445}%
\special{pa 880 420}%
\special{fp}%
\special{pa 830 445}%
\special{pa 880 470}%
\special{fp}%
% LINE 2 0 3 0 Black White  
% 4 4200 800 4175 850 4200 800 4225 850
% 
\special{pn 8}%
\special{pa 4200 800}%
\special{pa 4175 850}%
\special{fp}%
\special{pa 4200 800}%
\special{pa 4225 850}%
\special{fp}%
% LINE 2 0 3 0 Black White  
% 4 3400 900 3375 950 3400 900 3425 950
% 
\special{pn 8}%
\special{pa 3400 900}%
\special{pa 3375 950}%
\special{fp}%
\special{pa 3400 900}%
\special{pa 3425 950}%
\special{fp}%
% LINE 2 0 3 0 Black White  
% 4 3610 605 3560 580 3610 605 3560 630
% 
\special{pn 8}%
\special{pa 3610 605}%
\special{pa 3560 580}%
\special{fp}%
\special{pa 3610 605}%
\special{pa 3560 630}%
\special{fp}%
% LINE 2 0 3 0 Black White  
% 4 3830 445 3880 420 3830 445 3880 470
% 
\special{pn 8}%
\special{pa 3830 445}%
\special{pa 3880 420}%
\special{fp}%
\special{pa 3830 445}%
\special{pa 3880 470}%
\special{fp}%
% LINE 2 0 3 0 Black White  
% 4 2030 445 2080 420 2030 445 2080 470
% 
\special{pn 8}%
\special{pa 2030 445}%
\special{pa 2080 420}%
\special{fp}%
\special{pa 2030 445}%
\special{pa 2080 470}%
\special{fp}%
% LINE 2 0 3 0 Black White  
% 4 1810 605 1760 580 1810 605 1760 630
% 
\special{pn 8}%
\special{pa 1810 605}%
\special{pa 1760 580}%
\special{fp}%
\special{pa 1810 605}%
\special{pa 1760 630}%
\special{fp}%
% LINE 2 0 3 0 Black White  
% 4 610 605 560 580 610 605 560 630
% 
\special{pn 8}%
\special{pa 610 605}%
\special{pa 560 580}%
\special{fp}%
\special{pa 610 605}%
\special{pa 560 630}%
\special{fp}%
% LINE 2 0 3 0 Black White  
% 4 2400 800 2375 850 2400 800 2425 850
% 
\special{pn 8}%
\special{pa 2400 800}%
\special{pa 2375 850}%
\special{fp}%
\special{pa 2400 800}%
\special{pa 2425 850}%
\special{fp}%
% LINE 2 0 3 0 Black White  
% 4 1200 800 1175 850 1200 800 1225 850
% 
\special{pn 8}%
\special{pa 1200 800}%
\special{pa 1175 850}%
\special{fp}%
\special{pa 1200 800}%
\special{pa 1225 850}%
\special{fp}%
% LINE 2 0 3 0 Black White  
% 4 1600 900 1575 950 1600 900 1625 950
% 
\special{pn 8}%
\special{pa 1600 900}%
\special{pa 1575 950}%
\special{fp}%
\special{pa 1600 900}%
\special{pa 1625 950}%
\special{fp}%
% STR 2 0 3 0 Black White  
% 4 2420 700 2420 800 2 0 0 0
% $e_{g-1}^1$
\put(24.2000,-8.0000){\makebox(0,0)[lb]{$e_{g-1}^1$}}%
% STR 2 0 3 0 Black White  
% 4 1965 285 1965 385 2 0 0 0
% $e_{g-1}^2$
\put(19.6500,-3.8500){\makebox(0,0)[lb]{$e_{g-1}^2$}}%
% STR 2 0 3 0 Black White  
% 4 1590 475 1590 575 2 0 0 0
% $e_{g-1}^3$
\put(15.9000,-5.7500){\makebox(0,0)[lb]{$e_{g-1}^3$}}%
% STR 2 0 3 0 Black White  
% 4 1940 1100 1940 1200 2 0 0 0
% $e_{g-1}^4$
\put(19.4000,-12.0000){\makebox(0,0)[lb]{$e_{g-1}^4$}}%
% STR 2 0 3 0 Black White  
% 4 1625 885 1625 985 2 0 0 0
% $e_{g-1}^5$
\put(16.2500,-9.8500){\makebox(0,0)[lb]{$e_{g-1}^5$}}%
% STR 2 0 3 0 Black White  
% 4 740 1100 740 1200 2 0 0 0
% $e_g^4$
\put(7.4000,-12.0000){\makebox(0,0)[lb]{$e_g^4$}}%
% STR 2 0 3 0 Black White  
% 4 390 475 390 575 2 0 0 0
% $e_g^3$
\put(3.9000,-5.7500){\makebox(0,0)[lb]{$e_g^3$}}%
% STR 2 0 3 0 Black White  
% 4 765 285 765 385 2 0 0 0
% $e_g^2$
\put(7.6500,-3.8500){\makebox(0,0)[lb]{$e_g^2$}}%
% STR 2 0 3 0 Black White  
% 4 1220 700 1220 800 2 0 0 0
% $e_g^1$
\put(12.2000,-8.0000){\makebox(0,0)[lb]{$e_g^1$}}%
\end{picture}}%

%% file: gluing.tex
%WinTpicVersion4.10
\unitlength 0.1in
\begin{picture}( 34.0000,  8.5500)(  3.0000, -9.5500)
% LINE 1 0 3 0 Black White
% 2 700 300 700 700
% 
{\color[named]{Black}{%
\special{pn 13}%
\special{pa 700 300}%
\special{pa 700 700}%
\special{fp}%
}}%
% LINE 1 0 3 0 Black White
% 2 700 300 900 200
% 
{\color[named]{Black}{%
\special{pn 13}%
\special{pa 700 300}%
\special{pa 900 200}%
\special{fp}%
}}%
% LINE 1 0 3 0 Black White
% 2 700 300 500 200
% 
{\color[named]{Black}{%
\special{pn 13}%
\special{pa 700 300}%
\special{pa 500 200}%
\special{fp}%
}}%
% LINE 1 0 3 0 Black White
% 2 700 700 900 800
% 
{\color[named]{Black}{%
\special{pn 13}%
\special{pa 700 700}%
\special{pa 900 800}%
\special{fp}%
}}%
% LINE 1 0 3 0 Black White
% 2 700 700 500 800
% 
{\color[named]{Black}{%
\special{pn 13}%
\special{pa 700 700}%
\special{pa 500 800}%
\special{fp}%
}}%
% LINE 2 0 3 0 Black White
% 4 700 400 730 460 700 400 670 460
% 
{\color[named]{Black}{%
\special{pn 8}%
\special{pa 700 400}%
\special{pa 730 460}%
\special{fp}%
\special{pa 700 400}%
\special{pa 670 460}%
\special{fp}%
}}%
% LINE 2 1 3 0 Black White
% 2 900 200 1100 100
% 
{\color[named]{Black}{%
\special{pn 8}%
\special{pa 900 200}%
\special{pa 1100 100}%
\special{da 0.070}%
}}%
% LINE 2 1 3 0 Black White
% 2 500 200 300 100
% 
{\color[named]{Black}{%
\special{pn 8}%
\special{pa 500 200}%
\special{pa 300 100}%
\special{da 0.070}%
}}%
% LINE 2 1 3 0 Black White
% 2 500 800 300 900
% 
{\color[named]{Black}{%
\special{pn 8}%
\special{pa 500 800}%
\special{pa 300 900}%
\special{da 0.070}%
}}%
% LINE 2 1 3 0 Black White
% 2 900 800 1100 900
% 
{\color[named]{Black}{%
\special{pn 8}%
\special{pa 900 800}%
\special{pa 1100 900}%
\special{da 0.070}%
}}%
% LINE 1 0 3 0 Black White
% 2 1100 500 1300 500
% 
{\color[named]{Black}{%
\special{pn 13}%
\special{pa 1100 500}%
\special{pa 1300 500}%
\special{fp}%
}}%
% ELLIPSE 2 0 3 0 Black White
% 4 1600 500 1300 350 1300 350 1300 350
% 
{\color[named]{Black}{%
\special{pn 8}%
\special{ar 1600 500 300 150  0.0000000 6.2831853}%
}}%
% LINE 1 0 3 0 Black White
% 2 2900 300 2900 700
% 
{\color[named]{Black}{%
\special{pn 13}%
\special{pa 2900 300}%
\special{pa 2900 700}%
\special{fp}%
}}%
% LINE 1 0 3 0 Black White
% 2 2900 300 3100 200
% 
{\color[named]{Black}{%
\special{pn 13}%
\special{pa 2900 300}%
\special{pa 3100 200}%
\special{fp}%
}}%
% LINE 1 0 3 0 Black White
% 2 2900 300 2700 200
% 
{\color[named]{Black}{%
\special{pn 13}%
\special{pa 2900 300}%
\special{pa 2700 200}%
\special{fp}%
}}%
% LINE 1 0 3 0 Black White
% 2 2900 700 3100 800
% 
{\color[named]{Black}{%
\special{pn 13}%
\special{pa 2900 700}%
\special{pa 3100 800}%
\special{fp}%
}}%
% LINE 1 0 3 0 Black White
% 2 2900 700 2700 800
% 
{\color[named]{Black}{%
\special{pn 13}%
\special{pa 2900 700}%
\special{pa 2700 800}%
\special{fp}%
}}%
% LINE 2 0 3 0 Black White
% 4 2900 350 2930 410 2900 350 2870 410
% 
{\color[named]{Black}{%
\special{pn 8}%
\special{pa 2900 350}%
\special{pa 2930 410}%
\special{fp}%
\special{pa 2900 350}%
\special{pa 2870 410}%
\special{fp}%
}}%
% LINE 2 1 3 0 Black White
% 2 3100 200 3300 100
% 
{\color[named]{Black}{%
\special{pn 8}%
\special{pa 3100 200}%
\special{pa 3300 100}%
\special{da 0.070}%
}}%
% LINE 2 1 3 0 Black White
% 2 2700 200 2500 100
% 
{\color[named]{Black}{%
\special{pn 8}%
\special{pa 2700 200}%
\special{pa 2500 100}%
\special{da 0.070}%
}}%
% LINE 2 1 3 0 Black White
% 2 2700 800 2500 900
% 
{\color[named]{Black}{%
\special{pn 8}%
\special{pa 2700 800}%
\special{pa 2500 900}%
\special{da 0.070}%
}}%
% LINE 2 1 3 0 Black White
% 2 3100 800 3300 900
% 
{\color[named]{Black}{%
\special{pn 8}%
\special{pa 3100 800}%
\special{pa 3300 900}%
\special{da 0.070}%
}}%
% LINE 1 0 3 0 Black White
% 2 2900 500 3100 500
% 
{\color[named]{Black}{%
\special{pn 13}%
\special{pa 2900 500}%
\special{pa 3100 500}%
\special{fp}%
}}%
% ELLIPSE 2 0 3 0 Black White
% 4 3400 500 3100 350 3100 350 3100 350
% 
{\color[named]{Black}{%
\special{pn 8}%
\special{ar 3400 500 300 150  0.0000000 6.2831853}%
}}%
% LINE 2 0 3 0 Black White
% 4 2900 550 2930 610 2900 550 2870 610
% 
{\color[named]{Black}{%
\special{pn 8}%
\special{pa 2900 550}%
\special{pa 2930 610}%
\special{fp}%
\special{pa 2900 550}%
\special{pa 2870 610}%
\special{fp}%
}}%
% DOT 0 0 3 0 Black White
% 1 2900 500
% 
{\color[named]{Black}{%
\special{pn 4}%
\special{sh 1}%
\special{ar 2900 500 26 26 0  6.28318530717959E+0000}%
}}%
% STR 2 0 3 0 Black White
% 4 600 1050 600 1100 2 0 0 0
% $G$
\put(6.0000,-11.0000){\makebox(0,0)[lb]{$G$}}%
% STR 2 0 3 0 Black White
% 4 1500 495 1500 545 2 0 0 0
% $G^{\prime}$
\put(15.0000,-5.4500){\makebox(0,0)[lb]{$G^{\prime}$}}%
% STR 2 0 3 0 Black White
% 4 755 550 755 600 2 0 0 0
% $e$
\put(7.5500,-6.0000){\makebox(0,0)[lb]{$e$}}%
% STR 2 0 3 0 Black White
% 4 2940 355 2940 405 2 0 0 0
% $e_2$
\put(29.4000,-4.0500){\makebox(0,0)[lb]{$e_2$}}%
% STR 2 0 3 0 Black White
% 4 2940 635 2940 685 2 0 0 0
% $e_1$
\put(29.4000,-6.8500){\makebox(0,0)[lb]{$e_1$}}%
% STR 2 0 3 0 Black White
% 4 2735 495 2735 545 2 0 0 0
% $v_e$
\put(27.3500,-5.4500){\makebox(0,0)[lb]{$v_e$}}%
% STR 2 0 3 0 Black White
% 4 2900 1050 2900 1100 2 0 0 0
% $G^{\prime \prime}$
\put(29.0000,-11.0000){\makebox(0,0)[lb]{$G^{\prime \prime}$}}%
\end{picture}%

%% file: tail_slide.tex
%WinTpicVersion4.10
\unitlength 0.1in
\begin{picture}( 37.0000,  8.5000)(  2.0000, -8.3300)
% LINE 2 1 3 0 Black White
% 2 400 500 200 500
% 
{\color[named]{Black}{%
\special{pn 8}%
\special{pa 400 500}%
\special{pa 200 500}%
\special{da 0.070}%
}}%
% LINE 2 1 3 0 Black White
% 2 1600 500 1800 500
% 
{\color[named]{Black}{%
\special{pn 8}%
\special{pa 1600 500}%
\special{pa 1800 500}%
\special{da 0.070}%
}}%
% LINE 2 1 3 0 Black White
% 2 2500 500 2300 500
% 
{\color[named]{Black}{%
\special{pn 8}%
\special{pa 2500 500}%
\special{pa 2300 500}%
\special{da 0.070}%
}}%
% LINE 2 1 3 0 Black White
% 2 3700 500 3900 500
% 
{\color[named]{Black}{%
\special{pn 8}%
\special{pa 3700 500}%
\special{pa 3900 500}%
\special{da 0.070}%
}}%
% STR 2 0 3 0 Black White
% 4 810 570 810 620 2 0 0 0
% $e_1$
\put(8.1000,-6.2000){\makebox(0,0)[lb]{$e_1$}}%
% LINE 2 0 3 0 Black White
% 4 1000 400 1030 340 1000 400 970 340
% 
{\color[named]{Black}{%
\special{pn 8}%
\special{pa 1000 400}%
\special{pa 1030 340}%
\special{fp}%
\special{pa 1000 400}%
\special{pa 970 340}%
\special{fp}%
}}%
% LINE 2 0 3 0 Black White
% 4 1100 500 1160 470 1100 500 1160 530
% 
{\color[named]{Black}{%
\special{pn 8}%
\special{pa 1100 500}%
\special{pa 1160 470}%
\special{fp}%
\special{pa 1100 500}%
\special{pa 1160 530}%
\special{fp}%
}}%
% LINE 1 0 3 0 Black White
% 2 400 500 1600 500
% 
{\color[named]{Black}{%
\special{pn 13}%
\special{pa 400 500}%
\special{pa 1600 500}%
\special{fp}%
}}%
% LINE 1 0 3 0 Black White
% 2 2500 500 3700 500
% 
{\color[named]{Black}{%
\special{pn 13}%
\special{pa 2500 500}%
\special{pa 3700 500}%
\special{fp}%
}}%
% DOT 0 0 3 0 Black White
% 1 1000 500
% 
{\color[named]{Black}{%
\special{pn 4}%
\special{sh 1}%
\special{ar 1000 500 26 26 0  6.28318530717959E+0000}%
}}%
% DOT 0 0 3 0 Black White
% 1 700 500
% 
{\color[named]{Black}{%
\special{pn 4}%
\special{sh 1}%
\special{ar 700 500 26 26 0  6.28318530717959E+0000}%
}}%
% DOT 0 0 3 0 Black White
% 1 3100 500
% 
{\color[named]{Black}{%
\special{pn 4}%
\special{sh 1}%
\special{ar 3100 500 26 26 0  6.28318530717959E+0000}%
}}%
% LINE 1 0 3 0 Black White
% 2 3100 500 3100 200
% 
{\color[named]{Black}{%
\special{pn 13}%
\special{pa 3100 500}%
\special{pa 3100 200}%
\special{fp}%
}}%
% LINE 1 0 3 0 Black White
% 2 1000 500 1000 200
% 
{\color[named]{Black}{%
\special{pn 13}%
\special{pa 1000 500}%
\special{pa 1000 200}%
\special{fp}%
}}%
% LINE 1 0 3 0 Black White
% 2 700 800 700 500
% 
{\color[named]{Black}{%
\special{pn 13}%
\special{pa 700 800}%
\special{pa 700 500}%
\special{fp}%
}}%
% LINE 1 0 3 0 Black White
% 2 2800 800 2800 700
% 
{\color[named]{Black}{%
\special{pn 13}%
\special{pa 2800 800}%
\special{pa 2800 700}%
\special{fp}%
}}%
% CIRCLE 1 0 3 0 Black White
% 4 2900 700 2900 600 2900 600 2800 700
% 
{\color[named]{Black}{%
\special{pn 13}%
\special{ar 2900 700 100 100  3.1415927 4.7123890}%
}}%
% LINE 1 0 3 0 Black White
% 2 2900 600 3300 600
% 
{\color[named]{Black}{%
\special{pn 13}%
\special{pa 2900 600}%
\special{pa 3300 600}%
\special{fp}%
}}%
% CIRCLE 1 0 3 0 Black White
% 4 3300 500 3300 600 3300 600 3400 500
% 
{\color[named]{Black}{%
\special{pn 13}%
\special{ar 3300 500 100 100  6.2831853 6.2831853}%
\special{ar 3300 500 100 100  0.0000000 1.5707963}%
}}%
% DOT 0 0 3 0 Black White
% 1 3400 500
% 
{\color[named]{Black}{%
\special{pn 4}%
\special{sh 1}%
\special{ar 3400 500 26 26 0  6.28318530717959E+0000}%
}}%
% LINE 2 0 3 0 Black White
% 4 3100 400 3130 340 3100 400 3070 340
% 
{\color[named]{Black}{%
\special{pn 8}%
\special{pa 3100 400}%
\special{pa 3130 340}%
\special{fp}%
\special{pa 3100 400}%
\special{pa 3070 340}%
\special{fp}%
}}%
% LINE 2 0 3 0 Black White
% 4 3500 500 3560 470 3500 500 3560 530
% 
{\color[named]{Black}{%
\special{pn 8}%
\special{pa 3500 500}%
\special{pa 3560 470}%
\special{fp}%
\special{pa 3500 500}%
\special{pa 3560 530}%
\special{fp}%
}}%
% STR 2 0 3 0 Black White
% 4 3070 85 3070 135 2 0 0 0
% $c$
\put(30.7000,-1.3500){\makebox(0,0)[lb]{$c$}}%
% STR 2 0 3 0 Black White
% 4 3575 410 3575 460 2 0 0 0
% $b$
\put(35.7500,-4.6000){\makebox(0,0)[lb]{$b$}}%
% STR 2 0 3 0 Black White
% 4 1175 410 1175 460 2 0 0 0
% $b$
\put(11.7500,-4.6000){\makebox(0,0)[lb]{$b$}}%
% STR 2 0 3 0 Black White
% 4 875 935 875 985 2 0 0 0
% $G$
\put(8.7500,-9.8500){\makebox(0,0)[lb]{$G$}}%
% STR 2 0 3 0 Black White
% 4 2975 935 2975 985 2 0 0 0
% $G^{\prime}$
\put(29.7500,-9.8500){\makebox(0,0)[lb]{$G^{\prime}$}}%
% STR 2 0 3 0 Black White
% 4 585 665 585 715 2 0 0 0
% $t$
\put(5.8500,-7.1500){\makebox(0,0)[lb]{$t$}}%
% STR 2 0 3 0 Black White
% 4 630 400 630 450 2 0 0 0
% $v_1$
\put(6.3000,-4.5000){\makebox(0,0)[lb]{$v_1$}}%
% STR 2 0 3 0 Black White
% 4 995 575 995 625 2 0 0 0
% $v_2$
\put(9.9500,-6.2500){\makebox(0,0)[lb]{$v_2$}}%
% STR 2 0 3 0 Black White
% 4 3395 635 3395 685 2 0 0 0
% $v_1^{\prime}$
\put(33.9500,-6.8500){\makebox(0,0)[lb]{$v_1^{\prime}$}}%
% STR 2 0 3 0 Black White
% 4 2905 425 2905 475 2 0 0 0
% $v_2^{\prime}$
\put(29.0500,-4.7500){\makebox(0,0)[lb]{$v_2^{\prime}$}}%
% STR 2 0 3 0 Black White
% 4 970 85 970 135 2 0 0 0
% $c$
\put(9.7000,-1.3500){\makebox(0,0)[lb]{$c$}}%
% LINE 2 0 3 0 Black White
% 4 840 500 780 470 840 500 780 530
% 
{\color[named]{Black}{%
\special{pn 8}%
\special{pa 840 500}%
\special{pa 780 470}%
\special{fp}%
\special{pa 840 500}%
\special{pa 780 530}%
\special{fp}%
}}%
\end{picture}%

%% file: chorddiagram.tex
%WinTpicVersion4.10
\unitlength 0.1in
\begin{picture}( 34.4000,  9.5400)(  4.0000,-11.0000)
% CIRCLE 2 0 3 0 Black White
% 4 1200 832 1680 832 1680 832 880 832
% 
{\color[named]{Black}{%
\special{pn 8}%
\special{ar 1200 832 480 480  3.1415927 6.2831853}%
}}%
% LINE 2 0 3 0 Black White
% 2 400 832 2000 832
% 
{\color[named]{Black}{%
\special{pn 8}%
\special{pa 400 832}%
\special{pa 2000 832}%
\special{fp}%
}}%
% CIRCLE 2 0 3 0 Black White
% 4 3040 832 3520 832 3520 832 2720 832
% 
{\color[named]{Black}{%
\special{pn 8}%
\special{ar 3040 832 480 480  3.1415927 6.2831853}%
}}%
% LINE 2 0 3 0 Black White
% 2 2240 832 3840 832
% 
{\color[named]{Black}{%
\special{pn 8}%
\special{pa 2240 832}%
\special{pa 3840 832}%
\special{fp}%
}}%
% DOT 0 0 3 0 Black White
% 1 720 832
% 
{\color[named]{Black}{%
\special{pn 4}%
\special{sh 1}%
\special{ar 720 832 26 26 0  6.28318530717959E+0000}%
}}%
% DOT 0 0 3 0 Black White
% 1 1680 832
% 
{\color[named]{Black}{%
\special{pn 4}%
\special{sh 1}%
\special{ar 1680 832 26 26 0  6.28318530717959E+0000}%
}}%
% STR 2 0 3 0 Black White
% 4 668 986 668 1026 2 0 0 0
% $v^{\prime}_k$
\put(6.6800,-10.2600){\makebox(0,0)[lb]{$v^{\prime}_k$}}%
% STR 2 0 3 0 Black White
% 4 1641 956 1641 996 2 0 0 0
% $v_k$
\put(16.4100,-9.9600){\makebox(0,0)[lb]{$v_k$}}%
% LINE 2 0 3 0 Black White
% 4 1200 352 1280 312 1200 352 1280 392
% 
{\color[named]{Black}{%
\special{pn 8}%
\special{pa 1200 352}%
\special{pa 1280 312}%
\special{fp}%
\special{pa 1200 352}%
\special{pa 1280 392}%
\special{fp}%
}}%
% LINE 2 0 3 0 Black White
% 4 3040 352 2960 312 3040 352 2960 392
% 
{\color[named]{Black}{%
\special{pn 8}%
\special{pa 3040 352}%
\special{pa 2960 312}%
\special{fp}%
\special{pa 3040 352}%
\special{pa 2960 392}%
\special{fp}%
}}%
% LINE 2 0 3 0 Black White
% 4 1840 832 1760 792 1840 832 1760 872
% 
{\color[named]{Black}{%
\special{pn 8}%
\special{pa 1840 832}%
\special{pa 1760 792}%
\special{fp}%
\special{pa 1840 832}%
\special{pa 1760 872}%
\special{fp}%
}}%
% STR 2 0 3 0 Black White
% 4 1720 720 1720 760 2 0 0 0
% $e_{i_k}$
\put(17.2000,-7.6000){\makebox(0,0)[lb]{$e_{i_k}$}}%
% LINE 2 0 3 0 Black White
% 4 640 832 560 792 640 832 560 872
% 
{\color[named]{Black}{%
\special{pn 8}%
\special{pa 640 832}%
\special{pa 560 792}%
\special{fp}%
\special{pa 640 832}%
\special{pa 560 872}%
\special{fp}%
}}%
% STR 2 0 3 0 Black White
% 4 528 724 528 764 2 0 0 0
% $e_{i^{\prime}_k}$
\put(5.2800,-7.6400){\makebox(0,0)[lb]{$e_{i^{\prime}_k}$}}%
% STR 2 0 3 0 Black White
% 4 1168 236 1168 276 2 0 0 0
% $f_k$
\put(11.6800,-2.7600){\makebox(0,0)[lb]{$f_k$}}%
% STR 2 0 3 0 Black White
% 4 3008 236 3008 276 2 0 0 0
% $f_k$
\put(30.0800,-2.7600){\makebox(0,0)[lb]{$f_k$}}%
% DOT 0 0 3 0 Black White
% 1 2560 832
% 
{\color[named]{Black}{%
\special{pn 4}%
\special{sh 1}%
\special{ar 2560 832 26 26 0  6.28318530717959E+0000}%
}}%
% DOT 0 0 3 0 Black White
% 1 3520 832
% 
{\color[named]{Black}{%
\special{pn 4}%
\special{sh 1}%
\special{ar 3520 832 26 26 0  6.28318530717959E+0000}%
}}%
% LINE 2 0 3 0 Black White
% 4 2720 832 2640 792 2720 832 2640 872
% 
{\color[named]{Black}{%
\special{pn 8}%
\special{pa 2720 832}%
\special{pa 2640 792}%
\special{fp}%
\special{pa 2720 832}%
\special{pa 2640 872}%
\special{fp}%
}}%
% LINE 2 0 3 0 Black White
% 4 3440 832 3360 792 3440 832 3360 872
% 
{\color[named]{Black}{%
\special{pn 8}%
\special{pa 3440 832}%
\special{pa 3360 792}%
\special{fp}%
\special{pa 3440 832}%
\special{pa 3360 872}%
\special{fp}%
}}%
% STR 2 0 3 0 Black White
% 4 2521 956 2521 996 2 0 0 0
% $v_k$
\put(25.2100,-9.9600){\makebox(0,0)[lb]{$v_k$}}%
% STR 2 0 3 0 Black White
% 4 3468 986 3468 1026 2 0 0 0
% $v^{\prime}_k$
\put(34.6800,-10.2600){\makebox(0,0)[lb]{$v^{\prime}_k$}}%
% STR 2 0 3 0 Black White
% 4 3321 740 3321 780 2 0 0 0
% $e_{i^{\prime}_k}$
\put(33.2100,-7.8000){\makebox(0,0)[lb]{$e_{i^{\prime}_k}$}}%
% STR 2 0 3 0 Black White
% 4 2616 740 2616 780 2 0 0 0
% $e_{i_k}$
\put(26.1600,-7.8000){\makebox(0,0)[lb]{$e_{i_k}$}}%
% DOT 0 0 3 0 Black White
% 1 3280 832
% 
{\color[named]{Black}{%
\special{pn 4}%
\special{sh 1}%
\special{ar 3280 832 26 26 0  6.28318530717959E+0000}%
}}%
% DOT 0 0 3 0 Black White
% 1 2800 832
% 
{\color[named]{Black}{%
\special{pn 4}%
\special{sh 1}%
\special{ar 2800 832 26 26 0  6.28318530717959E+0000}%
}}%
% DOT 0 0 3 0 Black White
% 1 1920 832
% 
{\color[named]{Black}{%
\special{pn 4}%
\special{sh 1}%
\special{ar 1920 832 26 26 0  6.28318530717959E+0000}%
}}%
% DOT 0 0 3 0 Black White
% 1 480 832
% 
{\color[named]{Black}{%
\special{pn 4}%
\special{sh 1}%
\special{ar 480 832 26 26 0  6.28318530717959E+0000}%
}}%
% STR 2 0 3 0 Black White
% 4 2580 1190 2580 1230 2 0 0 0
% the case $i_k<i^{\prime}_k$
\put(25.8000,-12.3000){\makebox(0,0)[lb]{the case $i_k<i^{\prime}_k$}}%
% STR 2 0 3 0 Black White
% 4 720 1190 720 1230 2 0 0 0
% the case $i_k>i^{\prime}_k$
\put(7.2000,-12.3000){\makebox(0,0)[lb]{the case $i_k>i^{\prime}_k$}}%
\end{picture}%

%% file: primitive.tex
%WinTpicVersion4.10
\unitlength 0.1in
\begin{picture}( 45.3200, 23.0300)(  2.7000,-24.5500)
% POLYGON 2 0 3 0 Black White
% 28 1386 340 1602 340 1812 390 2005 487 2171 626 2299 798 2385 997 2423 1210 2410 1426 2348 1633 2240 1820 2092 1976 1911 2095 1708 2169 1494 2194 1280 2169 1076 2095 895 1976 748 1820 640 1633 578 1426 565 1210 603 997 688 798 817 626 983 487 1175 390 1386 340
% 
{\color[named]{Black}{%
\special{pn 8}%
\special{pa 1386 340}%
\special{pa 1602 340}%
\special{pa 1812 390}%
\special{pa 2006 488}%
\special{pa 2172 626}%
\special{pa 2300 798}%
\special{pa 2386 998}%
\special{pa 2424 1210}%
\special{pa 2410 1426}%
\special{pa 2348 1634}%
\special{pa 2240 1820}%
\special{pa 2092 1976}%
\special{pa 1912 2096}%
\special{pa 1708 2170}%
\special{pa 1494 2194}%
\special{pa 1280 2170}%
\special{pa 1076 2096}%
\special{pa 896 1976}%
\special{pa 748 1820}%
\special{pa 640 1634}%
\special{pa 578 1426}%
\special{pa 566 1210}%
\special{pa 604 998}%
\special{pa 688 798}%
\special{pa 818 626}%
\special{pa 984 488}%
\special{pa 1176 390}%
\special{pa 1386 340}%
\special{pa 1602 340}%
\special{fp}%
}}%
% DOT 0 0 3 0 Black White
% 1 1278 2169
% 
{\color[named]{Black}{%
\special{pn 4}%
\special{sh 1}%
\special{ar 1278 2170 26 26 0  6.28318530717959E+0000}%
}}%
% DOT 0 0 3 0 Black White
% 1 1076 2095
% 
{\color[named]{Black}{%
\special{pn 4}%
\special{sh 1}%
\special{ar 1076 2096 26 26 0  6.28318530717959E+0000}%
}}%
% STR 2 0 3 0 Black White
% 4 1110 2228 1110 2242 2 0 0 0
% $t$
\put(11.1000,-22.4200){\makebox(0,0)[lb]{$t$}}%
% DOT 0 0 3 0 Black White
% 1 750 1820
% 
{\color[named]{Black}{%
\special{pn 4}%
\special{sh 1}%
\special{ar 750 1820 26 26 0  6.28318530717959E+0000}%
}}%
% DOT 0 0 3 0 Black White
% 1 640 1632
% 
{\color[named]{Black}{%
\special{pn 4}%
\special{sh 1}%
\special{ar 640 1632 26 26 0  6.28318530717959E+0000}%
}}%
% DOT 0 0 3 0 Black White
% 1 578 1427
% 
{\color[named]{Black}{%
\special{pn 4}%
\special{sh 1}%
\special{ar 578 1428 26 26 0  6.28318530717959E+0000}%
}}%
% DOT 0 0 3 0 Black White
% 1 567 1200
% 
{\color[named]{Black}{%
\special{pn 4}%
\special{sh 1}%
\special{ar 568 1200 26 26 0  6.28318530717959E+0000}%
}}%
% STR 2 0 3 0 Black White
% 4 552 1790 552 1805 2 0 0 0
% $e_n$
\put(5.5200,-18.0500){\makebox(0,0)[lb]{$e_n$}}%
% STR 2 0 3 0 Black White
% 4 452 1574 452 1589 2 0 0 0
% $h^{\prime}$
\put(4.5200,-15.8900){\makebox(0,0)[lb]{$h^{\prime}$}}%
% STR 2 0 3 0 Black White
% 4 270 1367 270 1382 2 0 0 0
% $e_{n+1}$
\put(2.7000,-13.8200){\makebox(0,0)[lb]{$e_{n+1}$}}%
% DOT 0 0 3 0 Black White
% 1 686 797
% 
{\color[named]{Black}{%
\special{pn 4}%
\special{sh 1}%
\special{ar 686 798 26 26 0  6.28318530717959E+0000}%
}}%
% STR 2 0 3 0 Black White
% 4 429 686 429 700 2 0 0 0
% $e_{4g-2}$
\put(4.2900,-7.0000){\makebox(0,0)[lb]{$e_{4g-2}$}}%
% DOT 0 0 3 0 Black White
% 1 819 624
% 
{\color[named]{Black}{%
\special{pn 4}%
\special{sh 1}%
\special{ar 820 624 26 26 0  6.28318530717959E+0000}%
}}%
% DOT 0 0 3 0 Black White
% 1 1180 387
% 
{\color[named]{Black}{%
\special{pn 4}%
\special{sh 1}%
\special{ar 1180 388 26 26 0  6.28318530717959E+0000}%
}}%
% DOT 0 0 3 0 Black White
% 1 1384 340
% 
{\color[named]{Black}{%
\special{pn 4}%
\special{sh 1}%
\special{ar 1384 340 26 26 0  6.28318530717959E+0000}%
}}%
% DOT 0 0 3 0 Black White
% 1 1605 340
% 
{\color[named]{Black}{%
\special{pn 4}%
\special{sh 1}%
\special{ar 1606 340 26 26 0  6.28318530717959E+0000}%
}}%
% STR 2 0 3 0 Black White
% 4 1104 286 1104 301 2 0 0 0
% $e_{n+1}$
\put(11.0400,-3.0100){\makebox(0,0)[lb]{$e_{n+1}$}}%
% STR 2 0 3 0 Black White
% 4 1457 283 1457 297 2 0 0 0
% $f_1$
\put(14.5700,-2.9700){\makebox(0,0)[lb]{$f_1$}}%
% DOT 0 0 3 0 Black White
% 1 2008 488
% 
{\color[named]{Black}{%
\special{pn 4}%
\special{sh 1}%
\special{ar 2008 488 26 26 0  6.28318530717959E+0000}%
}}%
% DOT 0 0 3 0 Black White
% 1 2178 631
% 
{\color[named]{Black}{%
\special{pn 4}%
\special{sh 1}%
\special{ar 2178 632 26 26 0  6.28318530717959E+0000}%
}}%
% DOT 0 0 3 0 Black White
% 1 2300 799
% 
{\color[named]{Black}{%
\special{pn 4}%
\special{sh 1}%
\special{ar 2300 800 26 26 0  6.28318530717959E+0000}%
}}%
% STR 2 0 3 0 Black White
% 4 2109 515 2109 529 2 0 0 0
% $f_1$
\put(21.0900,-5.2900){\makebox(0,0)[lb]{$f_1$}}%
% STR 2 0 3 0 Black White
% 4 2300 708 2300 723 2 0 0 0
% $h^{\prime}$
\put(23.0000,-7.2300){\makebox(0,0)[lb]{$h^{\prime}$}}%
% DOT 0 0 3 0 Black White
% 1 2411 1427
% 
{\color[named]{Black}{%
\special{pn 4}%
\special{sh 1}%
\special{ar 2412 1428 26 26 0  6.28318530717959E+0000}%
}}%
% DOT 0 0 3 0 Black White
% 1 2349 1633
% 
{\color[named]{Black}{%
\special{pn 4}%
\special{sh 1}%
\special{ar 2350 1634 26 26 0  6.28318530717959E+0000}%
}}%
% DOT 0 0 3 0 Black White
% 1 2387 994
% 
{\color[named]{Black}{%
\special{pn 4}%
\special{sh 1}%
\special{ar 2388 994 26 26 0  6.28318530717959E+0000}%
}}%
% DOT 0 0 3 0 Black White
% 1 1913 2097
% 
{\color[named]{Black}{%
\special{pn 4}%
\special{sh 1}%
\special{ar 1914 2098 26 26 0  6.28318530717959E+0000}%
}}%
% DOT 0 0 3 0 Black White
% 1 1710 2170
% 
{\color[named]{Black}{%
\special{pn 4}%
\special{sh 1}%
\special{ar 1710 2170 26 26 0  6.28318530717959E+0000}%
}}%
% STR 2 0 3 0 Black White
% 4 2448 1580 2448 1594 2 0 0 0
% $f_2$
\put(24.4800,-15.9400){\makebox(0,0)[lb]{$f_2$}}%
% STR 2 0 3 0 Black White
% 4 2340 1806 2340 1821 2 0 0 0
% $e_n$
\put(23.4000,-18.2100){\makebox(0,0)[lb]{$e_n$}}%
% STR 2 0 3 0 Black White
% 4 1825 2231 1825 2246 2 0 0 0
% $t$
\put(18.2500,-22.4600){\makebox(0,0)[lb]{$t$}}%
% LINE 2 0 3 0 Black White
% 4 761 700 717 723 761 700 754 750
% 
{\color[named]{Black}{%
\special{pn 8}%
\special{pa 762 700}%
\special{pa 718 724}%
\special{fp}%
\special{pa 762 700}%
\special{pa 754 750}%
\special{fp}%
}}%
% LINE 2 0 3 0 Black White
% 4 1523 344 1481 316 1523 344 1476 362
% 
{\color[named]{Black}{%
\special{pn 8}%
\special{pa 1524 344}%
\special{pa 1482 316}%
\special{fp}%
\special{pa 1524 344}%
\special{pa 1476 362}%
\special{fp}%
}}%
% LINE 2 0 3 0 Black White
% 4 1261 369 1308 381 1261 369 1299 337
% 
{\color[named]{Black}{%
\special{pn 8}%
\special{pa 1262 370}%
\special{pa 1308 382}%
\special{fp}%
\special{pa 1262 370}%
\special{pa 1300 338}%
\special{fp}%
}}%
% LINE 2 0 3 0 Black White
% 4 2351 914 2352 865 2351 914 2310 886
% 
{\color[named]{Black}{%
\special{pn 8}%
\special{pa 2352 914}%
\special{pa 2352 866}%
\special{fp}%
\special{pa 2352 914}%
\special{pa 2310 886}%
\special{fp}%
}}%
% LINE 2 0 3 0 Black White
% 4 2388 1502 2352 1537 2388 1502 2397 1551
% 
{\color[named]{Black}{%
\special{pn 8}%
\special{pa 2388 1502}%
\special{pa 2352 1538}%
\special{fp}%
\special{pa 2388 1502}%
\special{pa 2398 1552}%
\special{fp}%
}}%
% LINE 2 0 3 0 Black White
% 4 2309 1701 2267 1728 2309 1701 2307 1751
% 
{\color[named]{Black}{%
\special{pn 8}%
\special{pa 2310 1702}%
\special{pa 2268 1728}%
\special{fp}%
\special{pa 2310 1702}%
\special{pa 2308 1752}%
\special{fp}%
}}%
% LINE 2 0 3 0 Black White
% 4 2223 696 2234 744 2223 696 2269 715
% 
{\color[named]{Black}{%
\special{pn 8}%
\special{pa 2224 696}%
\special{pa 2234 744}%
\special{fp}%
\special{pa 2224 696}%
\special{pa 2270 716}%
\special{fp}%
}}%
% LINE 2 0 3 0 Black White
% 4 2082 549 2097 596 2082 549 2129 564
% 
{\color[named]{Black}{%
\special{pn 8}%
\special{pa 2082 550}%
\special{pa 2098 596}%
\special{fp}%
\special{pa 2082 550}%
\special{pa 2130 564}%
\special{fp}%
}}%
% LINE 2 0 3 0 Black White
% 4 570 1272 549 1317 570 1272 595 1315
% 
{\color[named]{Black}{%
\special{pn 8}%
\special{pa 570 1272}%
\special{pa 550 1318}%
\special{fp}%
\special{pa 570 1272}%
\special{pa 596 1316}%
\special{fp}%
}}%
% LINE 2 0 3 0 Black White
% 4 599 1507 595 1556 599 1507 637 1539
% 
{\color[named]{Black}{%
\special{pn 8}%
\special{pa 600 1508}%
\special{pa 596 1556}%
\special{fp}%
\special{pa 600 1508}%
\special{pa 638 1540}%
\special{fp}%
}}%
% LINE 2 0 3 0 Black White
% 4 676 1697 682 1748 676 1697 720 1723
% 
{\color[named]{Black}{%
\special{pn 8}%
\special{pa 676 1698}%
\special{pa 682 1748}%
\special{fp}%
\special{pa 676 1698}%
\special{pa 720 1724}%
\special{fp}%
}}%
% LINE 2 0 3 0 Black White
% 4 1160 2122 1189 2164 1160 2122 1210 2122
% 
{\color[named]{Black}{%
\special{pn 8}%
\special{pa 1160 2122}%
\special{pa 1190 2164}%
\special{fp}%
\special{pa 1160 2122}%
\special{pa 1210 2122}%
\special{fp}%
}}%
% LINE 2 0 3 0 Black White
% 4 1827 2127 1776 2119 1827 2127 1792 2163
% 
{\color[named]{Black}{%
\special{pn 8}%
\special{pa 1828 2128}%
\special{pa 1776 2120}%
\special{fp}%
\special{pa 1828 2128}%
\special{pa 1792 2164}%
\special{fp}%
}}%
% SPLINE 2 0 3 0 Black White
% 3 604 1303 1151 929 1288 391
% 
{\color[named]{Black}{%
\special{pn 8}%
\special{pa 604 1304}%
\special{pa 636 1288}%
\special{pa 668 1274}%
\special{pa 760 1230}%
\special{pa 790 1214}%
\special{pa 850 1182}%
\special{pa 880 1166}%
\special{pa 908 1150}%
\special{pa 934 1132}%
\special{pa 962 1114}%
\special{pa 988 1096}%
\special{pa 1012 1078}%
\special{pa 1036 1058}%
\special{pa 1060 1038}%
\special{pa 1082 1016}%
\special{pa 1102 994}%
\special{pa 1122 970}%
\special{pa 1140 948}%
\special{pa 1156 922}%
\special{pa 1172 896}%
\special{pa 1186 870}%
\special{pa 1210 814}%
\special{pa 1220 784}%
\special{pa 1238 724}%
\special{pa 1252 660}%
\special{pa 1264 596}%
\special{pa 1268 562}%
\special{pa 1276 496}%
\special{pa 1288 394}%
\special{pa 1288 392}%
\special{sp}%
}}%
% LINE 2 0 3 0 Black White
% 4 1131 956 1088 981 1131 956 1126 1006
% 
{\color[named]{Black}{%
\special{pn 8}%
\special{pa 1132 956}%
\special{pa 1088 982}%
\special{fp}%
\special{pa 1132 956}%
\special{pa 1126 1006}%
\special{fp}%
}}%
% LINE 2 0 3 0 Black White
% 4 1508 1704 1463 1681 1508 1704 1463 1727
% 
{\color[named]{Black}{%
\special{pn 8}%
\special{pa 1508 1704}%
\special{pa 1464 1682}%
\special{fp}%
\special{pa 1508 1704}%
\special{pa 1464 1728}%
\special{fp}%
}}%
% LINE 2 0 3 0 Black White
% 4 1748 579 1714 542 1748 579 1697 583
% 
{\color[named]{Black}{%
\special{pn 8}%
\special{pa 1748 580}%
\special{pa 1714 542}%
\special{fp}%
\special{pa 1748 580}%
\special{pa 1698 584}%
\special{fp}%
}}%
% LINE 2 0 3 0 Black White
% 4 2168 1222 2189 1177 2168 1222 2144 1178
% 
{\color[named]{Black}{%
\special{pn 8}%
\special{pa 2168 1222}%
\special{pa 2190 1178}%
\special{fp}%
\special{pa 2168 1222}%
\special{pa 2144 1178}%
\special{fp}%
}}%
% STR 2 0 3 0 Black White
% 4 883 902 883 916 2 0 0 0
% $\hat{e}_{n+1}$
\put(8.8300,-9.1600){\makebox(0,0)[lb]{$\hat{e}_{n+1}$}}%
% STR 2 0 3 0 Black White
% 4 1998 1219 1998 1233 2 0 0 0
% $\hat{f}_2$
\put(19.9800,-12.3300){\makebox(0,0)[lb]{$\hat{f}_2$}}%
% STR 2 0 3 0 Black White
% 4 2408 900 2408 914 2 0 0 0
% $f_2$
\put(24.0800,-9.1400){\makebox(0,0)[lb]{$f_2$}}%
% DOT 0 0 3 0 Black White
% 1 2238 1819
% 
{\color[named]{Black}{%
\special{pn 4}%
\special{sh 1}%
\special{ar 2238 1820 26 26 0  6.28318530717959E+0000}%
}}%
% LINE 2 0 3 0 Black White
% 2 712 1704 2274 1704
% 
{\color[named]{Black}{%
\special{pn 8}%
\special{pa 712 1704}%
\special{pa 2274 1704}%
\special{fp}%
}}%
% SPLINE 2 0 3 0 Black White
% 3 2310 905 2166 1222 2351 1515
% 
{\color[named]{Black}{%
\special{pn 8}%
\special{pa 2310 906}%
\special{pa 2290 936}%
\special{pa 2270 964}%
\special{pa 2250 994}%
\special{pa 2232 1024}%
\special{pa 2214 1054}%
\special{pa 2200 1082}%
\special{pa 2186 1112}%
\special{pa 2176 1140}%
\special{pa 2170 1170}%
\special{pa 2166 1198}%
\special{pa 2166 1226}%
\special{pa 2172 1254}%
\special{pa 2180 1282}%
\special{pa 2192 1308}%
\special{pa 2206 1336}%
\special{pa 2222 1364}%
\special{pa 2242 1390}%
\special{pa 2264 1416}%
\special{pa 2310 1470}%
\special{pa 2334 1496}%
\special{pa 2352 1516}%
\special{sp}%
}}%
% SPLINE 2 0 3 0 Black White
% 3 1491 369 1722 571 2067 571
% 
{\color[named]{Black}{%
\special{pn 8}%
\special{pa 1492 370}%
\special{pa 1514 396}%
\special{pa 1536 422}%
\special{pa 1558 446}%
\special{pa 1582 470}%
\special{pa 1606 492}%
\special{pa 1630 514}%
\special{pa 1654 534}%
\special{pa 1680 550}%
\special{pa 1706 564}%
\special{pa 1734 576}%
\special{pa 1762 584}%
\special{pa 1792 590}%
\special{pa 1824 594}%
\special{pa 1854 594}%
\special{pa 1888 594}%
\special{pa 1920 592}%
\special{pa 1954 588}%
\special{pa 1986 584}%
\special{pa 2020 580}%
\special{pa 2054 574}%
\special{pa 2068 572}%
\special{sp}%
}}%
% STR 2 0 3 0 Black White
% 4 1677 753 1677 767 2 0 0 0
% $\hat{f}_1$
\put(16.7700,-7.6700){\makebox(0,0)[lb]{$\hat{f}_1$}}%
% STR 2 0 3 0 Black White
% 4 1479 1622 1479 1636 2 0 0 0
% $\hat{e}_n$
\put(14.7900,-16.3600){\makebox(0,0)[lb]{$\hat{e}_n$}}%
% POLYGON 2 0 3 0 Black White
% 28 3726 340 3942 340 4152 390 4345 487 4511 626 4639 799 4725 997 4763 1210 4750 1426 4688 1633 4580 1820 4432 1976 4251 2095 4048 2170 3834 2194 3620 2170 3416 2095 3235 1976 3088 1820 2980 1633 2918 1426 2905 1210 2943 997 3028 799 3157 626 3323 487 3515 390 3726 340
% 
{\color[named]{Black}{%
\special{pn 8}%
\special{pa 3726 340}%
\special{pa 3942 340}%
\special{pa 4152 390}%
\special{pa 4346 488}%
\special{pa 4512 626}%
\special{pa 4640 800}%
\special{pa 4726 998}%
\special{pa 4764 1210}%
\special{pa 4750 1426}%
\special{pa 4688 1634}%
\special{pa 4580 1820}%
\special{pa 4432 1976}%
\special{pa 4252 2096}%
\special{pa 4048 2170}%
\special{pa 3834 2194}%
\special{pa 3620 2170}%
\special{pa 3416 2096}%
\special{pa 3236 1976}%
\special{pa 3088 1820}%
\special{pa 2980 1634}%
\special{pa 2918 1426}%
\special{pa 2906 1210}%
\special{pa 2944 998}%
\special{pa 3028 800}%
\special{pa 3158 626}%
\special{pa 3324 488}%
\special{pa 3516 390}%
\special{pa 3726 340}%
\special{pa 3942 340}%
\special{fp}%
}}%
% DOT 0 0 3 0 Black White
% 1 3618 2169
% 
{\color[named]{Black}{%
\special{pn 4}%
\special{sh 1}%
\special{ar 3618 2170 26 26 0  6.28318530717959E+0000}%
}}%
% DOT 0 0 3 0 Black White
% 1 3416 2095
% 
{\color[named]{Black}{%
\special{pn 4}%
\special{sh 1}%
\special{ar 3416 2096 26 26 0  6.28318530717959E+0000}%
}}%
% STR 2 0 3 0 Black White
% 4 3450 2228 3450 2243 2 0 0 0
% $t$
\put(34.5000,-22.4300){\makebox(0,0)[lb]{$t$}}%
% DOT 0 0 3 0 Black White
% 1 3090 1821
% 
{\color[named]{Black}{%
\special{pn 4}%
\special{sh 1}%
\special{ar 3090 1822 26 26 0  6.28318530717959E+0000}%
}}%
% DOT 0 0 3 0 Black White
% 1 2980 1632
% 
{\color[named]{Black}{%
\special{pn 4}%
\special{sh 1}%
\special{ar 2980 1632 26 26 0  6.28318530717959E+0000}%
}}%
% DOT 0 0 3 0 Black White
% 1 2918 1427
% 
{\color[named]{Black}{%
\special{pn 4}%
\special{sh 1}%
\special{ar 2918 1428 26 26 0  6.28318530717959E+0000}%
}}%
% DOT 0 0 3 0 Black White
% 1 2907 1200
% 
{\color[named]{Black}{%
\special{pn 4}%
\special{sh 1}%
\special{ar 2908 1200 26 26 0  6.28318530717959E+0000}%
}}%
% STR 2 0 3 0 Black White
% 4 2890 1790 2890 1805 2 0 0 0
% $e_n$
\put(28.9000,-18.0500){\makebox(0,0)[lb]{$e_n$}}%
% STR 2 0 3 0 Black White
% 4 2792 1574 2792 1589 2 0 0 0
% $h^{\prime}$
\put(27.9200,-15.8900){\makebox(0,0)[lb]{$h^{\prime}$}}%
% STR 2 0 3 0 Black White
% 4 2606 1368 2606 1382 2 0 0 0
% $e_{n+1}$
\put(26.0600,-13.8200){\makebox(0,0)[lb]{$e_{n+1}$}}%
% DOT 0 0 3 0 Black White
% 1 3026 798
% 
{\color[named]{Black}{%
\special{pn 4}%
\special{sh 1}%
\special{ar 3026 798 26 26 0  6.28318530717959E+0000}%
}}%
% STR 2 0 3 0 Black White
% 4 2773 687 2773 701 2 0 0 0
% $e_{4g-2}$
\put(27.7300,-7.0100){\makebox(0,0)[lb]{$e_{4g-2}$}}%
% DOT 0 0 3 0 Black White
% 1 3159 624
% 
{\color[named]{Black}{%
\special{pn 4}%
\special{sh 1}%
\special{ar 3160 624 26 26 0  6.28318530717959E+0000}%
}}%
% DOT 0 0 3 0 Black White
% 1 3520 388
% 
{\color[named]{Black}{%
\special{pn 4}%
\special{sh 1}%
\special{ar 3520 388 26 26 0  6.28318530717959E+0000}%
}}%
% DOT 0 0 3 0 Black White
% 1 3724 340
% 
{\color[named]{Black}{%
\special{pn 4}%
\special{sh 1}%
\special{ar 3724 340 26 26 0  6.28318530717959E+0000}%
}}%
% DOT 0 0 3 0 Black White
% 1 3945 340
% 
{\color[named]{Black}{%
\special{pn 4}%
\special{sh 1}%
\special{ar 3946 340 26 26 0  6.28318530717959E+0000}%
}}%
% STR 2 0 3 0 Black White
% 4 3444 286 3444 301 2 0 0 0
% $e_{n+1}$
\put(34.4400,-3.0100){\makebox(0,0)[lb]{$e_{n+1}$}}%
% STR 2 0 3 0 Black White
% 4 3797 283 3797 297 2 0 0 0
% $f_1$
\put(37.9700,-2.9700){\makebox(0,0)[lb]{$f_1$}}%
% DOT 0 0 3 0 Black White
% 1 4348 489
% 
{\color[named]{Black}{%
\special{pn 4}%
\special{sh 1}%
\special{ar 4348 490 26 26 0  6.28318530717959E+0000}%
}}%
% DOT 0 0 3 0 Black White
% 1 4518 631
% 
{\color[named]{Black}{%
\special{pn 4}%
\special{sh 1}%
\special{ar 4518 632 26 26 0  6.28318530717959E+0000}%
}}%
% DOT 0 0 3 0 Black White
% 1 4640 799
% 
{\color[named]{Black}{%
\special{pn 4}%
\special{sh 1}%
\special{ar 4640 800 26 26 0  6.28318530717959E+0000}%
}}%
% STR 2 0 3 0 Black White
% 4 4449 515 4449 529 2 0 0 0
% $f_2$
\put(44.4900,-5.2900){\makebox(0,0)[lb]{$f_2$}}%
% STR 2 0 3 0 Black White
% 4 4640 709 4640 724 2 0 0 0
% $e_n$
\put(46.4000,-7.2400){\makebox(0,0)[lb]{$e_n$}}%
% DOT 0 0 3 0 Black White
% 1 4761 1202
% 
{\color[named]{Black}{%
\special{pn 4}%
\special{sh 1}%
\special{ar 4762 1202 26 26 0  6.28318530717959E+0000}%
}}%
% DOT 0 0 3 0 Black White
% 1 4751 1427
% 
{\color[named]{Black}{%
\special{pn 4}%
\special{sh 1}%
\special{ar 4752 1428 26 26 0  6.28318530717959E+0000}%
}}%
% DOT 0 0 3 0 Black White
% 1 4689 1634
% 
{\color[named]{Black}{%
\special{pn 4}%
\special{sh 1}%
\special{ar 4690 1634 26 26 0  6.28318530717959E+0000}%
}}%
% DOT 0 0 3 0 Black White
% 1 4584 1816
% 
{\color[named]{Black}{%
\special{pn 4}%
\special{sh 1}%
\special{ar 4584 1816 26 26 0  6.28318530717959E+0000}%
}}%
% DOT 0 0 3 0 Black White
% 1 4253 2097
% 
{\color[named]{Black}{%
\special{pn 4}%
\special{sh 1}%
\special{ar 4254 2098 26 26 0  6.28318530717959E+0000}%
}}%
% DOT 0 0 3 0 Black White
% 1 4050 2171
% 
{\color[named]{Black}{%
\special{pn 4}%
\special{sh 1}%
\special{ar 4050 2172 26 26 0  6.28318530717959E+0000}%
}}%
% STR 2 0 3 0 Black White
% 4 4802 1343 4802 1357 2 0 0 0
% $f_1$
\put(48.0200,-13.5700){\makebox(0,0)[lb]{$f_1$}}%
% STR 2 0 3 0 Black White
% 4 4788 1580 4788 1595 2 0 0 0
% $h^{\prime}$
\put(47.8800,-15.9500){\makebox(0,0)[lb]{$h^{\prime}$}}%
% STR 2 0 3 0 Black White
% 4 4680 1807 4680 1822 2 0 0 0
% $f_2$
\put(46.8000,-18.2200){\makebox(0,0)[lb]{$f_2$}}%
% STR 2 0 3 0 Black White
% 4 4165 2232 4165 2246 2 0 0 0
% $t$
\put(41.6500,-22.4600){\makebox(0,0)[lb]{$t$}}%
% LINE 2 0 3 0 Black White
% 4 3101 700 3057 724 3101 700 3094 750
% 
{\color[named]{Black}{%
\special{pn 8}%
\special{pa 3102 700}%
\special{pa 3058 724}%
\special{fp}%
\special{pa 3102 700}%
\special{pa 3094 750}%
\special{fp}%
}}%
% LINE 2 0 3 0 Black White
% 4 3863 344 3821 317 3863 344 3816 362
% 
{\color[named]{Black}{%
\special{pn 8}%
\special{pa 3864 344}%
\special{pa 3822 318}%
\special{fp}%
\special{pa 3864 344}%
\special{pa 3816 362}%
\special{fp}%
}}%
% LINE 2 0 3 0 Black White
% 4 3601 369 3648 382 3601 369 3639 337
% 
{\color[named]{Black}{%
\special{pn 8}%
\special{pa 3602 370}%
\special{pa 3648 382}%
\special{fp}%
\special{pa 3602 370}%
\special{pa 3640 338}%
\special{fp}%
}}%
% LINE 2 0 3 0 Black White
% 4 4757 1296 4735 1341 4757 1296 4780 1341
% 
{\color[named]{Black}{%
\special{pn 8}%
\special{pa 4758 1296}%
\special{pa 4736 1342}%
\special{fp}%
\special{pa 4758 1296}%
\special{pa 4780 1342}%
\special{fp}%
}}%
% LINE 2 0 3 0 Black White
% 4 4728 1502 4692 1537 4728 1502 4737 1552
% 
{\color[named]{Black}{%
\special{pn 8}%
\special{pa 4728 1502}%
\special{pa 4692 1538}%
\special{fp}%
\special{pa 4728 1502}%
\special{pa 4738 1552}%
\special{fp}%
}}%
% LINE 2 0 3 0 Black White
% 4 4649 1701 4607 1728 4649 1701 4647 1751
% 
{\color[named]{Black}{%
\special{pn 8}%
\special{pa 4650 1702}%
\special{pa 4608 1728}%
\special{fp}%
\special{pa 4650 1702}%
\special{pa 4648 1752}%
\special{fp}%
}}%
% LINE 2 0 3 0 Black White
% 4 4563 696 4574 745 4563 696 4609 716
% 
{\color[named]{Black}{%
\special{pn 8}%
\special{pa 4564 696}%
\special{pa 4574 746}%
\special{fp}%
\special{pa 4564 696}%
\special{pa 4610 716}%
\special{fp}%
}}%
% LINE 2 0 3 0 Black White
% 4 4441 568 4422 522 4441 568 4393 555
% 
{\color[named]{Black}{%
\special{pn 8}%
\special{pa 4442 568}%
\special{pa 4422 522}%
\special{fp}%
\special{pa 4442 568}%
\special{pa 4394 556}%
\special{fp}%
}}%
% LINE 2 0 3 0 Black White
% 4 2910 1272 2889 1318 2910 1272 2935 1315
% 
{\color[named]{Black}{%
\special{pn 8}%
\special{pa 2910 1272}%
\special{pa 2890 1318}%
\special{fp}%
\special{pa 2910 1272}%
\special{pa 2936 1316}%
\special{fp}%
}}%
% LINE 2 0 3 0 Black White
% 4 2939 1507 2935 1556 2939 1507 2977 1540
% 
{\color[named]{Black}{%
\special{pn 8}%
\special{pa 2940 1508}%
\special{pa 2936 1556}%
\special{fp}%
\special{pa 2940 1508}%
\special{pa 2978 1540}%
\special{fp}%
}}%
% LINE 2 0 3 0 Black White
% 4 3016 1698 3022 1748 3016 1698 3060 1723
% 
{\color[named]{Black}{%
\special{pn 8}%
\special{pa 3016 1698}%
\special{pa 3022 1748}%
\special{fp}%
\special{pa 3016 1698}%
\special{pa 3060 1724}%
\special{fp}%
}}%
% LINE 2 0 3 0 Black White
% 4 3500 2123 3529 2164 3500 2123 3550 2123
% 
{\color[named]{Black}{%
\special{pn 8}%
\special{pa 3500 2124}%
\special{pa 3530 2164}%
\special{fp}%
\special{pa 3500 2124}%
\special{pa 3550 2124}%
\special{fp}%
}}%
% LINE 2 0 3 0 Black White
% 4 4167 2128 4116 2120 4167 2128 4132 2164
% 
{\color[named]{Black}{%
\special{pn 8}%
\special{pa 4168 2128}%
\special{pa 4116 2120}%
\special{fp}%
\special{pa 4168 2128}%
\special{pa 4132 2164}%
\special{fp}%
}}%
% SPLINE 2 0 3 0 Black White
% 3 3831 374 3975 1037 4724 1314
% 
{\color[named]{Black}{%
\special{pn 8}%
\special{pa 3832 374}%
\special{pa 3832 410}%
\special{pa 3834 446}%
\special{pa 3836 482}%
\special{pa 3838 518}%
\special{pa 3842 588}%
\special{pa 3846 622}%
\special{pa 3854 690}%
\special{pa 3858 724}%
\special{pa 3870 788}%
\special{pa 3876 818}%
\special{pa 3884 848}%
\special{pa 3894 878}%
\special{pa 3904 906}%
\special{pa 3914 934}%
\special{pa 3926 960}%
\special{pa 3940 986}%
\special{pa 3956 1010}%
\special{pa 3972 1034}%
\special{pa 3990 1054}%
\special{pa 4010 1076}%
\special{pa 4030 1094}%
\special{pa 4052 1112}%
\special{pa 4076 1128}%
\special{pa 4100 1144}%
\special{pa 4126 1160}%
\special{pa 4152 1174}%
\special{pa 4180 1186}%
\special{pa 4210 1198}%
\special{pa 4270 1220}%
\special{pa 4302 1230}%
\special{pa 4334 1240}%
\special{pa 4366 1248}%
\special{pa 4400 1256}%
\special{pa 4434 1264}%
\special{pa 4538 1284}%
\special{pa 4646 1302}%
\special{pa 4682 1308}%
\special{pa 4724 1314}%
\special{sp}%
}}%
% SPLINE 2 0 3 0 Black White
% 3 4410 576 4323 1152 4609 1708
% 
{\color[named]{Black}{%
\special{pn 8}%
\special{pa 4410 576}%
\special{pa 4402 608}%
\special{pa 4392 642}%
\special{pa 4374 706}%
\special{pa 4364 738}%
\special{pa 4356 770}%
\special{pa 4336 866}%
\special{pa 4330 898}%
\special{pa 4326 930}%
\special{pa 4322 962}%
\special{pa 4318 992}%
\special{pa 4316 1024}%
\special{pa 4316 1086}%
\special{pa 4320 1116}%
\special{pa 4322 1146}%
\special{pa 4328 1176}%
\special{pa 4334 1206}%
\special{pa 4342 1236}%
\special{pa 4352 1266}%
\special{pa 4362 1294}%
\special{pa 4374 1324}%
\special{pa 4386 1352}%
\special{pa 4400 1380}%
\special{pa 4414 1408}%
\special{pa 4430 1438}%
\special{pa 4446 1466}%
\special{pa 4480 1522}%
\special{pa 4500 1550}%
\special{pa 4518 1578}%
\special{pa 4536 1604}%
\special{pa 4556 1632}%
\special{pa 4576 1660}%
\special{pa 4594 1688}%
\special{pa 4610 1708}%
\special{sp}%
}}%
% LINE 2 0 3 0 Black White
% 2 4554 734 3052 1696
% 
{\color[named]{Black}{%
\special{pn 8}%
\special{pa 4554 734}%
\special{pa 3052 1696}%
\special{fp}%
}}%
% SPLINE 2 0 3 0 Black White
% 3 2944 1303 3491 929 3628 391
% 
{\color[named]{Black}{%
\special{pn 8}%
\special{pa 2944 1304}%
\special{pa 2976 1288}%
\special{pa 3008 1274}%
\special{pa 3100 1230}%
\special{pa 3130 1214}%
\special{pa 3190 1182}%
\special{pa 3220 1166}%
\special{pa 3248 1150}%
\special{pa 3274 1132}%
\special{pa 3302 1114}%
\special{pa 3328 1096}%
\special{pa 3352 1078}%
\special{pa 3376 1058}%
\special{pa 3400 1038}%
\special{pa 3422 1016}%
\special{pa 3442 994}%
\special{pa 3462 970}%
\special{pa 3480 948}%
\special{pa 3496 922}%
\special{pa 3512 896}%
\special{pa 3526 870}%
\special{pa 3550 814}%
\special{pa 3560 784}%
\special{pa 3578 724}%
\special{pa 3592 660}%
\special{pa 3604 596}%
\special{pa 3608 562}%
\special{pa 3616 496}%
\special{pa 3628 394}%
\special{pa 3628 392}%
\special{sp}%
}}%
% LINE 2 0 3 0 Black White
% 4 3471 957 3428 982 3471 957 3466 1007
% 
{\color[named]{Black}{%
\special{pn 8}%
\special{pa 3472 958}%
\special{pa 3428 982}%
\special{fp}%
\special{pa 3472 958}%
\special{pa 3466 1008}%
\special{fp}%
}}%
% LINE 2 0 3 0 Black White
% 4 3630 1324 3580 1332 3630 1324 3606 1369
% 
{\color[named]{Black}{%
\special{pn 8}%
\special{pa 3630 1324}%
\special{pa 3580 1332}%
\special{fp}%
\special{pa 3630 1324}%
\special{pa 3606 1370}%
\special{fp}%
}}%
% LINE 2 0 3 0 Black White
% 4 3873 801 3882 752 3873 801 3838 763
% 
{\color[named]{Black}{%
\special{pn 8}%
\special{pa 3874 802}%
\special{pa 3882 752}%
\special{fp}%
\special{pa 3874 802}%
\special{pa 3838 764}%
\special{fp}%
}}%
% LINE 2 0 3 0 Black White
% 4 4472 1508 4467 1458 4472 1508 4429 1481
% 
{\color[named]{Black}{%
\special{pn 8}%
\special{pa 4472 1508}%
\special{pa 4468 1458}%
\special{fp}%
\special{pa 4472 1508}%
\special{pa 4430 1482}%
\special{fp}%
}}%
% STR 2 0 3 0 Black White
% 4 3223 902 3223 916 2 0 0 0
% $\hat{e}_{n+1}$
\put(32.2300,-9.1600){\makebox(0,0)[lb]{$\hat{e}_{n+1}$}}%
% STR 2 0 3 0 Black White
% 4 3612 1463 3612 1478 2 0 0 0
% $\hat{e}_n$
\put(36.1200,-14.7800){\makebox(0,0)[lb]{$\hat{e}_n$}}%
% STR 2 0 3 0 Black White
% 4 4292 1535 4292 1550 2 0 0 0
% $\hat{f}_2$
\put(42.9200,-15.5000){\makebox(0,0)[lb]{$\hat{f}_2$}}%
% STR 2 0 3 0 Black White
% 4 3919 761 3919 776 2 0 0 0
% $\hat{f}_1$
\put(39.1900,-7.7600){\makebox(0,0)[lb]{$\hat{f}_1$}}%
% STR 2 0 3 0 Black White
% 4 1000 2550 1000 2600 2 0 0 0
% the case $\bar{f}_1<\bar{e}_n$
\put(10.0000,-26.0000){\makebox(0,0)[lb]{the case $\bar{f}_1<\bar{e}_n$}}%
% STR 2 0 3 0 Black White
% 4 3400 2550 3400 2600 2 0 0 0
% the case $\bar{e}_n<\bar{f}_1$
\put(34.0000,-26.0000){\makebox(0,0)[lb]{the case $\bar{e}_n<\bar{f}_1$}}%
\end{picture}%

%% file: punctured_trivial.tex
%WinTpicVersion4.10
\unitlength 0.1in
\begin{picture}( 37.5000,  8.0000)(  2.0000,-13.0000)
% LINE 2 0 3 0 Black White
% 2 600 600 1600 600
% 
{\color[named]{Black}{%
\special{pn 8}%
\special{pa 600 600}%
\special{pa 1600 600}%
\special{fp}%
}}%
% LINE 2 0 3 0 Black White
% 2 600 1200 1600 1200
% 
{\color[named]{Black}{%
\special{pn 8}%
\special{pa 600 1200}%
\special{pa 1600 1200}%
\special{fp}%
}}%
% DOT 0 0 3 0 Black White
% 1 600 600
% 
{\color[named]{Black}{%
\special{pn 4}%
\special{sh 1}%
\special{ar 600 600 26 26 0  6.28318530717959E+0000}%
}}%
% DOT 0 0 3 0 Black White
% 1 600 1200
% 
{\color[named]{Black}{%
\special{pn 4}%
\special{sh 1}%
\special{ar 600 1200 26 26 0  6.28318530717959E+0000}%
}}%
% DOT 0 0 3 0 Black White
% 1 1100 600
% 
{\color[named]{Black}{%
\special{pn 4}%
\special{sh 1}%
\special{ar 1100 600 26 26 0  6.28318530717959E+0000}%
}}%
% DOT 0 0 3 0 Black White
% 1 1000 1200
% 
{\color[named]{Black}{%
\special{pn 4}%
\special{sh 1}%
\special{ar 1000 1200 26 26 0  6.28318530717959E+0000}%
}}%
% CIRCLE 2 0 3 0 Black White
% 4 1100 1200 1000 1200 1000 1200 1200 1230
% 
{\color[named]{Black}{%
\special{pn 8}%
\special{ar 1100 1200 100 100  0.2914568 3.1415927}%
}}%
% ELLIPSE 2 0 3 0 Black White
% 4 1100 1200 1200 900 1200 1175 1100 775
% 
{\color[named]{Black}{%
\special{pn 8}%
\special{ar 1100 1200 100 300  4.7123890 6.2033553}%
}}%
% LINE 2 0 3 0 Black White
% 2 1100 900 1100 600
% 
{\color[named]{Black}{%
\special{pn 8}%
\special{pa 1100 900}%
\special{pa 1100 600}%
\special{fp}%
}}%
% CIRCLE 2 0 3 0 Black White
% 4 500 600 600 600 600 600 400 600
% 
{\color[named]{Black}{%
\special{pn 8}%
\special{ar 500 600 100 100  3.1415927 6.2831853}%
}}%
% LINE 2 0 3 0 Black White
% 2 400 600 400 1100
% 
{\color[named]{Black}{%
\special{pn 8}%
\special{pa 400 600}%
\special{pa 400 1100}%
\special{fp}%
}}%
% ELLIPSE 2 0 3 0 Black White
% 4 600 1100 400 1200 400 1100 600 1200
% 
{\color[named]{Black}{%
\special{pn 8}%
\special{ar 600 1100 200 100  1.5707963 3.1415927}%
}}%
% ELLIPSE 2 0 3 0 Black White
% 4 300 1000 200 800 300 800 200 1000
% 
{\color[named]{Black}{%
\special{pn 8}%
\special{ar 300 1000 100 200  3.1415927 4.7123890}%
}}%
% CIRCLE 2 0 3 0 Black White
% 4 500 1000 200 1000 200 1000 500 1400
% 
{\color[named]{Black}{%
\special{pn 8}%
\special{ar 500 1000 300 300  1.5707963 3.1415927}%
}}%
% CIRCLE 2 0 3 0 Black White
% 4 500 1200 500 1300 500 1300 600 1200
% 
{\color[named]{Black}{%
\special{pn 8}%
\special{ar 500 1200 100 100  6.2831853 6.2831853}%
\special{ar 500 1200 100 100  0.0000000 1.5707963}%
}}%
% LINE 2 0 3 0 Black White
% 2 300 800 375 750
% 
{\color[named]{Black}{%
\special{pn 8}%
\special{pa 300 800}%
\special{pa 376 750}%
\special{fp}%
}}%
% LINE 2 0 3 0 Black White
% 2 600 600 420 720
% 
{\color[named]{Black}{%
\special{pn 8}%
\special{pa 600 600}%
\special{pa 420 720}%
\special{fp}%
}}%
% CIRCLE 2 0 3 0 Black White
% 4 1700 600 1800 600 1800 600 1600 600
% 
{\color[named]{Black}{%
\special{pn 8}%
\special{ar 1700 600 100 100  3.1415927 6.2831853}%
}}%
% LINE 2 0 3 0 Black White
% 2 1800 600 1800 1100
% 
{\color[named]{Black}{%
\special{pn 8}%
\special{pa 1800 600}%
\special{pa 1800 1100}%
\special{fp}%
}}%
% ELLIPSE 2 0 3 0 Black White
% 4 1600 1100 1800 1200 1600 1200 1800 1100
% 
{\color[named]{Black}{%
\special{pn 8}%
\special{ar 1600 1100 200 100  6.2831853 6.2831853}%
\special{ar 1600 1100 200 100  0.0000000 1.5707963}%
}}%
% DOT 0 0 3 0 Black White
% 1 1600 1200
% 
{\color[named]{Black}{%
\special{pn 4}%
\special{sh 1}%
\special{ar 1600 1200 26 26 0  6.28318530717959E+0000}%
}}%
% DOT 0 0 3 0 Black White
% 1 1600 600
% 
{\color[named]{Black}{%
\special{pn 4}%
\special{sh 1}%
\special{ar 1600 600 26 26 0  6.28318530717959E+0000}%
}}%
% CIRCLE 2 0 3 0 Black White
% 4 1700 1200 1700 1300 1600 1200 1700 1300
% 
{\color[named]{Black}{%
\special{pn 8}%
\special{ar 1700 1200 100 100  1.5707963 3.1415927}%
}}%
% CIRCLE 2 0 3 0 Black White
% 4 1700 1000 2000 1000 1700 1400 2000 1000
% 
{\color[named]{Black}{%
\special{pn 8}%
\special{ar 1700 1000 300 300  6.2831853 6.2831853}%
\special{ar 1700 1000 300 300  0.0000000 1.5707963}%
}}%
% ELLIPSE 2 0 3 0 Black White
% 4 1900 1000 2000 800 2000 1000 1900 800
% 
{\color[named]{Black}{%
\special{pn 8}%
\special{ar 1900 1000 100 200  4.7123890 6.2831853}%
}}%
% LINE 2 0 3 0 Black White
% 2 1900 800 1825 750
% 
{\color[named]{Black}{%
\special{pn 8}%
\special{pa 1900 800}%
\special{pa 1826 750}%
\special{fp}%
}}%
% LINE 2 0 3 0 Black White
% 2 1600 600 1780 720
% 
{\color[named]{Black}{%
\special{pn 8}%
\special{pa 1600 600}%
\special{pa 1780 720}%
\special{fp}%
}}%
% LINE 2 0 3 0 Black White
% 2 2600 900 2900 700
% 
{\color[named]{Black}{%
\special{pn 8}%
\special{pa 2600 900}%
\special{pa 2900 700}%
\special{fp}%
}}%
% LINE 2 0 3 0 Black White
% 2 2900 700 3200 900
% 
{\color[named]{Black}{%
\special{pn 8}%
\special{pa 2900 700}%
\special{pa 3200 900}%
\special{fp}%
}}%
% LINE 2 0 3 0 Black White
% 2 3200 900 3500 700
% 
{\color[named]{Black}{%
\special{pn 8}%
\special{pa 3200 900}%
\special{pa 3500 700}%
\special{fp}%
}}%
% LINE 2 0 3 0 Black White
% 2 3500 700 3800 900
% 
{\color[named]{Black}{%
\special{pn 8}%
\special{pa 3500 700}%
\special{pa 3800 900}%
\special{fp}%
}}%
% DOT 0 0 3 0 Black White
% 1 2600 900
% 
{\color[named]{Black}{%
\special{pn 4}%
\special{sh 1}%
\special{ar 2600 900 26 26 0  6.28318530717959E+0000}%
}}%
% DOT 0 0 3 0 Black White
% 1 2900 700
% 
{\color[named]{Black}{%
\special{pn 4}%
\special{sh 1}%
\special{ar 2900 700 26 26 0  6.28318530717959E+0000}%
}}%
% DOT 0 0 3 0 Black White
% 1 3200 900
% 
{\color[named]{Black}{%
\special{pn 4}%
\special{sh 1}%
\special{ar 3200 900 26 26 0  6.28318530717959E+0000}%
}}%
% DOT 0 0 3 0 Black White
% 1 3500 700
% 
{\color[named]{Black}{%
\special{pn 4}%
\special{sh 1}%
\special{ar 3500 700 26 26 0  6.28318530717959E+0000}%
}}%
% DOT 0 0 3 0 Black White
% 1 3800 900
% 
{\color[named]{Black}{%
\special{pn 4}%
\special{sh 1}%
\special{ar 3800 900 26 26 0  6.28318530717959E+0000}%
}}%
% CIRCLE 2 0 3 0 Black White
% 4 2600 1050 2600 900 2600 900 2600 1300
% 
{\color[named]{Black}{%
\special{pn 8}%
\special{ar 2600 1050 150 150  1.5707963 4.7123890}%
}}%
% LINE 2 0 3 0 Black White
% 2 2600 1200 3800 1200
% 
{\color[named]{Black}{%
\special{pn 8}%
\special{pa 2600 1200}%
\special{pa 3800 1200}%
\special{fp}%
}}%
% CIRCLE 2 0 3 0 Black White
% 4 3800 1050 3800 900 3800 1300 3800 900
% 
{\color[named]{Black}{%
\special{pn 8}%
\special{ar 3800 1050 150 150  4.7123890 6.2831853}%
\special{ar 3800 1050 150 150  0.0000000 1.5707963}%
}}%
% DOT 0 0 3 0 Black White
% 1 3000 1200
% 
{\color[named]{Black}{%
\special{pn 4}%
\special{sh 1}%
\special{ar 3000 1200 26 26 0  6.28318530717959E+0000}%
}}%
% CIRCLE 2 0 3 0 Black White
% 4 3100 1200 3000 1200 3000 1200 3200 1225
% 
{\color[named]{Black}{%
\special{pn 8}%
\special{ar 3100 1200 100 100  0.2449787 3.1415927}%
}}%
% LINE 2 0 3 0 Black White
% 2 3200 1175 3200 900
% 
{\color[named]{Black}{%
\special{pn 8}%
\special{pa 3200 1176}%
\special{pa 3200 900}%
\special{fp}%
}}%
% LINE 2 0 3 0 Black White
% 2 2600 900 2600 700
% 
{\color[named]{Black}{%
\special{pn 8}%
\special{pa 2600 900}%
\special{pa 2600 700}%
\special{fp}%
}}%
% ELLIPSE 2 0 3 0 Black White
% 4 3050 700 3500 500 3600 700 2500 700
% 
{\color[named]{Black}{%
\special{pn 8}%
\special{ar 3050 700 450 200  3.1415927 6.2831853}%
}}%
% ELLIPSE 2 0 3 0 Black White
% 4 3350 700 2900 500 4200 700 3240 400
% 
{\color[named]{Black}{%
\special{pn 8}%
\special{ar 3350 700 450 200  4.5508461 6.2831853}%
}}%
% ELLIPSE 2 0 3 0 Black White
% 4 3350 700 2900 500 3120 500 2720 700
% 
{\color[named]{Black}{%
\special{pn 8}%
\special{ar 3350 700 450 200  3.1415927 4.2398920}%
}}%
% LINE 2 0 3 0 Black White
% 2 3800 700 3800 900
% 
{\color[named]{Black}{%
\special{pn 8}%
\special{pa 3800 700}%
\special{pa 3800 900}%
\special{fp}%
}}%
\end{picture}%

%% file: balanced.tex
%WinTpicVersion4.31b
{\unitlength 0.1in%
\begin{picture}( 18.0000,  8.0000)(  2.0000,-13.0000)%
% LINE 2 0 3 0 Black White  
% 2 600 600 1600 600
% 
\special{pn 8}%
\special{pa 600 600}%
\special{pa 1600 600}%
\special{fp}%
% LINE 2 0 3 0 Black White  
% 2 600 1200 1600 1200
% 
\special{pn 8}%
\special{pa 600 1200}%
\special{pa 1600 1200}%
\special{fp}%
% DOT 0 0 3 0 Black White  
% 1 600 600
% 
\special{pn 4}%
\special{sh 1}%
\special{ar 600 600 28 28 0  6.28318530717959E+0000}%
% DOT 0 0 3 0 Black White  
% 1 600 1200
% 
\special{pn 4}%
\special{sh 1}%
\special{ar 600 1200 28 28 0  6.28318530717959E+0000}%
% DOT 0 0 3 0 Black White  
% 1 1100 600
% 
\special{pn 4}%
\special{sh 1}%
\special{ar 1100 600 28 28 0  6.28318530717959E+0000}%
% DOT 0 0 3 0 Black White  
% 1 1100 1200
% 
\special{pn 4}%
\special{sh 1}%
\special{ar 1100 1200 28 28 0  6.28318530717959E+0000}%
% CIRCLE 2 0 3 0 Black White  
% 4 500 600 600 600 600 600 400 600
% 
\special{pn 8}%
\special{ar 500 600 100 100  3.1415927  6.2831853}%
% LINE 2 0 3 0 Black White  
% 2 400 600 400 1100
% 
\special{pn 8}%
\special{pa 400 600}%
\special{pa 400 1100}%
\special{fp}%
% ELLIPSE 2 0 3 0 Black White  
% 4 600 1100 400 1200 400 1100 600 1200
% 
\special{pn 8}%
\special{ar 600 1100 200 100  1.5707963  3.1415927}%
% ELLIPSE 2 0 3 0 Black White  
% 4 300 1000 200 800 300 800 200 1000
% 
\special{pn 8}%
\special{ar 300 1000 100 200  3.1415927  4.7123890}%
% CIRCLE 2 0 3 0 Black White  
% 4 500 1000 200 1000 200 1000 500 1400
% 
\special{pn 8}%
\special{ar 500 1000 300 300  1.5707963  3.1415927}%
% CIRCLE 2 0 3 0 Black White  
% 4 500 1200 500 1300 500 1300 600 1200
% 
\special{pn 8}%
\special{ar 500 1200 100 100  6.2831853  1.5707963}%
% LINE 2 0 3 0 Black White  
% 2 300 800 375 750
% 
\special{pn 8}%
\special{pa 300 800}%
\special{pa 375 750}%
\special{fp}%
% LINE 2 0 3 0 Black White  
% 2 600 600 420 720
% 
\special{pn 8}%
\special{pa 600 600}%
\special{pa 420 720}%
\special{fp}%
% CIRCLE 2 0 3 0 Black White  
% 4 1700 600 1800 600 1800 600 1600 600
% 
\special{pn 8}%
\special{ar 1700 600 100 100  3.1415927  6.2831853}%
% LINE 2 0 3 0 Black White  
% 2 1800 600 1800 1100
% 
\special{pn 8}%
\special{pa 1800 600}%
\special{pa 1800 1100}%
\special{fp}%
% ELLIPSE 2 0 3 0 Black White  
% 4 1600 1100 1800 1200 1600 1200 1800 1100
% 
\special{pn 8}%
\special{ar 1600 1100 200 100  6.2831853  1.5707963}%
% DOT 0 0 3 0 Black White  
% 1 1600 1200
% 
\special{pn 4}%
\special{sh 1}%
\special{ar 1600 1200 28 28 0  6.28318530717959E+0000}%
% DOT 0 0 3 0 Black White  
% 1 1600 600
% 
\special{pn 4}%
\special{sh 1}%
\special{ar 1600 600 28 28 0  6.28318530717959E+0000}%
% CIRCLE 2 0 3 0 Black White  
% 4 1700 1200 1700 1300 1600 1200 1700 1300
% 
\special{pn 8}%
\special{ar 1700 1200 100 100  1.5707963  3.1415927}%
% CIRCLE 2 0 3 0 Black White  
% 4 1700 1000 2000 1000 1700 1400 2000 1000
% 
\special{pn 8}%
\special{ar 1700 1000 300 300  6.2831853  1.5707963}%
% ELLIPSE 2 0 3 0 Black White  
% 4 1900 1000 2000 800 2000 1000 1900 800
% 
\special{pn 8}%
\special{ar 1900 1000 100 200  4.7123890  6.2831853}%
% LINE 2 0 3 0 Black White  
% 2 1900 800 1825 750
% 
\special{pn 8}%
\special{pa 1900 800}%
\special{pa 1825 750}%
\special{fp}%
% LINE 2 0 3 0 Black White  
% 2 1600 600 1780 720
% 
\special{pn 8}%
\special{pa 1600 600}%
\special{pa 1780 720}%
\special{fp}%
% LINE 2 0 3 0 Black White  
% 2 1100 1200 1100 600
% 
\special{pn 8}%
\special{pa 1100 1200}%
\special{pa 1100 600}%
\special{fp}%
\end{picture}}%

%% file: app.tex
%WinTpicVersion4.31b
{\unitlength 0.1in%
\begin{picture}( 31.4500, 22.0200)(  5.0500,-22.8000)%
% STR 2 0 3 0 Black White  
% 4 520 150 520 230 2 0 0 0
% a vertex of type 1
\put(5.2000,-2.3000){\makebox(0,0)[lb]{a vertex of type 1}}%
% STR 2 0 3 0 Black White  
% 4 1359 466 1359 546 2 0 0 0
% $e_2$
\put(13.5900,-5.4600){\makebox(0,0)[lb]{$e_2$}}%
% STR 2 0 3 0 Black White  
% 4 946 1023 946 1103 2 0 0 0
% $e_1$
\put(9.4600,-11.0300){\makebox(0,0)[lb]{$e_1$}}%
% STR 2 0 3 0 Black White  
% 4 505 448 505 528 2 0 0 0
% $e_3$
\put(5.0500,-5.2800){\makebox(0,0)[lb]{$e_3$}}%
% DOT 0 0 3 0 Black White  
% 1 1003 652
% 
\special{pn 4}%
\special{sh 1}%
\special{ar 1003 652 28 28 0  6.28318530717959E+0000}%
% LINE 2 0 3 0 Black White  
% 2 1003 652 1003 972
% 
\special{pn 8}%
\special{pa 1003 652}%
\special{pa 1003 972}%
\special{fp}%
% LINE 2 0 3 0 Black White  
% 2 1003 652 683 492
% 
\special{pn 8}%
\special{pa 1003 652}%
\special{pa 683 492}%
\special{fp}%
% LINE 2 0 3 0 Black White  
% 2 1003 652 1323 492
% 
\special{pn 8}%
\special{pa 1003 652}%
\special{pa 1323 492}%
\special{fp}%
% LINE 2 0 3 0 Black White  
% 4 1003 812 1035 876 1003 812 971 876
% 
\special{pn 8}%
\special{pa 1003 812}%
\special{pa 1035 876}%
\special{fp}%
\special{pa 1003 812}%
\special{pa 971 876}%
\special{fp}%
% LINE 2 0 3 0 Black White  
% 4 1221 548 1173 601 1221 548 1150 542
% 
\special{pn 8}%
\special{pa 1221 548}%
\special{pa 1173 601}%
\special{fp}%
\special{pa 1221 548}%
\special{pa 1150 542}%
\special{fp}%
% LINE 2 0 3 0 Black White  
% 4 789 546 837 599 789 546 860 540
% 
\special{pn 8}%
\special{pa 789 546}%
\special{pa 837 599}%
\special{fp}%
\special{pa 789 546}%
\special{pa 860 540}%
\special{fp}%
% CIRCLE 2 0 3 0 Black White  
% 4 1012 645 1111 686 1111 686 945 702
% 
\special{pn 8}%
\special{ar 1012 645 107 107  2.4366655  0.3926375}%
% DOT 0 0 3 0 Black White  
% 1 1106 685
% 
\special{pn 4}%
\special{sh 1}%
\special{ar 1106 685 28 28 0  6.28318530717959E+0000}%
% LINE 2 0 3 0 Black White  
% 4 930 722 934 668 930 722 890 687
% 
\special{pn 8}%
\special{pa 930 722}%
\special{pa 934 668}%
\special{fp}%
\special{pa 930 722}%
\special{pa 890 687}%
\special{fp}%
% SPLINE 1 0 3 0 Black White  
% 3 1096 968 1156 728 1366 578
% 
\special{pn 13}%
\special{pa 1096 968}%
\special{pa 1099 934}%
\special{pa 1103 900}%
\special{pa 1108 867}%
\special{pa 1114 835}%
\special{pa 1122 804}%
\special{pa 1132 774}%
\special{pa 1145 747}%
\special{pa 1161 721}%
\special{pa 1180 698}%
\special{pa 1202 677}%
\special{pa 1227 657}%
\special{pa 1253 639}%
\special{pa 1282 622}%
\special{pa 1311 606}%
\special{pa 1342 590}%
\special{pa 1366 578}%
\special{fp}%
% SPLINE 1 0 3 0 Black White  
% 3 904 970 882 722 646 568
% 
\special{pn 13}%
\special{pa 904 970}%
\special{pa 907 935}%
\special{pa 909 900}%
\special{pa 910 866}%
\special{pa 909 833}%
\special{pa 906 801}%
\special{pa 901 771}%
\special{pa 892 742}%
\special{pa 879 717}%
\special{pa 862 693}%
\special{pa 841 672}%
\special{pa 817 653}%
\special{pa 790 635}%
\special{pa 761 619}%
\special{pa 730 604}%
\special{pa 698 590}%
\special{pa 665 576}%
\special{pa 646 568}%
\special{fp}%
% SPLINE 1 0 3 0 Black White  
% 3 726 420 1002 506 1274 420
% 
\special{pn 13}%
\special{pa 726 420}%
\special{pa 757 434}%
\special{pa 787 448}%
\special{pa 818 461}%
\special{pa 848 474}%
\special{pa 879 484}%
\special{pa 909 493}%
\special{pa 940 500}%
\special{pa 970 504}%
\special{pa 1001 506}%
\special{pa 1031 504}%
\special{pa 1062 500}%
\special{pa 1093 493}%
\special{pa 1123 484}%
\special{pa 1154 473}%
\special{pa 1184 461}%
\special{pa 1215 448}%
\special{pa 1245 434}%
\special{pa 1274 420}%
\special{fp}%
% LINE 1 0 3 0 Black White  
% 2 1274 420 1368 580
% 
\special{pn 13}%
\special{pa 1274 420}%
\special{pa 1368 580}%
\special{fp}%
% LINE 1 0 3 0 Black White  
% 2 728 420 644 568
% 
\special{pn 13}%
\special{pa 728 420}%
\special{pa 644 568}%
\special{fp}%
% LINE 1 0 3 0 Black White  
% 2 906 970 1102 970
% 
\special{pn 13}%
\special{pa 906 970}%
\special{pa 1102 970}%
\special{fp}%
% DOT 0 0 3 0 Black White  
% 1 3283 754
% 
\special{pn 4}%
\special{sh 1}%
\special{ar 3283 754 28 28 0  6.28318530717959E+0000}%
% LINE 2 0 3 0 Black White  
% 2 3283 754 3283 434
% 
\special{pn 8}%
\special{pa 3283 754}%
\special{pa 3283 434}%
\special{fp}%
% LINE 2 0 3 0 Black White  
% 2 3283 754 2963 914
% 
\special{pn 8}%
\special{pa 3283 754}%
\special{pa 2963 914}%
\special{fp}%
% LINE 2 0 3 0 Black White  
% 2 3283 754 3603 914
% 
\special{pn 8}%
\special{pa 3283 754}%
\special{pa 3603 914}%
\special{fp}%
% LINE 2 0 3 0 Black White  
% 4 3282 506 3314 570 3282 506 3250 570
% 
\special{pn 8}%
\special{pa 3282 506}%
\special{pa 3314 570}%
\special{fp}%
\special{pa 3282 506}%
\special{pa 3250 570}%
\special{fp}%
% LINE 2 0 3 0 Black White  
% 4 3439 832 3487 885 3439 832 3510 826
% 
\special{pn 8}%
\special{pa 3439 832}%
\special{pa 3487 885}%
\special{fp}%
\special{pa 3439 832}%
\special{pa 3510 826}%
\special{fp}%
% LINE 2 0 3 0 Black White  
% 4 3119 832 3071 885 3119 832 3048 826
% 
\special{pn 8}%
\special{pa 3119 832}%
\special{pa 3071 885}%
\special{fp}%
\special{pa 3119 832}%
\special{pa 3048 826}%
\special{fp}%
% CIRCLE 2 0 3 0 Black White  
% 4 3284 757 3378 703 3211 708 3378 703
% 
\special{pn 8}%
\special{ar 3284 757 108 108  5.7617509  3.7327498}%
% DOT 0 0 3 0 Black White  
% 1 3377 704
% 
\special{pn 4}%
\special{sh 1}%
\special{ar 3377 704 28 28 0  6.28318530717959E+0000}%
% LINE 2 0 3 0 Black White  
% 4 3194 690 3204 743 3194 690 3158 730
% 
\special{pn 8}%
\special{pa 3194 690}%
\special{pa 3204 743}%
\special{fp}%
\special{pa 3194 690}%
\special{pa 3158 730}%
\special{fp}%
% STR 2 0 3 0 Black White  
% 4 3196 620 3196 700 2 0 0 0
% $v$
\put(31.9600,-7.0000){\makebox(0,0)[lb]{$v$}}%
% STR 2 0 3 0 Black White  
% 4 3636 914 3636 994 2 0 0 0
% $e_1$
\put(36.3600,-9.9400){\makebox(0,0)[lb]{$e_1$}}%
% STR 2 0 3 0 Black White  
% 4 3240 325 3240 405 2 0 0 0
% $e_2$
\put(32.4000,-4.0500){\makebox(0,0)[lb]{$e_2$}}%
% STR 2 0 3 0 Black White  
% 4 2808 906 2808 986 2 0 0 0
% $e_3$
\put(28.0800,-9.8600){\makebox(0,0)[lb]{$e_3$}}%
% SPLINE 1 0 3 0 Black White  
% 3 3380 435 3440 675 3650 825
% 
\special{pn 13}%
\special{pa 3380 435}%
\special{pa 3383 469}%
\special{pa 3387 503}%
\special{pa 3392 536}%
\special{pa 3398 568}%
\special{pa 3406 599}%
\special{pa 3416 629}%
\special{pa 3429 656}%
\special{pa 3445 682}%
\special{pa 3464 705}%
\special{pa 3486 726}%
\special{pa 3511 746}%
\special{pa 3537 764}%
\special{pa 3566 781}%
\special{pa 3595 797}%
\special{pa 3626 813}%
\special{pa 3650 825}%
\special{fp}%
% SPLINE 1 0 3 0 Black White  
% 3 3184 436 3162 684 2926 838
% 
\special{pn 13}%
\special{pa 3184 436}%
\special{pa 3187 471}%
\special{pa 3189 506}%
\special{pa 3190 540}%
\special{pa 3189 573}%
\special{pa 3186 605}%
\special{pa 3181 635}%
\special{pa 3172 664}%
\special{pa 3159 689}%
\special{pa 3142 713}%
\special{pa 3121 734}%
\special{pa 3097 753}%
\special{pa 3070 771}%
\special{pa 3041 787}%
\special{pa 3010 802}%
\special{pa 2978 816}%
\special{pa 2945 830}%
\special{pa 2926 838}%
\special{fp}%
% SPLINE 1 0 3 0 Black White  
% 3 3006 986 3282 900 3554 986
% 
\special{pn 13}%
\special{pa 3006 986}%
\special{pa 3037 972}%
\special{pa 3067 958}%
\special{pa 3098 945}%
\special{pa 3128 932}%
\special{pa 3159 922}%
\special{pa 3189 913}%
\special{pa 3220 906}%
\special{pa 3250 902}%
\special{pa 3281 900}%
\special{pa 3311 902}%
\special{pa 3342 906}%
\special{pa 3373 913}%
\special{pa 3403 922}%
\special{pa 3434 933}%
\special{pa 3464 945}%
\special{pa 3495 958}%
\special{pa 3525 972}%
\special{pa 3554 986}%
\special{fp}%
% LINE 1 0 3 0 Black White  
% 2 3554 986 3648 826
% 
\special{pn 13}%
\special{pa 3554 986}%
\special{pa 3648 826}%
\special{fp}%
% LINE 1 0 3 0 Black White  
% 2 3008 986 2924 838
% 
\special{pn 13}%
\special{pa 3008 986}%
\special{pa 2924 838}%
\special{fp}%
% LINE 1 0 3 0 Black White  
% 2 3186 436 3382 436
% 
\special{pn 13}%
\special{pa 3186 436}%
\special{pa 3382 436}%
\special{fp}%
% STR 2 0 3 0 Black White  
% 4 2820 150 2820 230 2 0 0 0
% a vertex of type 2
\put(28.2000,-2.3000){\makebox(0,0)[lb]{a vertex of type 2}}%
% LINE 1 0 3 0 Black White  
% 2 920 2278 1080 2278
% 
\special{pn 13}%
\special{pa 920 2278}%
\special{pa 1080 2278}%
\special{fp}%
% LINE 1 0 3 0 Black White  
% 2 680 1358 840 1358
% 
\special{pn 13}%
\special{pa 680 1358}%
\special{pa 840 1358}%
\special{fp}%
% LINE 1 0 3 0 Black White  
% 2 1160 1358 1320 1358
% 
\special{pn 13}%
\special{pa 1160 1358}%
\special{pa 1320 1358}%
\special{fp}%
% LINE 1 0 3 0 Black White  
% 2 840 1358 840 1518
% 
\special{pn 13}%
\special{pa 840 1358}%
\special{pa 840 1518}%
\special{fp}%
% LINE 1 0 3 0 Black White  
% 2 1160 1358 1160 1518
% 
\special{pn 13}%
\special{pa 1160 1358}%
\special{pa 1160 1518}%
\special{fp}%
% CIRCLE 1 0 3 0 Black White  
% 4 1000 1518 840 1518 840 1518 1080 1518
% 
\special{pn 13}%
\special{ar 1000 1518 160 160  6.2831853  3.1415927}%
% LINE 1 0 3 0 Black White  
% 2 680 1358 680 1598
% 
\special{pn 13}%
\special{pa 680 1358}%
\special{pa 680 1598}%
\special{fp}%
% LINE 1 0 3 0 Black White  
% 2 1320 1358 1320 1598
% 
\special{pn 13}%
\special{pa 1320 1358}%
\special{pa 1320 1598}%
\special{fp}%
% LINE 1 0 3 0 Black White  
% 2 1080 2038 1080 2278
% 
\special{pn 13}%
\special{pa 1080 2038}%
\special{pa 1080 2278}%
\special{fp}%
% LINE 1 0 3 0 Black White  
% 2 920 2038 920 2278
% 
\special{pn 13}%
\special{pa 920 2038}%
\special{pa 920 2278}%
\special{fp}%
% LINE 2 0 3 0 Black White  
% 2 1000 2278 1000 1838
% 
\special{pn 8}%
\special{pa 1000 2278}%
\special{pa 1000 1838}%
\special{fp}%
% SPLINE 2 0 3 0 Black White  
% 5 1000 1838 1020 1798 1080 1758 1200 1638 1240 1358
% 
\special{pn 8}%
\special{pa 1000 1838}%
\special{pa 1013 1808}%
\special{pa 1033 1785}%
\special{pa 1061 1768}%
\special{pa 1090 1752}%
\special{pa 1117 1734}%
\special{pa 1142 1712}%
\special{pa 1165 1689}%
\special{pa 1184 1664}%
\special{pa 1201 1637}%
\special{pa 1214 1608}%
\special{pa 1223 1579}%
\special{pa 1231 1548}%
\special{pa 1236 1517}%
\special{pa 1239 1484}%
\special{pa 1240 1452}%
\special{pa 1241 1418}%
\special{pa 1240 1385}%
\special{pa 1240 1358}%
\special{fp}%
% SPLINE 1 0 3 0 Black White  
% 5 1080 2038 1080 1958 1120 1838 1280 1718 1320 1598
% 
\special{pn 13}%
\special{pa 1080 2038}%
\special{pa 1079 2006}%
\special{pa 1079 1974}%
\special{pa 1082 1942}%
\special{pa 1087 1909}%
\special{pa 1097 1878}%
\special{pa 1112 1850}%
\special{pa 1132 1825}%
\special{pa 1157 1805}%
\special{pa 1184 1787}%
\special{pa 1213 1771}%
\special{pa 1241 1753}%
\special{pa 1266 1734}%
\special{pa 1285 1711}%
\special{pa 1299 1683}%
\special{pa 1309 1653}%
\special{pa 1316 1620}%
\special{pa 1320 1598}%
\special{fp}%
% SPLINE 1 0 3 0 Black White  
% 5 918 2038 918 1958 878 1838 718 1718 678 1598
% 
\special{pn 13}%
\special{pa 918 2038}%
\special{pa 919 2006}%
\special{pa 919 1974}%
\special{pa 916 1942}%
\special{pa 911 1909}%
\special{pa 901 1878}%
\special{pa 886 1850}%
\special{pa 866 1825}%
\special{pa 841 1805}%
\special{pa 814 1787}%
\special{pa 785 1771}%
\special{pa 757 1753}%
\special{pa 732 1734}%
\special{pa 713 1711}%
\special{pa 699 1683}%
\special{pa 689 1653}%
\special{pa 682 1620}%
\special{pa 678 1598}%
\special{fp}%
% SPLINE 2 0 3 0 Black White  
% 5 1000 1838 980 1798 920 1758 800 1638 760 1358
% 
\special{pn 8}%
\special{pa 1000 1838}%
\special{pa 987 1808}%
\special{pa 967 1785}%
\special{pa 939 1768}%
\special{pa 910 1752}%
\special{pa 883 1734}%
\special{pa 858 1712}%
\special{pa 835 1689}%
\special{pa 816 1664}%
\special{pa 799 1637}%
\special{pa 786 1608}%
\special{pa 777 1579}%
\special{pa 769 1548}%
\special{pa 764 1517}%
\special{pa 761 1484}%
\special{pa 760 1452}%
\special{pa 759 1418}%
\special{pa 760 1385}%
\special{pa 760 1358}%
\special{fp}%
% LINE 2 0 3 0 Black White  
% 4 1001 2052 1033 2116 1001 2052 969 2116
% 
\special{pn 8}%
\special{pa 1001 2052}%
\special{pa 1033 2116}%
\special{fp}%
\special{pa 1001 2052}%
\special{pa 969 2116}%
\special{fp}%
% LINE 2 0 3 0 Black White  
% 4 1239 1452 1271 1516 1239 1452 1207 1516
% 
\special{pn 8}%
\special{pa 1239 1452}%
\special{pa 1271 1516}%
\special{fp}%
\special{pa 1239 1452}%
\special{pa 1207 1516}%
\special{fp}%
% LINE 2 0 3 0 Black White  
% 4 759 1452 791 1516 759 1452 727 1516
% 
\special{pn 8}%
\special{pa 759 1452}%
\special{pa 791 1516}%
\special{fp}%
\special{pa 759 1452}%
\special{pa 727 1516}%
\special{fp}%
% LINE 1 0 3 0 Black White  
% 2 3200 1360 3360 1360
% 
\special{pn 13}%
\special{pa 3200 1360}%
\special{pa 3360 1360}%
\special{fp}%
% LINE 1 0 3 0 Black White  
% 2 2960 2280 3120 2280
% 
\special{pn 13}%
\special{pa 2960 2280}%
\special{pa 3120 2280}%
\special{fp}%
% LINE 1 0 3 0 Black White  
% 2 3440 2280 3600 2280
% 
\special{pn 13}%
\special{pa 3440 2280}%
\special{pa 3600 2280}%
\special{fp}%
% LINE 1 0 3 0 Black White  
% 2 3120 2280 3120 2120
% 
\special{pn 13}%
\special{pa 3120 2280}%
\special{pa 3120 2120}%
\special{fp}%
% LINE 1 0 3 0 Black White  
% 2 3440 2280 3440 2120
% 
\special{pn 13}%
\special{pa 3440 2280}%
\special{pa 3440 2120}%
\special{fp}%
% CIRCLE 1 0 3 0 Black White  
% 4 3280 2120 3120 2120 3360 2120 3120 2120
% 
\special{pn 13}%
\special{ar 3280 2120 160 160  3.1415927  6.2831853}%
% LINE 1 0 3 0 Black White  
% 2 2960 2280 2960 2040
% 
\special{pn 13}%
\special{pa 2960 2280}%
\special{pa 2960 2040}%
\special{fp}%
% LINE 1 0 3 0 Black White  
% 2 3600 2280 3600 2040
% 
\special{pn 13}%
\special{pa 3600 2280}%
\special{pa 3600 2040}%
\special{fp}%
% LINE 1 0 3 0 Black White  
% 2 3360 1600 3360 1360
% 
\special{pn 13}%
\special{pa 3360 1600}%
\special{pa 3360 1360}%
\special{fp}%
% LINE 1 0 3 0 Black White  
% 2 3200 1600 3200 1360
% 
\special{pn 13}%
\special{pa 3200 1600}%
\special{pa 3200 1360}%
\special{fp}%
% LINE 2 0 3 0 Black White  
% 2 3279 1360 3279 1800
% 
\special{pn 8}%
\special{pa 3279 1360}%
\special{pa 3279 1800}%
\special{fp}%
% SPLINE 2 0 3 0 Black White  
% 5 3279 1800 3299 1840 3359 1880 3479 2000 3519 2280
% 
\special{pn 8}%
\special{pa 3279 1800}%
\special{pa 3292 1830}%
\special{pa 3312 1853}%
\special{pa 3340 1870}%
\special{pa 3369 1886}%
\special{pa 3396 1904}%
\special{pa 3421 1926}%
\special{pa 3444 1949}%
\special{pa 3463 1974}%
\special{pa 3480 2001}%
\special{pa 3493 2030}%
\special{pa 3502 2059}%
\special{pa 3510 2090}%
\special{pa 3515 2121}%
\special{pa 3518 2154}%
\special{pa 3519 2186}%
\special{pa 3520 2220}%
\special{pa 3519 2253}%
\special{pa 3519 2280}%
\special{fp}%
% SPLINE 1 0 3 0 Black White  
% 5 3360 1600 3360 1680 3400 1800 3560 1920 3600 2040
% 
\special{pn 13}%
\special{pa 3360 1600}%
\special{pa 3359 1632}%
\special{pa 3359 1664}%
\special{pa 3362 1696}%
\special{pa 3367 1729}%
\special{pa 3377 1760}%
\special{pa 3392 1788}%
\special{pa 3412 1813}%
\special{pa 3437 1833}%
\special{pa 3464 1851}%
\special{pa 3493 1867}%
\special{pa 3521 1885}%
\special{pa 3546 1904}%
\special{pa 3565 1927}%
\special{pa 3579 1955}%
\special{pa 3589 1985}%
\special{pa 3596 2018}%
\special{pa 3600 2040}%
\special{fp}%
% SPLINE 1 0 3 0 Black White  
% 5 3198 1600 3198 1680 3158 1800 2998 1920 2958 2040
% 
\special{pn 13}%
\special{pa 3198 1600}%
\special{pa 3199 1632}%
\special{pa 3199 1664}%
\special{pa 3196 1696}%
\special{pa 3191 1729}%
\special{pa 3181 1760}%
\special{pa 3166 1788}%
\special{pa 3146 1813}%
\special{pa 3121 1833}%
\special{pa 3094 1851}%
\special{pa 3065 1867}%
\special{pa 3037 1885}%
\special{pa 3012 1904}%
\special{pa 2993 1927}%
\special{pa 2979 1955}%
\special{pa 2969 1985}%
\special{pa 2962 2018}%
\special{pa 2958 2040}%
\special{fp}%
% SPLINE 2 0 3 0 Black White  
% 5 3279 1800 3259 1840 3199 1880 3079 2000 3039 2280
% 
\special{pn 8}%
\special{pa 3279 1800}%
\special{pa 3266 1830}%
\special{pa 3246 1853}%
\special{pa 3218 1870}%
\special{pa 3189 1886}%
\special{pa 3162 1904}%
\special{pa 3137 1926}%
\special{pa 3114 1949}%
\special{pa 3095 1974}%
\special{pa 3078 2001}%
\special{pa 3065 2030}%
\special{pa 3056 2059}%
\special{pa 3048 2090}%
\special{pa 3043 2121}%
\special{pa 3040 2154}%
\special{pa 3039 2186}%
\special{pa 3038 2220}%
\special{pa 3039 2253}%
\special{pa 3039 2280}%
\special{fp}%
% LINE 2 0 3 0 Black White  
% 4 3281 1498 3313 1562 3281 1498 3249 1562
% 
\special{pn 8}%
\special{pa 3281 1498}%
\special{pa 3313 1562}%
\special{fp}%
\special{pa 3281 1498}%
\special{pa 3249 1562}%
\special{fp}%
% LINE 2 0 3 0 Black White  
% 4 3519 2134 3551 2198 3519 2134 3487 2198
% 
\special{pn 8}%
\special{pa 3519 2134}%
\special{pa 3551 2198}%
\special{fp}%
\special{pa 3519 2134}%
\special{pa 3487 2198}%
\special{fp}%
% LINE 2 0 3 0 Black White  
% 4 3039 2138 3071 2202 3039 2138 3007 2202
% 
\special{pn 8}%
\special{pa 3039 2138}%
\special{pa 3071 2202}%
\special{fp}%
\special{pa 3039 2138}%
\special{pa 3007 2202}%
\special{fp}%
% DOT 0 0 3 0 Black White  
% 1 1001 1838
% 
\special{pn 4}%
\special{sh 1}%
\special{ar 1001 1838 28 28 0  6.28318530717959E+0000}%
% DOT 0 0 3 0 Black White  
% 1 3278 1800
% 
\special{pn 4}%
\special{sh 1}%
\special{ar 3278 1800 28 28 0  6.28318530717959E+0000}%
% STR 2 0 3 0 Black White  
% 4 938 2393 938 2413 2 0 0 0
% $e_1$
\put(9.3800,-24.1300){\makebox(0,0)[lb]{$e_1$}}%
% STR 2 0 3 0 Black White  
% 4 1173 1298 1173 1318 2 0 0 0
% $e_2$
\put(11.7300,-13.1800){\makebox(0,0)[lb]{$e_2$}}%
% STR 2 0 3 0 Black White  
% 4 693 1298 693 1318 2 0 0 0
% $e_3$
\put(6.9300,-13.1800){\makebox(0,0)[lb]{$e_3$}}%
% STR 2 0 3 0 Black White  
% 4 3238 1300 3238 1320 2 0 0 0
% $e_2$
\put(32.3800,-13.2000){\makebox(0,0)[lb]{$e_2$}}%
% STR 2 0 3 0 Black White  
% 4 3478 2380 3478 2400 2 0 0 0
% $e_1$
\put(34.7800,-24.0000){\makebox(0,0)[lb]{$e_1$}}%
% STR 2 0 3 0 Black White  
% 4 2998 2380 2998 2400 2 0 0 0
% $e_3$
\put(29.9800,-24.0000){\makebox(0,0)[lb]{$e_3$}}%
% LINE 2 0 3 0 Black White  
% 2 1700 1800 2500 1800
% 
\special{pn 8}%
\special{pa 1700 1800}%
\special{pa 2500 1800}%
\special{fp}%
% LINE 2 0 3 0 Black White  
% 2 2100 1400 2100 2200
% 
\special{pn 8}%
\special{pa 2100 1400}%
\special{pa 2100 2200}%
\special{fp}%
% DOT 1 0 3 0 Black White  
% 1 2100 1800
% 
\special{pn 4}%
\special{sh 1}%
\special{ar 2100 1800 10 10 0  6.28318530717959E+0000}%
% LINE 2 0 3 0 Black White  
% 4 2100 1400 2132 1464 2100 1400 2068 1464
% 
\special{pn 8}%
\special{pa 2100 1400}%
\special{pa 2132 1464}%
\special{fp}%
\special{pa 2100 1400}%
\special{pa 2068 1464}%
\special{fp}%
% LINE 2 0 3 0 Black White  
% 4 2505 1800 2442 1833 2505 1800 2441 1769
% 
\special{pn 8}%
\special{pa 2505 1800}%
\special{pa 2442 1833}%
\special{fp}%
\special{pa 2505 1800}%
\special{pa 2441 1769}%
\special{fp}%
% STR 2 0 3 0 Black White  
% 4 2525 1775 2525 1825 2 0 0 0
% $x$
\put(25.2500,-18.2500){\makebox(0,0)[lb]{$x$}}%
% STR 2 0 3 0 Black White  
% 4 2060 1310 2060 1360 2 0 0 0
% $y$
\put(20.6000,-13.6000){\makebox(0,0)[lb]{$y$}}%
% VECTOR 2 0 3 0 Black White  
% 2 2117 1761 2237 1761
% 
\special{pn 8}%
\special{pa 2117 1761}%
\special{pa 2237 1761}%
\special{fp}%
\special{sh 1}%
\special{pa 2237 1761}%
\special{pa 2170 1741}%
\special{pa 2184 1761}%
\special{pa 2170 1781}%
\special{pa 2237 1761}%
\special{fp}%
% VECTOR 2 0 3 0 Black White  
% 2 2276 1760 2396 1760
% 
\special{pn 8}%
\special{pa 2276 1760}%
\special{pa 2396 1760}%
\special{fp}%
\special{sh 1}%
\special{pa 2396 1760}%
\special{pa 2329 1740}%
\special{pa 2343 1760}%
\special{pa 2329 1780}%
\special{pa 2396 1760}%
\special{fp}%
% VECTOR 2 0 3 0 Black White  
% 2 2276 1680 2396 1680
% 
\special{pn 8}%
\special{pa 2276 1680}%
\special{pa 2396 1680}%
\special{fp}%
\special{sh 1}%
\special{pa 2396 1680}%
\special{pa 2329 1660}%
\special{pa 2343 1680}%
\special{pa 2329 1700}%
\special{pa 2396 1680}%
\special{fp}%
% VECTOR 2 0 3 0 Black White  
% 2 2116 1680 2236 1680
% 
\special{pn 8}%
\special{pa 2116 1680}%
\special{pa 2236 1680}%
\special{fp}%
\special{sh 1}%
\special{pa 2236 1680}%
\special{pa 2169 1660}%
\special{pa 2183 1680}%
\special{pa 2169 1700}%
\special{pa 2236 1680}%
\special{fp}%
% VECTOR 2 0 3 0 Black White  
% 2 2116 1600 2236 1600
% 
\special{pn 8}%
\special{pa 2116 1600}%
\special{pa 2236 1600}%
\special{fp}%
\special{sh 1}%
\special{pa 2236 1600}%
\special{pa 2169 1580}%
\special{pa 2183 1600}%
\special{pa 2169 1620}%
\special{pa 2236 1600}%
\special{fp}%
% VECTOR 2 0 3 0 Black White  
% 2 2276 1600 2396 1600
% 
\special{pn 8}%
\special{pa 2276 1600}%
\special{pa 2396 1600}%
\special{fp}%
\special{sh 1}%
\special{pa 2396 1600}%
\special{pa 2329 1580}%
\special{pa 2343 1600}%
\special{pa 2329 1620}%
\special{pa 2396 1600}%
\special{fp}%
% VECTOR 2 0 3 0 Black White  
% 2 2276 1520 2396 1520
% 
\special{pn 8}%
\special{pa 2276 1520}%
\special{pa 2396 1520}%
\special{fp}%
\special{sh 1}%
\special{pa 2396 1520}%
\special{pa 2329 1500}%
\special{pa 2343 1520}%
\special{pa 2329 1540}%
\special{pa 2396 1520}%
\special{fp}%
% VECTOR 2 0 3 0 Black White  
% 2 2116 1520 2236 1520
% 
\special{pn 8}%
\special{pa 2116 1520}%
\special{pa 2236 1520}%
\special{fp}%
\special{sh 1}%
\special{pa 2236 1520}%
\special{pa 2169 1500}%
\special{pa 2183 1520}%
\special{pa 2169 1540}%
\special{pa 2236 1520}%
\special{fp}%
% VECTOR 2 0 3 0 Black White  
% 2 1956 1520 2076 1520
% 
\special{pn 8}%
\special{pa 1956 1520}%
\special{pa 2076 1520}%
\special{fp}%
\special{sh 1}%
\special{pa 2076 1520}%
\special{pa 2009 1500}%
\special{pa 2023 1520}%
\special{pa 2009 1540}%
\special{pa 2076 1520}%
\special{fp}%
% VECTOR 2 0 3 0 Black White  
% 2 1796 1520 1916 1520
% 
\special{pn 8}%
\special{pa 1796 1520}%
\special{pa 1916 1520}%
\special{fp}%
\special{sh 1}%
\special{pa 1916 1520}%
\special{pa 1849 1500}%
\special{pa 1863 1520}%
\special{pa 1849 1540}%
\special{pa 1916 1520}%
\special{fp}%
% VECTOR 2 0 3 0 Black White  
% 2 1796 1600 1916 1600
% 
\special{pn 8}%
\special{pa 1796 1600}%
\special{pa 1916 1600}%
\special{fp}%
\special{sh 1}%
\special{pa 1916 1600}%
\special{pa 1849 1580}%
\special{pa 1863 1600}%
\special{pa 1849 1620}%
\special{pa 1916 1600}%
\special{fp}%
% VECTOR 2 0 3 0 Black White  
% 2 1956 1600 2076 1600
% 
\special{pn 8}%
\special{pa 1956 1600}%
\special{pa 2076 1600}%
\special{fp}%
\special{sh 1}%
\special{pa 2076 1600}%
\special{pa 2009 1580}%
\special{pa 2023 1600}%
\special{pa 2009 1620}%
\special{pa 2076 1600}%
\special{fp}%
% VECTOR 2 0 3 0 Black White  
% 2 1956 1680 2076 1680
% 
\special{pn 8}%
\special{pa 1956 1680}%
\special{pa 2076 1680}%
\special{fp}%
\special{sh 1}%
\special{pa 2076 1680}%
\special{pa 2009 1660}%
\special{pa 2023 1680}%
\special{pa 2009 1700}%
\special{pa 2076 1680}%
\special{fp}%
% VECTOR 2 0 3 0 Black White  
% 2 1956 1760 2076 1760
% 
\special{pn 8}%
\special{pa 1956 1760}%
\special{pa 2076 1760}%
\special{fp}%
\special{sh 1}%
\special{pa 2076 1760}%
\special{pa 2009 1740}%
\special{pa 2023 1760}%
\special{pa 2009 1780}%
\special{pa 2076 1760}%
\special{fp}%
% VECTOR 2 0 3 0 Black White  
% 2 1796 1760 1916 1760
% 
\special{pn 8}%
\special{pa 1796 1760}%
\special{pa 1916 1760}%
\special{fp}%
\special{sh 1}%
\special{pa 1916 1760}%
\special{pa 1849 1740}%
\special{pa 1863 1760}%
\special{pa 1849 1780}%
\special{pa 1916 1760}%
\special{fp}%
% VECTOR 2 0 3 0 Black White  
% 2 1796 1680 1916 1680
% 
\special{pn 8}%
\special{pa 1796 1680}%
\special{pa 1916 1680}%
\special{fp}%
\special{sh 1}%
\special{pa 1916 1680}%
\special{pa 1849 1660}%
\special{pa 1863 1680}%
\special{pa 1849 1700}%
\special{pa 1916 1680}%
\special{fp}%
% VECTOR 2 0 3 0 Black White  
% 2 1796 1840 1916 1840
% 
\special{pn 8}%
\special{pa 1796 1840}%
\special{pa 1916 1840}%
\special{fp}%
\special{sh 1}%
\special{pa 1916 1840}%
\special{pa 1849 1820}%
\special{pa 1863 1840}%
\special{pa 1849 1860}%
\special{pa 1916 1840}%
\special{fp}%
% VECTOR 2 0 3 0 Black White  
% 2 1796 1920 1916 1920
% 
\special{pn 8}%
\special{pa 1796 1920}%
\special{pa 1916 1920}%
\special{fp}%
\special{sh 1}%
\special{pa 1916 1920}%
\special{pa 1849 1900}%
\special{pa 1863 1920}%
\special{pa 1849 1940}%
\special{pa 1916 1920}%
\special{fp}%
% VECTOR 2 0 3 0 Black White  
% 2 1796 2080 1916 2080
% 
\special{pn 8}%
\special{pa 1796 2080}%
\special{pa 1916 2080}%
\special{fp}%
\special{sh 1}%
\special{pa 1916 2080}%
\special{pa 1849 2060}%
\special{pa 1863 2080}%
\special{pa 1849 2100}%
\special{pa 1916 2080}%
\special{fp}%
% VECTOR 2 0 3 0 Black White  
% 2 1796 2000 1916 2000
% 
\special{pn 8}%
\special{pa 1796 2000}%
\special{pa 1916 2000}%
\special{fp}%
\special{sh 1}%
\special{pa 1916 2000}%
\special{pa 1849 1980}%
\special{pa 1863 2000}%
\special{pa 1849 2020}%
\special{pa 1916 2000}%
\special{fp}%
% VECTOR 2 0 3 0 Black White  
% 2 1956 1840 2076 1840
% 
\special{pn 8}%
\special{pa 1956 1840}%
\special{pa 2076 1840}%
\special{fp}%
\special{sh 1}%
\special{pa 2076 1840}%
\special{pa 2009 1820}%
\special{pa 2023 1840}%
\special{pa 2009 1860}%
\special{pa 2076 1840}%
\special{fp}%
% VECTOR 2 0 3 0 Black White  
% 2 1956 1920 2076 1920
% 
\special{pn 8}%
\special{pa 1956 1920}%
\special{pa 2076 1920}%
\special{fp}%
\special{sh 1}%
\special{pa 2076 1920}%
\special{pa 2009 1900}%
\special{pa 2023 1920}%
\special{pa 2009 1940}%
\special{pa 2076 1920}%
\special{fp}%
% VECTOR 2 0 3 0 Black White  
% 2 1956 2080 2076 2080
% 
\special{pn 8}%
\special{pa 1956 2080}%
\special{pa 2076 2080}%
\special{fp}%
\special{sh 1}%
\special{pa 2076 2080}%
\special{pa 2009 2060}%
\special{pa 2023 2080}%
\special{pa 2009 2100}%
\special{pa 2076 2080}%
\special{fp}%
% VECTOR 2 0 3 0 Black White  
% 2 1956 2000 2076 2000
% 
\special{pn 8}%
\special{pa 1956 2000}%
\special{pa 2076 2000}%
\special{fp}%
\special{sh 1}%
\special{pa 2076 2000}%
\special{pa 2009 1980}%
\special{pa 2023 2000}%
\special{pa 2009 2020}%
\special{pa 2076 2000}%
\special{fp}%
% VECTOR 2 0 3 0 Black White  
% 2 2116 1840 2236 1840
% 
\special{pn 8}%
\special{pa 2116 1840}%
\special{pa 2236 1840}%
\special{fp}%
\special{sh 1}%
\special{pa 2236 1840}%
\special{pa 2169 1820}%
\special{pa 2183 1840}%
\special{pa 2169 1860}%
\special{pa 2236 1840}%
\special{fp}%
% VECTOR 2 0 3 0 Black White  
% 2 2116 1920 2236 1920
% 
\special{pn 8}%
\special{pa 2116 1920}%
\special{pa 2236 1920}%
\special{fp}%
\special{sh 1}%
\special{pa 2236 1920}%
\special{pa 2169 1900}%
\special{pa 2183 1920}%
\special{pa 2169 1940}%
\special{pa 2236 1920}%
\special{fp}%
% VECTOR 2 0 3 0 Black White  
% 2 2116 2080 2236 2080
% 
\special{pn 8}%
\special{pa 2116 2080}%
\special{pa 2236 2080}%
\special{fp}%
\special{sh 1}%
\special{pa 2236 2080}%
\special{pa 2169 2060}%
\special{pa 2183 2080}%
\special{pa 2169 2100}%
\special{pa 2236 2080}%
\special{fp}%
% VECTOR 2 0 3 0 Black White  
% 2 2116 2000 2236 2000
% 
\special{pn 8}%
\special{pa 2116 2000}%
\special{pa 2236 2000}%
\special{fp}%
\special{sh 1}%
\special{pa 2236 2000}%
\special{pa 2169 1980}%
\special{pa 2183 2000}%
\special{pa 2169 2020}%
\special{pa 2236 2000}%
\special{fp}%
% VECTOR 2 0 3 0 Black White  
% 2 2276 1840 2396 1840
% 
\special{pn 8}%
\special{pa 2276 1840}%
\special{pa 2396 1840}%
\special{fp}%
\special{sh 1}%
\special{pa 2396 1840}%
\special{pa 2329 1820}%
\special{pa 2343 1840}%
\special{pa 2329 1860}%
\special{pa 2396 1840}%
\special{fp}%
% VECTOR 2 0 3 0 Black White  
% 2 2276 1920 2396 1920
% 
\special{pn 8}%
\special{pa 2276 1920}%
\special{pa 2396 1920}%
\special{fp}%
\special{sh 1}%
\special{pa 2396 1920}%
\special{pa 2329 1900}%
\special{pa 2343 1920}%
\special{pa 2329 1940}%
\special{pa 2396 1920}%
\special{fp}%
% VECTOR 2 0 3 0 Black White  
% 2 2276 2080 2396 2080
% 
\special{pn 8}%
\special{pa 2276 2080}%
\special{pa 2396 2080}%
\special{fp}%
\special{sh 1}%
\special{pa 2396 2080}%
\special{pa 2329 2060}%
\special{pa 2343 2080}%
\special{pa 2329 2100}%
\special{pa 2396 2080}%
\special{fp}%
% VECTOR 2 0 3 0 Black White  
% 2 2276 2000 2396 2000
% 
\special{pn 8}%
\special{pa 2276 2000}%
\special{pa 2396 2000}%
\special{fp}%
\special{sh 1}%
\special{pa 2396 2000}%
\special{pa 2329 1980}%
\special{pa 2343 2000}%
\special{pa 2329 2020}%
\special{pa 2396 2000}%
\special{fp}%
% STR 2 0 3 0 Black White  
% 4 2006 2392 2006 2412 2 0 0 0
% $\frac{\partial}{\partial x}$
\put(20.0600,-24.1200){\makebox(0,0)[lb]{$\frac{\partial}{\partial x}$}}%
% STR 2 0 3 0 Black White  
% 4 1018 766 1018 786 2 0 0 0
% $v$
\put(10.1800,-7.8600){\makebox(0,0)[lb]{$v$}}%
\end{picture}}%

%% file: xi_G.bbl
\begin{thebibliography}{00}

\bibitem{ABP}
J. E. Andersen, A. Bene and R. C. Penner,
Groupoid extensions of mapping class representations for bordered surfaces,
{\it Topology Appl.} {\bf 156}, 2713--2725 (2009)

\bibitem{B1}
A. Bene,
A chord diagrammatic presentation of the mapping class group of a once bordered surface,
{\it Geom.\ Dedicata} {\bf 144},
171--190 (2010)

\bibitem{B2}
A. Bene,
Mapping class factorization via fatgraph Nielsen reduction,
{\it Osaka J.\ Math} {\bf 48},
1047--1061 (2011)

\bibitem{BKP}
A. Bene, N. Kawazumi and R. C. Penner,
Canonical extensions of the Johnson homomorphisms to the Torelli groupoid,
{\it Adv.\ Math.} {\bf 221}, 627--659 (2009)

\bibitem{BE}
B. H. Bowditch and D. B. A. Epstein,
Natural triangulations associated to a surface,
{\it Topology} {\bf 27}, 91--117 (1988)

\bibitem{Earle}
C. J. Earle,
Families of Riemann surfaces and Jacobi varieties,
{\it Ann.\ Math.} {\bf 107}, 255--286 (1978)

%\bibitem{Godin}
%V. Godin,
%The unstable integral homology of the mapping class groups of a surface with boundary,
%{\it Math.\ Ann.} {\bf 337}, 15--60 (2007)

\bibitem{Har1}
J. Harer,
Stability of the homology of the mapping class group of an orientable surface,
{\it Ann.\ Math.} {\bf 121}, 215--249 (1985)

\bibitem{Har2}
J. Harer,
The virtual cohomological dimension of the mapping class group of an orientable surface,
{\it Invent.\ Math.} {\bf 84}, 157--176 (1986)

\bibitem{HZ}
J. Harer and D. Zagier,
The Euler characteristic of the moduli space of curves,
{\it Invent.\ Math.} {\bf 85}, 457--485 (1986)

\bibitem{JohSpin}
D. Johnson,
Spin structures and quadratic forms on surfaces,
{\it J.\ London Math.\ Soc.} {\bf 22}, 365--373 (1980)

\bibitem{Joh80} D. Johnson,
An abelian quotient of the mapping class group $\mathcal{I}_g$,
{\it Math.\ Ann.} 249 (1980), 225--242.

\bibitem{Joh83} D. Johnson,
A survey of the Torelli group,
In: {\it Low Dimensional Topology},
Contemp. Math. 20, AMS Providence 1983, 165--179.

\bibitem{KawSurv}
N. Kawazumi,
Canonical 2-forms on the moduli of Riemann surfaces,
In: {\it Handbook of Teichm\"uller theory, Vol.\ II}, 217--237,
A. Papadopoulos (ed), EMS Publishing House, Zurich (2009)

\bibitem{KawMor}
N. Kawazumi and S. Morita,
The primary approximation to the cohomology of the moduli space of curves and cocycles for the stable characteristic classes,
{\it Math.\ Res.\ Lett.} {\bf 3}, 629--641 (1996)

\bibitem{Kon}
M. Kontsevich,
Intersection theory on the moduli space of curves and the matrix Airy function,
{\it Comm.\ Math.\ Phys.} {\bf 147}, 1--23 (1992)

\bibitem{KPT}
Y.\ Kuno, R.\ C.\ Penner, and V.\ Turaev,
Marked fatgraph complexes and surface automorphisms,
Geom. Dedicata \textbf{167}, 151--166 (2013)

\bibitem{Mas}
G. Massuyeau,
Canonical extensions of Morita homomorphisms to the Ptolemy groupoid,
{\it Geom.\ Dedicate} {\bf 158}, 365--395 (2012)

\bibitem{MorFam1}
S. Morita,
Families of Jacobian manifolds and characteristic classes of surface bundles I,
{\it Ann.\ Inst.\ Fourier} {\bf 39}, 777--810 (1989)

\bibitem{MorExt}
S. Morita,
The extension of Johnson's homomorphism from the Torelli group to the mapping class group,
{\it Invent.\ math.} {\bf 111}, 197--224 (1993)

\bibitem{Mor93} S. Morita,
Abelian quotients of subgroups of the mapping class group of surfaces,
{\it Duke Math.\ J.} 70 (1993), 699--726.

\bibitem{MP}
S.\ Morita and R.\ C.\ Penner,
Torelli groups, extended Johnson homomorphisms, and new
cycles on the moduli space of curves,
Math.\ Proc.\ Camb.\ Phil.\ Soc. \textbf{144}, 651--671 (2008)

\bibitem{P87}
R. C. Penner,
The decorated Teichm\"uller space of punctured surfaces,
{\it Comm.\ Math.\ Phys.} {\bf 113}, 299--339 (1987)

\bibitem{P88}
R. C. Penner,
Perturbative series and the moduli space of Riemann surfaces,
{\it J.\ Differential Geom.} {\bf 27}, 35--53 (1988)

\bibitem{P04}
R. C. Penner,
Decorated Teichm\"uller space of bordered surface,
{\it Comm.\ Anal.\ Geom.} {\bf 12}, 793--820 (2004)

\bibitem{P-book}
R. C. Penner,
{\it Decorated Teichm\"uller Theory}, European Mathematical Society, 2012

\bibitem{PenPriv}
R. C. Penner,
private communication 2015.

\bibitem{PZ15}
R. C. Penner and A. M. Zeitlin,
Decorated super-Teichm\"uller space,
preprint, arXiv:1509.06302v2 (2015)

\bibitem{Stre}
K. Strebel,
{\it Quadratic differentials}, Ergeb.\ Math.\ Grenzgeb. {\bf 5},
Springer-Verlag, Heidelberg, 1984 

\end{thebibliography}
